\documentclass[11pt]{amsart}

\usepackage{amsmath,graphicx,amssymb,amsthm,amstext}
\usepackage{caption}
\usepackage{subcaption}
\usepackage{tikz}
\usepackage[ruled]{algorithm2e}
\usepackage{listings}
\usepackage{hyperref}

\usetikzlibrary{automata,arrows,calc,positioning}

\newtheorem{theorem}{Theorem}[section]
\newtheorem{lemma}[theorem]{Lemma}

\newtheorem{proposition}[theorem]{Proposition}

\theoremstyle{definition}
\newtheorem*{definition}{Definition}

\theoremstyle{remark}

\title{Distance domatic numbers for grid graphs}
\author{Alex Cameron and Jiasheng Yan}
\address[Alex Cameron]{Vanderbilt University}
\email{alexander.cameron@vanderbilt.edu}
\address[Jiasheng Yan]{Vanderbilt University}
\email{jiasheng.yan@vanderbilt.edu}
\date{}

\begin{document}
\maketitle

\begin{abstract}
We say that a vertex-coloring of a graph is a proper $k$-distance domatic coloring if for each color, every vertex is within distance $k$ from a vertex receiving that color. The maximum number of colors for which such a coloring exists is called the $k$-distance domatic number of the graph. The problem of determining the $k$-distance domatic number is motivated by questions about multi-agent networks including arrangements of sensors and robotics. Here, we find the exact $k$-distance domatic numbers for all grid graphs formed from the Cartesian product of two sufficiently long paths.
\end{abstract}

Given a graph $G$ and a subset of its vertices $S \subseteq V(G)$, we say that $S$ \textit{dominates} $G$ if for each vertex $x \in V(G)$, either $x \in S$ or there exists some vertex $y \in S$ for which $xy \in E(G)$. We call a partition of the vertices, \[V(G) = S_1 \cup S_2 \cup \cdots \cup S_t,\] into $t$ disjoint parts a \textit{domatic partition} if each $S_i$ is a dominating set of $G$. The maximum such $t$ is known as the \textit{domatic number} of a graph, and was introduced by Cockayne and Hedetniemi in 1977 \cite{1977} and has been studied extensively since then. Of particular relevance here, is paper by Chang \cite{chang94} in which he, among other things, calculated the domatic numbers for all two-dimensional grid graphs, the Cartesian product of two paths.

In 1983, Zelinka \cite{zelinka83} extended the definition to a \textit{$k$-domatic number}, the maximum number of parts in a partition of the vertices such that each part is ``$k$-dominating" in the sense that every vertex of the graph is within distance $k$ from some vertex of the set. The term ``$k$-domatic" has come to refer to more than one thing in the literature so on a formal level we would use \textit{$k$-distance domatic number} of a graph as we do in the definition section. However, in this paper, we only ever mean this definition so we often use the shorter ``$k$-domatic" in the proofs. Recently, in 2015, Kiser \cite{masters} determined the $2$-distance domatic numbers for all two-dimensional grid graphs including the infinite grid, the Cartesian product of two paths that are infinite in either direction. In this paper, we determine the $k$-distance domatic numbers for ``almost" all two-dimensional grid graphs as well as for the infinite case.

This problem is motivated more recently by the study of information sharing in multi-agent networks. In fact, we originally came up with the definition of a distance domatic number while talking about spy networks. Imagine that the vertices of a graph represent the agents of a covert network of spies, and let an edge between two vertices denote that two agents are in contact. Each agent holds one piece of information, and has access to the other information in the network through their contacts. However, we assume that information transmitted through the network becomes less reliable the more agents it needed to travel through so for practical purposes we assume that any particular agent can only access information held by an agent within a certain limited distance on the graph. The distance domatic number is the maximum number of pieces of information that can be stored in this spy network if we wish all agents to have reliable access to all of the information, a kind of measurement of the carrying capacity of information for a given network.

Similarly, Abbas, Egerstedt, Liu, Thomas, and Whalen \cite{robots} explained the same concept but with the motivation of studying networks of robots which each make decisions based on information communicated by other nearby robots. This work has spun off into a line of papers about \textit{coupon numbers} of graphs \cite{coupon, shi17}, which are the same as domatic numbers with the added twist that a vertex can never dominate itself.

\section{Definitions and main result}

\begin{definition}
Given a graph $G$ and two vertices $x,y \in V(G)$, let $d(x,y)$ denote the \textit{distance} between $x$ and $y$ - the length of the shortest path with $x$ and $y$ as endpoints. If $x=y$, then $d(x,x)=0$.
\end{definition}

\begin{definition}
For a fixed positive integer $k$, a graph $G$, and a vertex $x \in V(G)$, let \[f_k(x) = |\{y \in V(G) : d(x,y) \leq k\}|.\] Let \[\delta_k(G) = \min_{x \in V(G)}{f_k(x)},\] the \textit{minimum $k$-distance degree} of $G$.
\end{definition}

\begin{definition}
Given a graph $G$ and a subset of its vertices $S \subseteq V(G)$, we say that $S$ \textit{$k$-distance dominates} $G$ if for each vertex $x \in V(G)$, there exists some vertex $y \in S$ such that $d(x,y) \leq k$.
\end{definition}

\begin{definition}
Given a graph $G$, we call a partition of the vertices, \[V(G) = S_1 \cup S_2 \cup \cdots \cup S_t,\] into $t$ disjoint parts a \textit{$k$-distance domatic partition} if each $S_i$ is a $k$-distance dominating set of $G$. Such a partition is equivalent to an assignment of $t$ colors to the vertices, \[c:V(G) \rightarrow \{1,2,\ldots,t\},\] such that for each vertex $x \in V(G)$ and for each color $i \in \{1,\ldots,t\}$, there exists a vertex $y \in V(G)$ such that $d(x,y) \leq k$ and $c(y) = i$. We will call such a coloring a \textit{proper $k$-distance domatic coloring}. For a given positive integer $k$, let $d_k(G)$ denote the maximum $t$ for which such an assignment exists, the \textit{$k$-distance domatic number} of $G$.
\end{definition}

\begin{definition}
Note that for any graph $G$, \[d_k(G) \leq \delta_k(G).\] Graphs that achieve this upper bound, that is any graph $G$ for which \[d_k(G) = \delta_k(G),\] are called \textit{$k$-distance domatically-full}.
\end{definition}

\begin{definition}
Given positive integers $s$ and $t$, let the $s \times t$ \textit{grid graph} $G_{s,t}$ be the graph with vertex set \[V(G_{s,t}) = \{(i,j) \in \mathbb{Z}^2: 1 \leq i \leq s, 1 \leq j \leq t\}\] and edge set \[E(G_{s,t}) =\{(a,b)(c,d) : |a-c|+|b-d| =1\}.\]
\end{definition}

In all that follows, we will generally assume that $r \leq l$ when talking about the grid graph $G_{r,l}$ since $G_{r,l}$ is isomorphic to $G_{l,r}$.

\begin{definition}
Let $G_{\infty,\infty}$ denote the infinite graph with vertex set $V=\mathbb{Z}^2$ and edge set \[E=\{(a,b)(c,d) : |a-c|+|b-d| = 1\}.\]
\end{definition}

\begin{theorem}
\label{mainthm2}
Let $k,r,l$ be some positive integers such that $r \leq l$ and either $1 \leq r \leq k+2$ or $k+1+\left\lceil\frac{k+1}{2}\right\rceil \leq r$, then all but finitely many of the grid graphs $G_{r,l}$ are $k$-distance domatically-full.
\end{theorem}

This is our main result, and it is a straight-forward consequence of Lemmas~\ref{up},~\ref{mainthm}, and~\ref{small}. In Section~\ref{upsection}, we calculate the minimum $k$-distance degree of every grid graph to establish the upper bound on the $k$-distance domatic numbers. In Section~\ref{blocksection}, we define explicit vertex colorings of each grid graph of the form $G_{r,2k-r+3}$ for $r \leq k+1$ and demonstrate that these colorings are proper $k$-distance domatic. In Section~\ref{extsection}, we extend this coloring for $r=k+1$ to cases for larger $r$ and $l$. In Section~\ref{smallsection}, we extend the colorings given in Section~\ref{blocksection} for larger values of $l$ whenever $r \leq k$. In Section~\ref{k=3}, we calculate the $3$-distance domatic numbers for all grids. In Section~\ref{infsection}, we calculate the $k$-distance domatic numbers for the infinite grid graph.

\section{Upper bound}
\label{upsection}

\begin{lemma}
\label{up}
For any integers $l \geq r \geq 1$, \[d_k(G_{r,l}) \leq \sum_{i=0}^{r-1} \max{\left\{0,\min{\left\{l,k+1-i\right\}}\right\}}.\] In particular, \[d_k(G_{r,l}) \leq r(k+1) - \frac{r(r-1)}{2}\] if $r \leq k+1 \leq l$, and \[d_k(G_{r,l}) \leq \frac{(k+1)(k+2)}{2}\] when $k+1 \leq r \leq l$.
\end{lemma}

\begin{proof}
The $k$-domatic number of any graph $G$ is at most its minimum $k$-degree, $d_k(G) \leq \delta_k(G)$. In particular, this minimum $k$-degree is achieved by the corner vertices of $G_{r,l}$. The corner vertex $(1,1)$ is within distance $k$ of $\min{\left\{l,k+1-(i-1)\right\}}$ other vertices in the $i$th row for any $i \leq k+1$ and $0$ for any $i \geq k+2$. Therefore, it is within distance $k$ of \[\sum_{i=0}^{\min{\{r-1,k\}}}\min{\left\{l,k+1-i\right\}}=\sum_{i=0}^{r-1} \max{\left\{0,\min{\left\{l,k+1-i\right\}}\right\}}\] vertices total.
\end{proof}

\section{The standard block coloring of $G_{r,2k-r+3}$ for $r \leq k+1$}
\label{blocksection}

For convenience, we assume an ordering on the vertices of any grid graph $G_{r,l}$ given by $(x_1,y_1) \leq (x_2,y_2)$ if and only if $x_1 < x_2$ or $x_1=x_2$ and $y_1 \leq y_2$.

Let $1 \leq r \leq k+1$, and let \[N_r = \left(k+1-\frac{r-1}{2}\right)r.\] Let $[N_r] = \{1,\ldots,N_r\}$. Note that $2N_r = r(2k-r+3)$. We now define a coloring of the vertices of $G_{r,2k-r+3}$ which can be thought of as simply coloring the first $N_r$ vertices colors $1,\ldots,N_r$ in order, then coloring the last $N_r$ vertices with the colors $1,\ldots,N_r$ in reverse order. This coloring is illustrated in Figure~\ref{block}.

\begin{definition}
The vertex-coloring $\varphi_r:V(G_{r,2k-r+3}) \rightarrow [N_r]$ given by \[\varphi_r((j,i)) = \left\{
        \begin{array}{ll}
           (i-1)r+j & \quad 1 \leq  (i-1)(k+1)+j \leq N_r \\
           2N_r+1 - (i-1)r-j & \quad N_r+1 \leq  (i-1)r+j \leq 2N_r
        \end{array}
    \right. \] is the \textbf{standard block coloring} of $G_{r,2k-r+3}$. When $r=k+1$, then we drop the $r$ subscript and denote the coloring of $G_{k+1,k+2}$ with $\varphi$.
\end{definition}

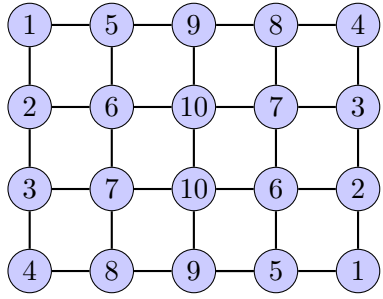
\begin{figure}
\tikzset{state/.style={circle,fill=blue!20,draw,minimum size=1.5em,inner sep=0pt},
            }
\begin{tikzpicture}

    \node[state] (1) {1};
    \node[state] (2) [below=0.5cm of 1] {2};
    \node[state] (3) [below =0.5cm of 2] {3};
    \node[state] (4) [below =0.5cm of 3] {4};
    \node[state] (5) [right =0.5cm of 1] {5};
    \node[state] (6) [below=0.5cm of 5 ] {6};
    \node[state] (7) [below=0.5cm of 6] {7};
    \node[state] (8) [below =0.5cm of 7] {8};
    \node[state] (9) [right =0.5cm of 5] {9};
    \node[state] (10) [below =0.5cm of 9] {10};
    \node[state] (11) [below =0.5cm of 10] {10};
    \node[state] (12) [below =0.5cm of 11] {9};
    \node[state] (13) [right =0.5cm of 9] {8};
    \node[state] (14) [below =0.5cm of 13] {7};
    \node[state] (15) [below =0.5cm of 14] {6};
    \node[state] (16) [below =0.5cm of 15] {5};
    \node[state] (17) [right =0.5cm of 13] {4};
    \node[state] (18) [below =0.5cm of 17] {3};
    \node[state] (19) [below =0.5cm of 18] {2};
    \node[state] (20) [below =0.5cm of 19] {1};
    
        \path[draw,thick]
    (1) edge node {} (2)
    (2) edge node {} (6)
    (2) edge node {} (3)
    (3) edge node {} (4)
    (3) edge node {} (7)
    (4) edge node {} (8)
    (1) edge node {} (5)
    (5) edge node {} (6)
    (5) edge node {} (9)
    (6) edge node {} (10)
    (6) edge node {} (7)
    (7) edge node {} (11)
    (7) edge node {} (8)
    (8) edge node {} (12)
    (10) edge node {} (11)
    (9) edge node {} (10)
    (11) edge node {} (12)
    (12) edge node {} (16)
    (13) edge node {} (9)
    (13) edge node {} (14)
    (14) edge node {} (15)
    (15) edge node {} (16)
    (10) edge node {} (14)
    (11) edge node {} (15)
    (12) edge node {} (16)
    (17) edge node {} (13)
    (18) edge node {} (14)
    (19) edge node {} (15)
    (20) edge node {} (16)
    (17) edge node {} (18)
    (18) edge node {} (19)
    (19) edge node {} (20);
    
    \end{tikzpicture}
    \caption{The standard block coloring $\varphi$ of $G_{4,5}$ for $k=3$.}
    \label{block}
    \end{figure}

\begin{lemma}
\label{blockproof}
The standard block coloring of $G_{r,2k-r+3}$ is a proper $k$-domatic coloring.
\end{lemma}

\begin{proof}
Let $X$ be some color from $[N_r]$, and let $(j,i)$ be the vertex of $G_{r,2k-r+3}$ for which $X=(i-1)r+j$. Then we claim that \[\varphi_r((j,i))=\varphi_r((r+1-j,2k-r+4-i))=X.\] First, note that
\begin{align*}
(i-1)r+j &\leq N_r\\
N_r+1+(i-1)r+j &\leq 2N_r+1\\
N_r+1 &\leq 2r\left(k+1-\frac{r-1}{2}\right) +1-ir+r-j\\
N_r+1 &\leq (2k-r+4-i-1)r+r+1-j
\end{align*}
and
\begin{align*}
1 &\leq (i-1)r+j\\
2N_r+1 -(i-1)r-j&\leq 2N_r\\
2r\left(k+1-\frac{r-1}{2}\right)+1 -(i-1)r-j &\leq 2N_r\\
(2k-r+4-i-1)r+r+1-j &\leq 2N_r.
\end{align*}
So it follows that
\begin{align*}
\varphi_r((r+1-j,2k-r+4-i)) &=2N_r+1 - (2k-r+4-i-1)r-(r+1-j)\\
&= 2N_r+1-2N_r+ir-r-1+j\\
&=(i-1)r+j\\
&= X,
\end{align*}
as desired.

We now claim that every vertex of $G_{r,2k-r+3}$ is within distance $k$ from either $(j,i)$ or $(r+1-j,2k-r+4-i)$ or both. Since $X$ was an arbitrary color, then this is enough to show that every vertex is within distance $k$ from at least one vertex from each color class.

Pick an arbitrary vertex $(x, y)$ from $G_{r,2k-r+3}$. Note that the distance between $(x,y)$ and another vertex $(a,b)$ can be given by the formula, \[d((r,c),(x,y)) = |a-x|+|b-y|.\] Suppose, towards a contradiction that both $d((x,y),(j,i)) \geq k+1$ and $d((x,y),(r+1-j,2k-r+4-i)) \geq k+1$. Then it follows that \[| j-x | + | i-y |+  | r+1-j-x|+ | 2k-r+4-i-y| \geq 2k+2.\]

Now, we know by our choice of $(j,i)$ that  $i \leq 2k-r+4-i$. There are three possible cases of where $y$ could be located with respect to $i$ and $2k-r+4-i$. First, if $y \leq i$, then \[|i-y|+|2k-r+4-i-y|=2k-r+4-2y \leq 2k-r+2\] since $y \geq 1$. Similarly, if  $i \leq y \leq 2k-r+4-i$, then \[|i-y|+|2k-r+4-i-y| =2k-r+4-2i \leq 2k-r+2\] since $i \geq 1$, and if $i \leq 2k-r+4-i \leq y$, then \[|i-y|+|2k-r+4-i-y|=2y-2k+r-4 \leq 2k-r+2\] since $y \leq 2k-r+3$. Therefore, in all three possible cases, the sum of the horizontal distances is at most $2k-r+2$.

In terms of the vertical distance, either $j \leq r+1-j$ or $ r+1-j \leq j$. In the first case, if $x \leq j \leq r+1-j$, then \[|j-x|+|r+1-j-x| = r+1-2x \leq r-1\] since $x \geq 1$. Similarly, if $j \leq x \leq r+1-j$, then  \[|j-x|+|r+1-j-x| = r+1-2j \leq r-1\] since $j \geq 1$, and if $j \leq r+1-j \leq x$, then  \[|j-x|+|r+1-j-x| = 2x-r-1 \leq r-1\] since $x \leq r$. Therefore, in all three possible cases, the sum of the vertical distances is at most $r-1$. Consequently, \[| j-x | + | i-y |+  | r+1-j-x|+ | 2k-r+4-i-y| \leq 2k-r+2+r-1 = 2k+1,\] which is a contradiction to our beginning assumption that \[| j-x | + | i-y |+  | r+1-j-x|+ | 2k-r+4-i-y| \geq 2k+2.\] Hence, every vertex is within distance $k$ from at least one vertex with the color $X$. So $\varphi_r$ is a proper $k$-domatic coloring.
\end{proof}

\section{Extending the coloring}
\label{extsection}

Consider the set of vertex-colorings of $G_{k+1,k+2}$ that are isomorphic to our block coloring $\varphi$. Each such coloring is given by its assignments of the first $N=\frac{(k+1)(k+2)}{2}$ vertices to $N$ distinct colors since the colors of the remaining $N$ vertices are determined by the first $N$. Therefore, this set can be thought of as the set of permutations $S_N$ for $N=\frac{(k+1)(k+2)}{2}$.

We will extend the coloring $\varphi$ of $G_{k+1,k+2}$ to a more general grid graph $G_{r,l}$, by considering larger grids as, possibly overlapping, copies of $G_{k+1,k+2}$, each colored with some permutation of $\varphi$. Therefore, each vertex will be contained inside some appropriately-colored copy of $G_{k+1,k+2}$. The key will be to ensure that the coloring of each copy of $G_{k+1,k+2}$ agrees with the coloring of any other copy of $G_{k+1,k+2}$ wherever the two copies overlap. This need is what leads to the following section of technical lemmas.

\subsection{Permutation lemmas}

Fix some positive integer $k$, and let $N = \frac{(k+1)(k+2)}{2}$. Let $[N]$ denote the set $\{1,\ldots,N\}$, and let \[S_N = \left\{f:[N] \rightarrow [N] : f \text{ is a bijection}\right\}\] be the set of permutations of the elements of $[N]$.

Given some positive integer $1 \leq s \leq \left\lfloor \frac{k+2}{2} \right\rfloor$, let $h_s:S_N \rightarrow S_N$ be defined by sending each permutation $f$ to the permutation $h_s(f)$ given by \[h_s(f)(x) = \left\{
        \begin{array}{ll}
            f(s(k+1)+1-x) & \quad 1 \leq x \leq s(k+1) \\
            f(x) & \quad s(k+1)+1 \leq x \leq N
        \end{array}
    \right. .\] We can verify that $h_s(f)$ is indeed a permutation.
    
\begin{proposition}
If $f \in S_N$, and $1 \leq s \leq \left\lfloor \frac{k+2}{2} \right\rfloor$ is a positive integer, then $h_s(f) \in S_N$.
\end{proposition}

\begin{proof}
Let $x_1,x_2 \in [N]$ such that $h_s(f)(x_1) = h_s(x_2)$. If $1 \leq x_1,x_2 \leq s(k+1)$, then
\begin{align*}
f(s(k+1)+1-x_1) &= f(s(k+1)+1-x_2)\\
s(k+1)+1-x_1 &= s(k+1)+1-x_2\\
x_1 &= x_2.
\end{align*}
If $s(k+1)+1 \leq x_1,x_2 \leq N$, then $f(x_1)=f(x_2)$. So $x_1=x_2$. If $1 \leq x_1 \leq s(k+1)$ and $s(k+1)+1 \leq x_2 \leq N$, then
\begin{align*}
f(s(k+1)+1-x_1) &= f(x_2)\\
s(k+1)+1-x_1 &= x_2\\
\end{align*}
Since $1 \leq x_1$, then this implies that $s(k+1)+1 \leq x_2 \leq s(k+1)$, a contradiction. So $h_s(f)$ is injective. Since the domain and codomain are both $[N]$, a finite set, then this implies that $h_s(f)$ is a bijection, and therefore a permutation in $S_N$.
\end{proof}
    
Next, let $R_x$ denote the remainder of an integer $x$ when divided by $k+1$, and let  \[r_x = \left\{
        \begin{array}{ll}
            R_x & \quad 1 \leq R_x \leq k\\
            k+1 & \quad R_x =0
        \end{array}
    \right. .\]Given some positive integer $1 \leq t \leq \left\lfloor \frac{k+1}{2} \right\rfloor$, let $v_t:S_N \rightarrow S_N$ be defined by \[v_t(f)(x) = \left\{
        \begin{array}{ll}
            f(x+k+1-t) & \quad 1 \leq x \leq (k+1)\left\lfloor \frac{k+2}{2} \right\rfloor \text{ and } 1 \leq r_x \leq t\\
           f(x) & \quad 1 \leq x \leq (k+1)\left\lfloor \frac{k+2}{2} \right\rfloor \text{ and } t+1 \leq r_x \leq k+1-t\\
           f(x-k-1+t) & \quad 1 \leq x \leq (k+1)\left\lfloor \frac{k+2}{2} \right\rfloor \text{ and } k+2-t \leq r_x \leq k+1\\
           f((k+1)^2+t+1-x) & \quad (k+1)\left\lfloor \frac{k+2}{2} \right\rfloor + 1  \leq x \leq N \text{ and } 1 \leq r_x \leq t\\
           f(x) & \quad (k+1)\left\lfloor \frac{k+2}{2} \right\rfloor + 1  \leq x \leq N  \text{ and } t+1 \leq r_x \leq \frac{k+1}{2}
        \end{array}
    \right. .\] Note that the final two cases only apply when $k$ is odd. We can verify that $v_t(f)$ is also a permutation of $[N]$.
    
\begin{proposition}
If $f \in S_N$, and $1 \leq t \leq \left\lfloor \frac{k+1}{2} \right\rfloor$ is a positive integer, then $v_t(f) \in S_N$.
\end{proposition}

\begin{proof}
We will show that $v_t(f):[N] \rightarrow [N]$ is surjective, and therefore, a bijection since it's domain and codomain have the same finite cardinality. Let $y \in [N]$. Then there exists a $z \in [N]$ such that $f(z)=y$ since $f$ is a bijection.

\textbf{Case 1: $\mathbf{1 \leq z \leq (k+1)\left\lfloor \frac{k+2}{2} \right\rfloor}$}. If $1 \leq r_z \leq t$, then let $z = m(k+1)+r_z$ for some $0 \leq m \leq \left\lfloor \frac{k+2}{2} \right\rfloor -1$. Let \[x = z+k+1-t = m(k+1)+r_z+k+1-t.\] Then \[k+2-t \leq r_z+k+1-t \leq k+1,\] so $k+2-t \leq r_x \leq k+1$. Moreover, \[x \geq k+2-\left\lfloor \frac{k+1}{2} \right\rfloor = \left\lceil \frac{k+1}{2} \right\rceil + 1 \geq 1,\] and \[x \leq \left(\left\lfloor \frac{k+2}{2} \right\rfloor -1\right)(k+1)+t+k+1-t = (k+1)\left\lfloor \frac{k+2}{2} \right\rfloor.\] Hence, \[v_t(f)(x) = f(x-k-1+t) = f(z)=y.\] Next, if $t+1 \leq r_z \leq k+1-t$, then we get that \[v_t(f)(z)=f(z)=y\] immediately.

Finally, if $k+2-t \leq r_z \leq k+1$, then let $z=m(k+1)+r_z$ where $0 \leq m \leq \left\lfloor \frac{k+2}{2} \right\rfloor - 1$. Now, let \[x = z-k-1+t = (m-1)(k+1) + r_z +t.\] The term $r_z+t$ adds at least $k+2$ and at most $k+1+t$. So $1 \leq r_x \leq t$. Moreover, \[x \geq (m-1)(k+1) + k+2 \geq -k-1+k+2=1,\] and \[x \leq m(k+1)+t \leq \left( \left\lfloor \frac{k+2}{2} \right\rfloor -1 \right)(k+1)+\left\lfloor \frac{k+1}{2} \right\rfloor  \leq (k+1)\left\lfloor \frac{k+2}{2} \right\rfloor.\] Therefore, \[v_t(f)(x) = f(x+k+1-t)=f(z)=y.\]

\textbf{Case 2: $\mathbf{(k+1)\left\lfloor \frac{k+2}{2} \right\rfloor+1 \leq z \leq N}$}. Since we may now assume that $k$ is odd, then we know that \[\frac{(k+1)^2}{2} +1 \leq z \leq \frac{(k+1)^2}{2} + \frac{k+1}{2}.\] If $1 \leq r_z \leq t$, then it follows that  \[\frac{(k+1)^2}{2} +1 \leq z \leq \frac{(k+1)^2}{2} +t.\] Let $x = (k+1)^2 +t+1-z$, then \[(k+1)^2+t+1 - \frac{(k+1)^2}{2} -1\geq x \geq (k+1)^2+t+1 - \frac{(k+1)^2}{2} -t.\] Therefore,  \[\frac{(k+1)^2}{2} +1 \leq x \leq \frac{(k+1)^2}{2} +t.\] And so \[v_t(f)(x) = f((k+1)^2+t+1-x)=f(z)=y.\] If instead \[\frac{(k+1)^2}{2} +t+1 \leq z \leq \frac{(k+1)^2}{2} + \frac{k+1}{2},\] then \[v_t(f)(z)=f(z)=y.\]
\end{proof}
    
We will now show that mappings $h_s$ and $v_t$ commute under function composition.
    
\begin{lemma}
\label{commute}
Let $f \in S_N$, and fix two positive integers $s$ and $t$ such that $1 \leq s \leq \left\lfloor \frac{k+2}{2} \right\rfloor$ and $1 \leq t \leq \left\lfloor \frac{k+1}{2} \right\rfloor$. Then \[v_t(h_s(f)) = h_s(v_t(f)).\]
\end{lemma}

\begin{proof}
We will prove this by showing that $h_s^{-1}(v_t(h_s(f)))(x) = v_t(f)(x)$ for any $f\in S_N$ and any $x \in [N]$. First, note that $h_s^{-1}=h_s$ since $1 \leq x \leq s(k+1)$ implies that $1 \leq s(k+1)+1-x \leq s(k+1)$ so \[h_s(h_s(f))(x) = \left\{
        \begin{array}{ll}
            h_s(f)(s(k+1)+1-x) & \quad 1 \leq x \leq s(k+1) \\
            h_s(f)(x) & \quad s(k+1)+1 \leq x \leq N
        \end{array}
    \right. \]  which in turn gives $f(x)$ in each case. Therefore, we must show that $h_s(v_t(h_s(f)))(x)=v_t(f)(x)$ for any $x \in [N]$.
    
\textbf{Case 1: $\mathbf{1 \leq x \leq s(k+1)}$}. In this case, \[h_s(v_t(h_s(f)))(x) = v_t(h_s(f))(s(k+1)+1-x)\] and \[v_t(h_s(f))(s(k+1)+1-x) =  \left\{
        \begin{array}{ll}
           h_s(f)(s(k+1)+1-x+k+1-t) & \quad 1 \leq r \leq t\\
           h_s(f)(s(k+1)+1-x) & \quad t+1 \leq r \leq k+1-t\\
           h_s(f)(s(k+1)+1-x-k-1+t) & \quad k+2-t \leq r \leq k+1
        \end{array}
    \right. \] where \[r=r_{s(k+1)+1-x}.\] Note that $1 \leq x \leq s(k+1)$ implies \[1 \leq s(k+1)+1-x \leq s(k+1) \leq (k+1)\left\lfloor \frac{k+2}{2} \right\rfloor\] so the final two cases in the definition of $v_t$ are unnecessary here. Now, in the first case we note that
\begin{align*}
(s+1)(k+1)-t &\geq (s+1)(k+1) - \left\lfloor \frac{k+1}{2} \right\rfloor\\
&= s(k+1) + \left\lceil \frac{k+1}{2} \right\rceil\\
&\geq s(k+1) \geq x.
\end{align*}
Therefore, 
\begin{align*}
x &\leq (s+1)(k+1) - t\\
1 &\leq s(k+1)+1-x+k+1-t.
\end{align*}
Moreover, we know that $s(k+1)+1-x = m(k+1)+r$ for some $0 \leq m \leq s-1$ and $1 \leq r \leq t$. Hence,
\begin{align*}
s(k+1)+1-x+k+1-t &= m(k+1)+r +k+1-t\\
&= (m+1)(k+1) + (r-t)\\
&\leq s(k+1)
\end{align*}
since $r-t \leq 0$ and $m+1 \leq s$. Therefore, \[h_s(f)(s(k+1)+1-x+k+1-t) = f(x-k-1+t)\] in the first case.

Next, the second case is straight-forward since $1 \leq s(k+1)+1-x \leq s(k+1)$. There, we get that \[h_s(f)(s(k+1)+1-x) = f(x).\]

In the last case, we let $s(k+1)+1-x = m(k+1)+r$ for some $0 \leq m \leq s-1$ and some $k+2-t \leq R \leq k+1$. Then we get that
\begin{align*}
s(k+1)+1-x-k-1+t &= (m-1)(k+1)+r+t\\
&\geq -k-1+k+2-t+t=1,\\
\end{align*}
and
\begin{align*}
s(k+1)+1-x-k-1+t &= (m-1)(k+1)+r+t\\
&\leq (s-2)(k+1)+k+1+t\\
&\leq (s-1)(k+1) + \left\lfloor \frac{k+1}{2} \right\rfloor\\
&\leq s(k+1).
\end{align*}
Therefore, in these final cases we get that \[h_s(f)(s(k+1)+1-x-k-1+t) = f(x+k+1-t).\] Hence, all together we get that \[h_s(v_t(h_s(f))(x)) =  \left\{
        \begin{array}{ll}
          f(x-k-1+t) & \quad 1 \leq r \leq t\\
          f(x) & \quad t+1 \leq r \leq k+1-t\\
          f(x+k+1-t) & \quad k+2-t \leq r \leq k+1
        \end{array}
    \right. .\] Here, the $r$ is still standing for $r_{s(k+1)+1-x}$. Let $s(k+1)+1-x = m(k+1) + r$ for some $0 \leq m \leq s-1$. Then \[x = (s-m)(k+1) + 1-r = (s-m-1)(k+1) +(k+2-r).\] If $1 \leq r \leq t$, then \[k+2-t \leq k+2-r \leq k+1.\] Hence, $k+2-t \leq r_x \leq k+1$. If $t+1 \leq r \leq k+1-t$, then \[t+1 \leq k+2-r \leq k+1-t\] also and so $t+1 \leq r_x \leq k+1-t$. If $k+2-t \leq r \leq k+1$, then $1 \leq k+2-r \leq t$ and so $1 \leq r_x \leq t$. Thus, \[h_s(v_t(h_s(f))(x)) =  \left\{
        \begin{array}{ll}
          f(x+k+1-t) & \quad 1 \leq r_x \leq t\\
          f(x) & \quad t+1 \leq r_x \leq k+1-t\\
          f(x-k-1+t) & \quad k+2-t \leq r_x \leq k+1
        \end{array}
    \right. .\] This matches $v_t(f)(x)$ for any $1 \leq x \leq s(k+1)$.
    
\textbf{Case 2: $\mathbf{s(k+1)+1 \leq x \leq (k+1) \left\lfloor \frac{k+1}{2} \right\rfloor}$}. In this case $h_s(v_t(h_s(f)))(x) = v_t(h_s(f))(x)$, and \[v_t(h_s(f))(x) = \left\{
        \begin{array}{ll}
           h_s(f)(x+k+1-t) & \quad 1 \leq r_x \leq t\\
           h_s(f)(x) & \quad  t+1 \leq r_x \leq k+1-t\\
           h_s(f)(x-k-1+t) & \quad  k+2-t \leq r_x \leq k+1
        \end{array}
    \right. \]
    
If $1 \leq r_x \leq t$, then let $x=m(k+1)+r_x$ for some $s \leq m \leq \left\lfloor \frac{k+2}{2} \right\rfloor -1$ and we get that
\begin{align*}
x+k+1-t &= m(k+1)+k+1+r_x-t\\
&\geq s(k+1)+\left\lceil \frac{k+1}{2} \right\rceil +1\\
&\geq s(k+1)+1,
\end{align*}
and
\begin{align*}
x+k+1-t &= (m+1)(k+1)+r_x-t\\
&\leq \left( \left\lfloor \frac{k+2}{2} \right\rfloor \right)(k+1).
\end{align*}
Therefore, $h_s(f)(x+k+1-t) = f(x+k+1-t)$ in the first case.

In the second case, we get that $h_s(f)(x) = f(x)$ just by $s(k+1)+1 \leq x \leq (k+1)\left\lfloor \frac{k+2}{2} \right\rfloor$.

In the third case, let $x = m(k+1)+r_x$ for some $s \leq m \leq \left\lfloor \frac{k+2}{2} \right\rfloor -1$ and some $k+2-t \leq r_x \leq k+1$. In this case we get that
\begin{align*}
x-k-1+t &= m(k+1)-k-1)+r_x+t\\
&\geq s(k+1)+-k-1+k+2\\
&\geq s(k+1)+1,
\end{align*}
and
\begin{align*}
x-k-1+t &= (m-1)(k+1)+r_x+t\\
&\leq \left( \left\lfloor \frac{k+2}{2} \right\rfloor -2\right)(k+1) + k+1+t\\
&\leq \left( \left\lfloor \frac{k+2}{2} \right\rfloor \right)(k+1) - (k+1)+ \left\lfloor \frac{k+1}{2} \right\rfloor\\
&=\left( \left\lfloor \frac{k+2}{2} \right\rfloor \right)(k+1) - \left\lceil \frac{k+1}{2} \right\rceil\\
&\leq \left( \left\lfloor \frac{k+2}{2} \right\rfloor \right)(k+1).
\end{align*}
Therefore, $h_s(f)(x-k-1+t) = f(x-k-1+t)$ in the last case. Hence,  \[h_s(v_t(h_s(f)))(x) = \left\{
        \begin{array}{ll}
           f(x+k+1-t) & \quad 1 \leq r_x \leq t\\
           f(x) & \quad  t+1 \leq r_x \leq k+1-t\\
           f(x-k-1+t) & \quad  k+2-t \leq r_x \leq k+1
        \end{array}
    \right. \] whenever $s(k+1)+1 \leq x \leq (k+1)\left\lfloor \frac{k+2}{2} \right\rfloor$, which is equal to $v_t(f)(x)$ for the same values of $x$.
    
\textbf{Case 3: $\mathbf{k}$ is odd and $\mathbf{\frac{(k+1)^2}{2} + 1  \leq x \leq N}$}. In this case, $h_s(v_t(h_s(f)))(x) = v_t(h_s(f))(x)$, and\[v_t(h_s(f))(x) = \left\{
        \begin{array}{ll}
           h_s(f)((k+1)^2+t+1-x) & \quad 1 \leq r_x \leq t\\
           h_s(f)(x) & \quad t+1 \leq r_x \leq \frac{k+1}{2}
        \end{array}
    \right. \]  if $k$ is odd and $\frac{(k+1)^2}{2} + 1  \leq x \leq N$.

If \[\frac{(k+1)^2}{2} + 1  \leq x \leq \frac{(k+1)^2}{2} + t,\] then \[(k+1)^2+t+1-\frac{(k+1)^2}{2} - t \leq (k+1)^2+t+1-x \leq (k+1)^2+t+1-\frac{(k+1)^2}{2} - 1.\] So  \[\frac{(k+1)^2}{2} + 1  \leq (k+1)^2+t+1-x \leq \frac{(k+1)^2}{2} + t.\] Therefore, $h_s(f)((k+1)^2+t+1-x) = f((k+1)^2+t+1-x)$. And if \[\frac{(k+1)^2}{2} + t+1  \leq x \leq \frac{(k+1)^2}{2} + \frac{k+1}{2},\] then $h_s(f)(x)=f(x)$. Hence, $h_s(v_t(h_s(f)))(x) = v_t(f)(x)$ whenever $\frac{(k+1)^2}{2} + 1  \leq x \leq N$ as well. Thus, $h_s(v_t(h_s(f))) = v_t(f)$ for all $f \in S_N$.
\end{proof}

\subsection{Main lemmas}

We will now demonstrate how to color more general grid graphs by using the standard block coloring $\varphi$ on the first $G_{k+1,k+2}$ and then ``extending" it with suitable isomorphic colorings to the remain rows and columns.

\begin{lemma}
\label{c1}
For any positive integer $k$,
\[d_k\left(G_{k+1,k+2+J}\right) \geq \frac{(k+1)(k+2)}{2},\] for any $\left\lceil \frac{k+2}{2} \right\rceil \leq J < k+2$.
\end{lemma}

\begin{proof}
Consider this graph as two copies of $G_{k+1,k+2}$ which overlap in the $s=k+2-J$ middle columns. That is, $V(G_{k+1,k+2+J})= A \cup B$ where $A$ is the set of the first $(k+1)(k+2)$ vertices and $B$ is the set of the last $(k+1)(k+2)$ vertices. Note that these two sets have $|A \cap B| = s(k+1)$ vertices in common. Color the vertices of $A$ with the standard block coloring $\varphi$. Next, we observe that the vertices of $B$ are already partially colored. Specifically, the first $s(k+1)$ vertices received the colors $1,2,\ldots, s(k+1)$ in reverse order.

Since $J \geq \left\lceil \frac{k+2}{2} \right\rceil$, then it follows that $s \leq \left\lfloor \frac{k+2}{2} \right\rfloor$. Therefore, none of the colors given to these first $s$ columns of the graph induced on $B$ repeat. Hence, this partial coloring can be extended to the coloring of $G_{k+1,k+2}$. Specifically, it can be extended to the coloring $h_s(\varphi)$, though there could be several possibilities depending on $s$. In this way, each vertex is now contained within at least one copy of $G_{k+1,k+2}$ under a coloring isomorphic to the standard block coloring. Therefore, each vertex is within distance $k$ from every color. This process of extending the coloring is illustrated in Figure~\ref{hs}.
\end{proof}

\begin{figure}
\begin{subfigure}[b]{0.48\linewidth}
          \centering

          \resizebox{\linewidth}{!}{

\tikzset{state/.style={circle,fill=blue!20,draw,minimum size=1.5em,inner sep=0pt},
            }
\begin{tikzpicture}

\filldraw[fill=blue!70!white, draw=black] (-0.4,0.4) rectangle (4.8,-3.7);

    \node[state] (1) at (0,0)  {1};
    \node[state] (2) [below=0.5cm of 1] {2};
    \node[state] (3) [below =0.5cm of 2] {3};
    \node[state] (4) [below =0.5cm of 3] {4};
    \node[state] (5) [right =0.5cm of 1] {5};
    \node[state] (6) [below=0.5cm of 5 ] {6};
    \node[state] (7) [below=0.5cm of 6] {7};
    \node[state] (8) [below =0.5cm of 7] {8};
    \node[state] (9) [right =0.5cm of 5] {9};
    \node[state] (10) [below =0.5cm of 9] {10};
    \node[state] (11) [below =0.5cm of 10] {10};
    \node[state] (12) [below =0.5cm of 11] {9};
    \node[state] (13) [right =0.5cm of 9] {8};
    \node[state] (14) [below =0.5cm of 13] {7};
    \node[state] (15) [below =0.5cm of 14] {6};
    \node[state] (16) [below =0.5cm of 15] {5};
    \node[state] (17) [right =0.5cm of 13] {4};
    \node[state] (18) [below =0.5cm of 17] {3};
    \node[state] (19) [below =0.5cm of 18] {2};
    \node[state] (20) [below =0.5cm of 19] {1};
    
    \node[state] (41) [right=0.5cm of 17] {};
    \node[state] (42) [below=0.5cm of 41] {};
    \node[state] (43) [below =0.5cm of 42] {};
    \node[state] (44) [below =0.5cm of 43] {};
    \node[state] (45) [right =0.5cm of 41] {};
    \node[state] (46) [below=0.5cm of 45 ] {};
    \node[state] (47) [below=0.5cm of 46] {};
    \node[state] (48) [below =0.5cm of 47] {};
    \node[state] (49) [right =0.5cm of 45] {};
    \node[state] (50) [below =0.5cm of 49] {};
    \node[state] (51) [below =0.5cm of 50] {};
    \node[state] (52) [below =0.5cm of 51] {};

    \path[draw,thick]
    (1) edge node {} (2)
    (2) edge node {} (6)
    (2) edge node {} (3)
    (3) edge node {} (4)
    (3) edge node {} (7)
    (4) edge node {} (8)
    (1) edge node {} (5)
    (5) edge node {} (6)
    (5) edge node {} (9)
    (6) edge node {} (10)
    (6) edge node {} (7)
    (7) edge node {} (11)
    (7) edge node {} (8)
    (8) edge node {} (12)
    (10) edge node {} (11)
    (9) edge node {} (10)
    (11) edge node {} (12)
    (12) edge node {} (16)
    (13) edge node {} (9)
    (13) edge node {} (14)
    (14) edge node {} (15)
    (15) edge node {} (16)
    (10) edge node {} (14)
    (11) edge node {} (15)
    (12) edge node {} (16)
    (17) edge node {} (13)
    (18) edge node {} (14)
    (19) edge node {} (15)
    (20) edge node {} (16)
    (17) edge node {} (18)
    (18) edge node {} (19)
    (19) edge node {} (20)
    
    (17) edge node {} (41)
    (18) edge node {} (42)
    (19) edge node {} (43)
    (20) edge node {} (44)
    
    (41) edge node {} (42)
    (42) edge node {} (46)
    (42) edge node {} (43)
    (43) edge node {} (44)
    (43) edge node {} (47)
    (44) edge node {} (48)
    (41) edge node {} (45)
    (45) edge node {} (46)
    (45) edge node {} (49)
    (46) edge node {} (50)
    (46) edge node {} (47)
    (47) edge node {} (51)
    (47) edge node {} (48)
    (48) edge node {} (52)
    (50) edge node {} (51)
    (49) edge node {} (50)
    (51) edge node {} (52);
\end{tikzpicture}
}
\caption{}
\end{subfigure}
\begin{subfigure}[b]{0.48\linewidth}
          \centering
                    \resizebox{\linewidth}{!}{
\tikzset{state/.style={circle,fill=blue!20,draw,minimum size=1.5em,inner sep=0pt},
            }
\begin{tikzpicture}

\filldraw[fill=blue!70!white, draw=black] (-0.4,0.4) rectangle (4.8,-3.7);
\filldraw[fill=red!50!white, draw=black, opacity=0.6] (2.8,0.4) rectangle (8,-3.7);

    \node[state] (1) at (0,0)  {1};
    \node[state] (2) [below=0.5cm of 1] {2};
    \node[state] (3) [below =0.5cm of 2] {3};
    \node[state] (4) [below =0.5cm of 3] {4};
    \node[state] (5) [right =0.5cm of 1] {5};
    \node[state] (6) [below=0.5cm of 5 ] {6};
    \node[state] (7) [below=0.5cm of 6] {7};
    \node[state] (8) [below =0.5cm of 7] {8};
    \node[state] (9) [right =0.5cm of 5] {9};
    \node[state] (10) [below =0.5cm of 9] {10};
    \node[state] (11) [below =0.5cm of 10] {10};
    \node[state] (12) [below =0.5cm of 11] {9};
    \node[state] (13) [right =0.5cm of 9] {8};
    \node[state] (14) [below =0.5cm of 13] {7};
    \node[state] (15) [below =0.5cm of 14] {6};
    \node[state] (16) [below =0.5cm of 15] {5};
    \node[state] (17) [right =0.5cm of 13] {4};
    \node[state] (18) [below =0.5cm of 17] {3};
    \node[state] (19) [below =0.5cm of 18] {2};
    \node[state] (20) [below =0.5cm of 19] {1};
    
    \node[state] (41) [right=0.5cm of 17] {};
    \node[state] (42) [below=0.5cm of 41] {};
    \node[state] (43) [below =0.5cm of 42] {};
    \node[state] (44) [below =0.5cm of 43] {};
    \node[state] (45) [right =0.5cm of 41] {1};
    \node[state] (46) [below=0.5cm of 45 ] {2};
    \node[state] (47) [below=0.5cm of 46] {3};
    \node[state] (48) [below =0.5cm of 47] {4};
    \node[state] (49) [right =0.5cm of 45] {5};
    \node[state] (50) [below =0.5cm of 49] {6};
    \node[state] (51) [below =0.5cm of 50] {7};
    \node[state] (52) [below =0.5cm of 51] {8};

    \path[draw,thick]
    (1) edge node {} (2)
    (2) edge node {} (6)
    (2) edge node {} (3)
    (3) edge node {} (4)
    (3) edge node {} (7)
    (4) edge node {} (8)
    (1) edge node {} (5)
    (5) edge node {} (6)
    (5) edge node {} (9)
    (6) edge node {} (10)
    (6) edge node {} (7)
    (7) edge node {} (11)
    (7) edge node {} (8)
    (8) edge node {} (12)
    (10) edge node {} (11)
    (9) edge node {} (10)
    (11) edge node {} (12)
    (12) edge node {} (16)
    (13) edge node {} (9)
    (13) edge node {} (14)
    (14) edge node {} (15)
    (15) edge node {} (16)
    (10) edge node {} (14)
    (11) edge node {} (15)
    (12) edge node {} (16)
    (17) edge node {} (13)
    (18) edge node {} (14)
    (19) edge node {} (15)
    (20) edge node {} (16)
    (17) edge node {} (18)
    (18) edge node {} (19)
    (19) edge node {} (20)
    
    (17) edge node {} (41)
    (18) edge node {} (42)
    (19) edge node {} (43)
    (20) edge node {} (44)
    
    (41) edge node {} (42)
    (42) edge node {} (46)
    (42) edge node {} (43)
    (43) edge node {} (44)
    (43) edge node {} (47)
    (44) edge node {} (48)
    (41) edge node {} (45)
    (45) edge node {} (46)
    (45) edge node {} (49)
    (46) edge node {} (50)
    (46) edge node {} (47)
    (47) edge node {} (51)
    (47) edge node {} (48)
    (48) edge node {} (52)
    (50) edge node {} (51)
    (49) edge node {} (50)
    (51) edge node {} (52);
\end{tikzpicture}
          }
\caption{}
\end{subfigure}
\begin{subfigure}[b]{0.48\linewidth}
          \centering
                    \resizebox{\linewidth}{!}{
\tikzset{state/.style={circle,fill=blue!20,draw,minimum size=1.5em,inner sep=0pt},
            }
\begin{tikzpicture}

\filldraw[fill=blue!70!white, draw=black] (-0.4,0.4) rectangle (4.8,-3.7);
\filldraw[fill=red!50!white, draw=black, opacity=0.6] (2.8,0.4) rectangle (8,-3.7);

    \node[state] (1) at (0,0)  {1};
    \node[state] (2) [below=0.5cm of 1] {2};
    \node[state] (3) [below =0.5cm of 2] {3};
    \node[state] (4) [below =0.5cm of 3] {4};
    \node[state] (5) [right =0.5cm of 1] {5};
    \node[state] (6) [below=0.5cm of 5 ] {6};
    \node[state] (7) [below=0.5cm of 6] {7};
    \node[state] (8) [below =0.5cm of 7] {8};
    \node[state] (9) [right =0.5cm of 5] {9};
    \node[state] (10) [below =0.5cm of 9] {10};
    \node[state] (11) [below =0.5cm of 10] {10};
    \node[state] (12) [below =0.5cm of 11] {9};
    \node[state] (13) [right =0.5cm of 9] {8};
    \node[state] (14) [below =0.5cm of 13] {7};
    \node[state] (15) [below =0.5cm of 14] {6};
    \node[state] (16) [below =0.5cm of 15] {5};
    \node[state] (17) [right =0.5cm of 13] {4};
    \node[state] (18) [below =0.5cm of 17] {3};
    \node[state] (19) [below =0.5cm of 18] {2};
    \node[state] (20) [below =0.5cm of 19] {1};
    
    \node[state] (41) [right=0.5cm of 17] {9};
    \node[state] (42) [below=0.5cm of 41] {10};
    \node[state] (43) [below =0.5cm of 42] {10};
    \node[state] (44) [below =0.5cm of 43] {9};
    \node[state] (45) [right =0.5cm of 41] {1};
    \node[state] (46) [below=0.5cm of 45 ] {2};
    \node[state] (47) [below=0.5cm of 46] {3};
    \node[state] (48) [below =0.5cm of 47] {4};
    \node[state] (49) [right =0.5cm of 45] {5};
    \node[state] (50) [below =0.5cm of 49] {6};
    \node[state] (51) [below =0.5cm of 50] {7};
    \node[state] (52) [below =0.5cm of 51] {8};

    \path[draw,thick]
    (1) edge node {} (2)
    (2) edge node {} (6)
    (2) edge node {} (3)
    (3) edge node {} (4)
    (3) edge node {} (7)
    (4) edge node {} (8)
    (1) edge node {} (5)
    (5) edge node {} (6)
    (5) edge node {} (9)
    (6) edge node {} (10)
    (6) edge node {} (7)
    (7) edge node {} (11)
    (7) edge node {} (8)
    (8) edge node {} (12)
    (10) edge node {} (11)
    (9) edge node {} (10)
    (11) edge node {} (12)
    (12) edge node {} (16)
    (13) edge node {} (9)
    (13) edge node {} (14)
    (14) edge node {} (15)
    (15) edge node {} (16)
    (10) edge node {} (14)
    (11) edge node {} (15)
    (12) edge node {} (16)
    (17) edge node {} (13)
    (18) edge node {} (14)
    (19) edge node {} (15)
    (20) edge node {} (16)
    (17) edge node {} (18)
    (18) edge node {} (19)
    (19) edge node {} (20)
    
    (17) edge node {} (41)
    (18) edge node {} (42)
    (19) edge node {} (43)
    (20) edge node {} (44)
    
    (41) edge node {} (42)
    (42) edge node {} (46)
    (42) edge node {} (43)
    (43) edge node {} (44)
    (43) edge node {} (47)
    (44) edge node {} (48)
    (41) edge node {} (45)
    (45) edge node {} (46)
    (45) edge node {} (49)
    (46) edge node {} (50)
    (46) edge node {} (47)
    (47) edge node {} (51)
    (47) edge node {} (48)
    (48) edge node {} (52)
    (50) edge node {} (51)
    (49) edge node {} (50)
    (51) edge node {} (52);
\end{tikzpicture}
          }
\caption{}
\end{subfigure}
\caption{}
\label{hs}
\end{figure}

\begin{lemma}
\label{c2}
For any positive integer $k$,
\[d_k\left(G_{k+1,2(k+2)+J}\right) \geq \frac{(k+1)(k+2)}{2},\] for any $1 \leq J < \left\lceil \frac{k+2}{2} \right\rceil$.
\end{lemma}

\begin{proof}
Consider this graph as three copies of $G_{k+1,k+2}$. The first copy, call it graph $A$, is comprised of the first $k+2$ columns. The second copy, call it graph $B$, is comprised of the $k+2$ columns which begin at column $\left\lceil \frac{k+2}{2} \right\rceil +1$ and end with column $\left\lceil \frac{k+2}{2} \right\rceil +k+2$. $B$ overlaps with $A$ in \[s_1 = \left\lfloor \frac{k+2}{2} \right\rfloor\] columns - the first $s_1$ columns of $B$ are the same as the final $s_1$ columns of $A$. The third copy, call it graph $C$, is comprised of the final $k+2$ columns which means that it overlaps $B$ in \[s_2 = \left\lceil \frac{k+2}{2} \right\rceil - J\] columns - the first $s_2$ columns of $C$ are the final $s_2$ columns of $B$.

Now, we color $A$ with $\varphi$. Then we color $B$ with $h_{s_1}(\varphi)$. Since $s_1 \leq \left\lfloor \frac{k+2}{2} \right\rfloor$, then these two colorings agree on the vertices where they overlap. Finally, we color $C$ with $h_{s_2}(h_{s_1}(\varphi)$. Since $s_2 \leq \left\lfloor \frac{k+2}{2} \right\rfloor$, then the coloring of $C$ agrees with the coloring of $B$ on the vertices where both colors are applied. This process is illustrated in Figure~\ref{hss}.
\end{proof}

\begin{figure}
\begin{subfigure}[b]{0.48\linewidth}
          \centering

          \resizebox{\linewidth}{!}{

\tikzset{state/.style={circle,fill=blue!20,draw,minimum size=1.5em,inner sep=0pt},
            }
\begin{tikzpicture}

\filldraw[fill=blue!70!white, draw=black] (-0.4,0.4) rectangle (4.8,-3.7);

    \node[state] (1) at (0,0)  {1};
    \node[state] (2) [below=0.5cm of 1] {2};
    \node[state] (3) [below =0.5cm of 2] {3};
    \node[state] (4) [below =0.5cm of 3] {4};
    \node[state] (5) [right =0.5cm of 1] {5};
    \node[state] (6) [below=0.5cm of 5 ] {6};
    \node[state] (7) [below=0.5cm of 6] {7};
    \node[state] (8) [below =0.5cm of 7] {8};
    \node[state] (9) [right =0.5cm of 5] {9};
    \node[state] (10) [below =0.5cm of 9] {10};
    \node[state] (11) [below =0.5cm of 10] {10};
    \node[state] (12) [below =0.5cm of 11] {9};
    \node[state] (13) [right =0.5cm of 9] {8};
    \node[state] (14) [below =0.5cm of 13] {7};
    \node[state] (15) [below =0.5cm of 14] {6};
    \node[state] (16) [below =0.5cm of 15] {5};
    \node[state] (17) [right =0.5cm of 13] {4};
    \node[state] (18) [below =0.5cm of 17] {3};
    \node[state] (19) [below =0.5cm of 18] {2};
    \node[state] (20) [below =0.5cm of 19] {1};
    
    \node[state] (21) [right=0.5cm of 17]  {};
    \node[state] (22) [below=0.5cm of 21] {};
    \node[state] (23) [below =0.5cm of 22] {};
    \node[state] (24) [below =0.5cm of 23] {};
    \node[state] (25) [right =0.5cm of 21] {};
    \node[state] (26) [below=0.5cm of 25 ] {};
    \node[state] (27) [below=0.5cm of 26] {};
    \node[state] (28) [below =0.5cm of 27] {};
    \node[state] (29) [right =0.5cm of 25] {};
    \node[state] (30) [below =0.5cm of 29] {};
    \node[state] (31) [below =0.5cm of 30] {};
    \node[state] (32) [below =0.5cm of 31] {};
    \node[state] (33) [right =0.5cm of 29] {};
    \node[state] (34) [below =0.5cm of 33] {};
    \node[state] (35) [below =0.5cm of 34] {};
    \node[state] (36) [below =0.5cm of 35] {};
    \node[state] (37) [right =0.5cm of 33] {};
    \node[state] (38) [below =0.5cm of 37] {};
    \node[state] (39) [below =0.5cm of 38] {};
    \node[state] (40) [below =0.5cm of 39] {};
    
    \node[state] (41) [right=0.5cm of 37]  {};
    \node[state] (42) [below=0.5cm of 41] {};
    \node[state] (43) [below =0.5cm of 42] {};
    \node[state] (44) [below =0.5cm of 43] {};

    \path[draw,thick]
    (1) edge node {} (2)
    (2) edge node {} (6)
    (2) edge node {} (3)
    (3) edge node {} (4)
    (3) edge node {} (7)
    (4) edge node {} (8)
    (1) edge node {} (5)
    (5) edge node {} (6)
    (5) edge node {} (9)
    (6) edge node {} (10)
    (6) edge node {} (7)
    (7) edge node {} (11)
    (7) edge node {} (8)
    (8) edge node {} (12)
    (10) edge node {} (11)
    (9) edge node {} (10)
    (11) edge node {} (12)
    (12) edge node {} (16)
    (13) edge node {} (9)
    (13) edge node {} (14)
    (14) edge node {} (15)
    (15) edge node {} (16)
    (10) edge node {} (14)
    (11) edge node {} (15)
    (12) edge node {} (16)
    (17) edge node {} (13)
    (18) edge node {} (14)
    (19) edge node {} (15)
    (20) edge node {} (16)
    (17) edge node {} (18)
    (18) edge node {} (19)
    (19) edge node {} (20)
    
        (20) edge node {} (24)
            (19) edge node {} (23)
                (18) edge node {} (22)
                    (17) edge node {} (21)
    
  (21) edge node {} (22)
    (22) edge node {} (26)
    (22) edge node {} (23)
    (23) edge node {} (24)
    (23) edge node {} (27)
    (24) edge node {} (28)
    (21) edge node {} (25)
    (25) edge node {} (26)
    (25) edge node {} (29)
    (26) edge node {} (30)
    (26) edge node {} (27)
    (27) edge node {} (31)
    (27) edge node {} (28)
    (28) edge node {} (32)
    (30) edge node {} (31)
    (29) edge node {} (30)
    (31) edge node {} (32)
    (32) edge node {} (36)
    (33) edge node {} (29)
    (33) edge node {} (34)
    (34) edge node {} (35)
    (35) edge node {} (36)
    (30) edge node {} (34)
    (31) edge node {} (35)
    (32) edge node {} (36)
    (37) edge node {} (33)
    (38) edge node {} (34)
    (39) edge node {} (35)
    (40) edge node {} (36)
    (37) edge node {} (38)
    (38) edge node {} (39)
    (39) edge node {} (40)
    
            (40) edge node {} (44)
            (39) edge node {} (43)
                (38) edge node {} (42)
                    (37) edge node {} (41)
                    
                            (41) edge node {} (42)
            (42) edge node {} (43)
                (43) edge node {} (44);
\end{tikzpicture}
}
\caption{}
\end{subfigure}
\begin{subfigure}[b]{0.48\linewidth}
          \centering

          \resizebox{\linewidth}{!}{

\tikzset{state/.style={circle,fill=blue!20,draw,minimum size=1.5em,inner sep=0pt},
            }
\begin{tikzpicture}

\filldraw[fill=blue!70!white, draw=black] (-0.4,0.4) rectangle (4.8,-3.7);
\filldraw[fill=red!50!white, draw=black, opacity=0.6] (2.8,0.4) rectangle (8,-3.7);

    \node[state] (1) at (0,0)  {1};
    \node[state] (2) [below=0.5cm of 1] {2};
    \node[state] (3) [below =0.5cm of 2] {3};
    \node[state] (4) [below =0.5cm of 3] {4};
    \node[state] (5) [right =0.5cm of 1] {5};
    \node[state] (6) [below=0.5cm of 5 ] {6};
    \node[state] (7) [below=0.5cm of 6] {7};
    \node[state] (8) [below =0.5cm of 7] {8};
    \node[state] (9) [right =0.5cm of 5] {9};
    \node[state] (10) [below =0.5cm of 9] {10};
    \node[state] (11) [below =0.5cm of 10] {10};
    \node[state] (12) [below =0.5cm of 11] {9};
    \node[state] (13) [right =0.5cm of 9] {8};
    \node[state] (14) [below =0.5cm of 13] {7};
    \node[state] (15) [below =0.5cm of 14] {6};
    \node[state] (16) [below =0.5cm of 15] {5};
    \node[state] (17) [right =0.5cm of 13] {4};
    \node[state] (18) [below =0.5cm of 17] {3};
    \node[state] (19) [below =0.5cm of 18] {2};
    \node[state] (20) [below =0.5cm of 19] {1};
    
    \node[state] (21) [right=0.5cm of 17]  {9};
    \node[state] (22) [below=0.5cm of 21] {10};
    \node[state] (23) [below =0.5cm of 22] {10};
    \node[state] (24) [below =0.5cm of 23] {9};
    \node[state] (25) [right =0.5cm of 21] {1};
    \node[state] (26) [below=0.5cm of 25 ] {2};
    \node[state] (27) [below=0.5cm of 26] {3};
    \node[state] (28) [below =0.5cm of 27] {4};
    \node[state] (29) [right =0.5cm of 25] {5};
    \node[state] (30) [below =0.5cm of 29] {6};
    \node[state] (31) [below =0.5cm of 30] {7};
    \node[state] (32) [below =0.5cm of 31] {8};

    \node[state] (33) [right =0.5cm of 29] {};
    \node[state] (34) [below =0.5cm of 33] {};
    \node[state] (35) [below =0.5cm of 34] {};
    \node[state] (36) [below =0.5cm of 35] {};
    \node[state] (37) [right =0.5cm of 33] {};
    \node[state] (38) [below =0.5cm of 37] {};
    \node[state] (39) [below =0.5cm of 38] {};
    \node[state] (40) [below =0.5cm of 39] {};
    
    \node[state] (41) [right=0.5cm of 37]  {};
    \node[state] (42) [below=0.5cm of 41] {};
    \node[state] (43) [below =0.5cm of 42] {};
    \node[state] (44) [below =0.5cm of 43] {};

    \path[draw,thick]
    (1) edge node {} (2)
    (2) edge node {} (6)
    (2) edge node {} (3)
    (3) edge node {} (4)
    (3) edge node {} (7)
    (4) edge node {} (8)
    (1) edge node {} (5)
    (5) edge node {} (6)
    (5) edge node {} (9)
    (6) edge node {} (10)
    (6) edge node {} (7)
    (7) edge node {} (11)
    (7) edge node {} (8)
    (8) edge node {} (12)
    (10) edge node {} (11)
    (9) edge node {} (10)
    (11) edge node {} (12)
    (12) edge node {} (16)
    (13) edge node {} (9)
    (13) edge node {} (14)
    (14) edge node {} (15)
    (15) edge node {} (16)
    (10) edge node {} (14)
    (11) edge node {} (15)
    (12) edge node {} (16)
    (17) edge node {} (13)
    (18) edge node {} (14)
    (19) edge node {} (15)
    (20) edge node {} (16)
    (17) edge node {} (18)
    (18) edge node {} (19)
    (19) edge node {} (20)
    
        (20) edge node {} (24)
            (19) edge node {} (23)
                (18) edge node {} (22)
                    (17) edge node {} (21)
    
  (21) edge node {} (22)
    (22) edge node {} (26)
    (22) edge node {} (23)
    (23) edge node {} (24)
    (23) edge node {} (27)
    (24) edge node {} (28)
    (21) edge node {} (25)
    (25) edge node {} (26)
    (25) edge node {} (29)
    (26) edge node {} (30)
    (26) edge node {} (27)
    (27) edge node {} (31)
    (27) edge node {} (28)
    (28) edge node {} (32)
    (30) edge node {} (31)
    (29) edge node {} (30)
    (31) edge node {} (32)
    (32) edge node {} (36)
    (33) edge node {} (29)
    (33) edge node {} (34)
    (34) edge node {} (35)
    (35) edge node {} (36)
    (30) edge node {} (34)
    (31) edge node {} (35)
    (32) edge node {} (36)
    (37) edge node {} (33)
    (38) edge node {} (34)
    (39) edge node {} (35)
    (40) edge node {} (36)
    (37) edge node {} (38)
    (38) edge node {} (39)
    (39) edge node {} (40)
    
            (40) edge node {} (44)
            (39) edge node {} (43)
                (38) edge node {} (42)
                    (37) edge node {} (41)
                    
                            (41) edge node {} (42)
            (42) edge node {} (43)
                (43) edge node {} (44);
\end{tikzpicture}
}
\caption{}
\end{subfigure}
\begin{subfigure}[b]{0.48\linewidth}
          \centering

          \resizebox{\linewidth}{!}{

\tikzset{state/.style={circle,fill=blue!20,draw,minimum size=1.5em,inner sep=0pt},
            }
\begin{tikzpicture}

\filldraw[fill=blue!70!white, draw=black] (-0.4,0.4) rectangle (4.8,-3.7);
\filldraw[fill=red!50!white, draw=black, opacity=0.6] (2.8,0.4) rectangle (8,-3.7);
\filldraw[fill=green!50!white, draw=black, opacity=0.6] (6,0.4) rectangle (11.4,-3.7);

    \node[state] (1) at (0,0)  {1};
    \node[state] (2) [below=0.5cm of 1] {2};
    \node[state] (3) [below =0.5cm of 2] {3};
    \node[state] (4) [below =0.5cm of 3] {4};
    \node[state] (5) [right =0.5cm of 1] {5};
    \node[state] (6) [below=0.5cm of 5 ] {6};
    \node[state] (7) [below=0.5cm of 6] {7};
    \node[state] (8) [below =0.5cm of 7] {8};
    \node[state] (9) [right =0.5cm of 5] {9};
    \node[state] (10) [below =0.5cm of 9] {10};
    \node[state] (11) [below =0.5cm of 10] {10};
    \node[state] (12) [below =0.5cm of 11] {9};
    \node[state] (13) [right =0.5cm of 9] {8};
    \node[state] (14) [below =0.5cm of 13] {7};
    \node[state] (15) [below =0.5cm of 14] {6};
    \node[state] (16) [below =0.5cm of 15] {5};
    \node[state] (17) [right =0.5cm of 13] {4};
    \node[state] (18) [below =0.5cm of 17] {3};
    \node[state] (19) [below =0.5cm of 18] {2};
    \node[state] (20) [below =0.5cm of 19] {1};
    
    \node[state] (21) [right=0.5cm of 17]  {9};
    \node[state] (22) [below=0.5cm of 21] {10};
    \node[state] (23) [below =0.5cm of 22] {10};
    \node[state] (24) [below =0.5cm of 23] {9};
    \node[state] (25) [right =0.5cm of 21] {1};
    \node[state] (26) [below=0.5cm of 25 ] {2};
    \node[state] (27) [below=0.5cm of 26] {3};
    \node[state] (28) [below =0.5cm of 27] {4};
    \node[state] (29) [right =0.5cm of 25] {5};
    \node[state] (30) [below =0.5cm of 29] {6};
    \node[state] (31) [below =0.5cm of 30] {7};
    \node[state] (32) [below =0.5cm of 31] {8};

    \node[state] (33) [right =0.5cm of 29] {9};
    \node[state] (34) [below =0.5cm of 33] {10};
    \node[state] (35) [below =0.5cm of 34] {10};
    \node[state] (36) [below =0.5cm of 35] {9};
    \node[state] (37) [right =0.5cm of 33] {1};
    \node[state] (38) [below =0.5cm of 37] {2};
    \node[state] (39) [below =0.5cm of 38] {3};
    \node[state] (40) [below =0.5cm of 39] {4};
    
    \node[state] (41) [right=0.5cm of 37]  {5};
    \node[state] (42) [below=0.5cm of 41] {6};
    \node[state] (43) [below =0.5cm of 42] {7};
    \node[state] (44) [below =0.5cm of 43] {8};

    \path[draw,thick]
    (1) edge node {} (2)
    (2) edge node {} (6)
    (2) edge node {} (3)
    (3) edge node {} (4)
    (3) edge node {} (7)
    (4) edge node {} (8)
    (1) edge node {} (5)
    (5) edge node {} (6)
    (5) edge node {} (9)
    (6) edge node {} (10)
    (6) edge node {} (7)
    (7) edge node {} (11)
    (7) edge node {} (8)
    (8) edge node {} (12)
    (10) edge node {} (11)
    (9) edge node {} (10)
    (11) edge node {} (12)
    (12) edge node {} (16)
    (13) edge node {} (9)
    (13) edge node {} (14)
    (14) edge node {} (15)
    (15) edge node {} (16)
    (10) edge node {} (14)
    (11) edge node {} (15)
    (12) edge node {} (16)
    (17) edge node {} (13)
    (18) edge node {} (14)
    (19) edge node {} (15)
    (20) edge node {} (16)
    (17) edge node {} (18)
    (18) edge node {} (19)
    (19) edge node {} (20)
    
        (20) edge node {} (24)
            (19) edge node {} (23)
                (18) edge node {} (22)
                    (17) edge node {} (21)
    
  (21) edge node {} (22)
    (22) edge node {} (26)
    (22) edge node {} (23)
    (23) edge node {} (24)
    (23) edge node {} (27)
    (24) edge node {} (28)
    (21) edge node {} (25)
    (25) edge node {} (26)
    (25) edge node {} (29)
    (26) edge node {} (30)
    (26) edge node {} (27)
    (27) edge node {} (31)
    (27) edge node {} (28)
    (28) edge node {} (32)
    (30) edge node {} (31)
    (29) edge node {} (30)
    (31) edge node {} (32)
    (32) edge node {} (36)
    (33) edge node {} (29)
    (33) edge node {} (34)
    (34) edge node {} (35)
    (35) edge node {} (36)
    (30) edge node {} (34)
    (31) edge node {} (35)
    (32) edge node {} (36)
    (37) edge node {} (33)
    (38) edge node {} (34)
    (39) edge node {} (35)
    (40) edge node {} (36)
    (37) edge node {} (38)
    (38) edge node {} (39)
    (39) edge node {} (40)
    
            (40) edge node {} (44)
            (39) edge node {} (43)
                (38) edge node {} (42)
                    (37) edge node {} (41)
                    
                            (41) edge node {} (42)
            (42) edge node {} (43)
                (43) edge node {} (44);
\end{tikzpicture}
}
\caption{}
\end{subfigure}

\caption{}
\label{hss}
\end{figure}

\begin{lemma}
\label{r1}
For any positive integer $k$,
\[d_k\left(G_{k+1+I,k+2}\right) \geq \frac{(k+1)(k+2)}{2},\] for any $\left\lceil \frac{k+1}{2} \right\rceil \leq I < k+1$.
\end{lemma}

\begin{proof}
Consider this graph as two copies of $G_{k+1,k+2}$ which overlap in the $t=k+1-I$ middle rows. Color the first copy of $G_{k+1,k+2}$ (the first $k+1$ rows of our graph) with the standard block coloring $\varphi$. Next, observe that the second copy of $G_{k+1,k+2}$ (the final $k+1$ rows) is already partially colored.

Under the standard block coloring, a color that appears in row $i$ only appears again in row $k+2-i$. Since $I \geq \left\lceil \frac{k+1}{2} \right\rceil$, then $t \leq \left\lfloor \frac{k+1}{2} \right\rfloor$. Therefore, if a color appears in some row $i$ that belongs to the $t$ rows of overlap between our two copies of $G_{k+1,k+2}$, then \[k+2-t \leq i \leq k+1.\]  So it follows that \[t \geq k+2-i \geq 1.\] Thus, the color only appears again in a row above the intersection. Hence, the partial coloring given to the second $G_{k+1,k+2}$ does not repeat any colors.

Similarly, since $I \geq \left\lceil \frac{k+1}{2} \right\rceil$ rows remain to be colored, then it follows that this partial coloring can be extended to a coloring of the entire second copy of $G_{k+1,k+2}$ that is isomorphic to the standard block coloring. Specifically, it can be extended to the coloring $v_t(\varphi)$, though there could be several possibilities depending on $t$. This process is illustrated in Figure~\ref{vt}.

In this way, each vertex is now contained within at least one copy of $G_{k+1,k+2}$ under a coloring isomorphic to the standard block coloring. Therefore, each vertex is within distance $k$ from every color.
\end{proof}

\begin{figure}
\begin{subfigure}[b]{0.3\linewidth}
          \centering

          \resizebox{\linewidth}{!}{

\tikzset{state/.style={circle,fill=blue!20,draw,minimum size=1.5em,inner sep=0pt},
            }
\begin{tikzpicture}

\filldraw[fill=blue!70!white, draw=black] (-0.4,0.4) rectangle (4.8,-3.7);

    \node[state] (1) at (0,0)  {1};
    \node[state] (2) [below=0.5cm of 1] {2};
    \node[state] (3) [below =0.5cm of 2] {3};
    \node[state] (4) [below =0.5cm of 3] {4};
    \node[state] (5) [right =0.5cm of 1] {5};
    \node[state] (6) [below=0.5cm of 5 ] {6};
    \node[state] (7) [below=0.5cm of 6] {7};
    \node[state] (8) [below =0.5cm of 7] {8};
    \node[state] (9) [right =0.5cm of 5] {9};
    \node[state] (10) [below =0.5cm of 9] {10};
    \node[state] (11) [below =0.5cm of 10] {10};
    \node[state] (12) [below =0.5cm of 11] {9};
    \node[state] (13) [right =0.5cm of 9] {8};
    \node[state] (14) [below =0.5cm of 13] {7};
    \node[state] (15) [below =0.5cm of 14] {6};
    \node[state] (16) [below =0.5cm of 15] {5};
    \node[state] (17) [right =0.5cm of 13] {4};
    \node[state] (18) [below =0.5cm of 17] {3};
    \node[state] (19) [below =0.5cm of 18] {2};
    \node[state] (20) [below =0.5cm of 19] {1};
    
    \node[state] (21) [below=0.5cm of 4] {};
    \node[state] (22) [below=0.5cm of 21] {};
    \node[state] (23) [below =0.5cm of 22] {};
    \node[state] (25) [right =0.5cm of 21] {};
    \node[state] (26) [below=0.5cm of 25 ] {};
    \node[state] (27) [below=0.5cm of 26] {};
    \node[state] (29) [right =0.5cm of 25] {};
    \node[state] (30) [below =0.5cm of 29] {};
    \node[state] (31) [below =0.5cm of 30] {};
    \node[state] (33) [right =0.5cm of 29] {};
    \node[state] (34) [below =0.5cm of 33] {};
    \node[state] (35) [below =0.5cm of 34] {};
    \node[state] (37) [right =0.5cm of 33] {};
    \node[state] (38) [below =0.5cm of 37] {};
    \node[state] (39) [below =0.5cm of 38] {};
    
    \path[draw,thick]
    (1) edge node {} (2)
    (2) edge node {} (6)
    (2) edge node {} (3)
    (3) edge node {} (4)
    (3) edge node {} (7)
    (4) edge node {} (8)
    (1) edge node {} (5)
    (5) edge node {} (6)
    (5) edge node {} (9)
    (6) edge node {} (10)
    (6) edge node {} (7)
    (7) edge node {} (11)
    (7) edge node {} (8)
    (8) edge node {} (12)
    (10) edge node {} (11)
    (9) edge node {} (10)
    (11) edge node {} (12)
    (12) edge node {} (16)
    (13) edge node {} (9)
    (13) edge node {} (14)
    (14) edge node {} (15)
    (15) edge node {} (16)
    (10) edge node {} (14)
    (11) edge node {} (15)
    (12) edge node {} (16)
    (17) edge node {} (13)
    (18) edge node {} (14)
    (19) edge node {} (15)
    (20) edge node {} (16)
    (17) edge node {} (18)
    (18) edge node {} (19)
    (19) edge node {} (20)
    
    (4) edge node {} (21)
    (8) edge node {} (25)
    (12) edge node {} (29)
    (16) edge node {} (33)
    (20) edge node {} (37)
    
    (21) edge node {} (22)
    (22) edge node {} (26)
    (22) edge node {} (23)
    (23) edge node {} (27)
    (21) edge node {} (25)
    (25) edge node {} (26)
    (25) edge node {} (29)
    (26) edge node {} (30)
    (26) edge node {} (27)
    (27) edge node {} (31)
    (30) edge node {} (31)
    (29) edge node {} (30)
    (33) edge node {} (29)
    (33) edge node {} (34)
    (34) edge node {} (35)
    (30) edge node {} (34)
    (31) edge node {} (35)
    (37) edge node {} (33)
    (38) edge node {} (34)
    (39) edge node {} (35)
    (37) edge node {} (38)
    (38) edge node {} (39);

\end{tikzpicture}
}
\caption{}
\end{subfigure}
\begin{subfigure}[b]{0.3\linewidth}
          \centering

          \resizebox{\linewidth}{!}{

\tikzset{state/.style={circle,fill=blue!20,draw,minimum size=1.5em,inner sep=0pt},
            }
\begin{tikzpicture}

\filldraw[fill=blue!70!white, draw=black] (-0.4,0.4) rectangle (4.8,-3.7);
\filldraw[fill=red!50!white, draw=black, opacity=0.6] (-0.4,-2.8) rectangle (4.8,-6.875);

    \node[state] (1) at (0,0)  {1};
    \node[state] (2) [below=0.5cm of 1] {2};
    \node[state] (3) [below =0.5cm of 2] {3};
    \node[state] (4) [below =0.5cm of 3] {4};
    \node[state] (5) [right =0.5cm of 1] {5};
    \node[state] (6) [below=0.5cm of 5 ] {6};
    \node[state] (7) [below=0.5cm of 6] {7};
    \node[state] (8) [below =0.5cm of 7] {8};
    \node[state] (9) [right =0.5cm of 5] {9};
    \node[state] (10) [below =0.5cm of 9] {10};
    \node[state] (11) [below =0.5cm of 10] {10};
    \node[state] (12) [below =0.5cm of 11] {9};
    \node[state] (13) [right =0.5cm of 9] {8};
    \node[state] (14) [below =0.5cm of 13] {7};
    \node[state] (15) [below =0.5cm of 14] {6};
    \node[state] (16) [below =0.5cm of 15] {5};
    \node[state] (17) [right =0.5cm of 13] {4};
    \node[state] (18) [below =0.5cm of 17] {3};
    \node[state] (19) [below =0.5cm of 18] {2};
    \node[state] (20) [below =0.5cm of 19] {1};
    
    \node[state] (21) [below=0.5cm of 4] {};
    \node[state] (22) [below=0.5cm of 21] {};
    \node[state] (23) [below =0.5cm of 22] {1};
    \node[state] (25) [right =0.5cm of 21] {};
    \node[state] (26) [below=0.5cm of 25 ] {};
    \node[state] (27) [below=0.5cm of 26] {5};
    \node[state] (29) [right =0.5cm of 25] {};
    \node[state] (30) [below =0.5cm of 29] {};
    \node[state] (31) [below =0.5cm of 30] {9};
    \node[state] (33) [right =0.5cm of 29] {};
    \node[state] (34) [below =0.5cm of 33] {};
    \node[state] (35) [below =0.5cm of 34] {8};
    \node[state] (37) [right =0.5cm of 33] {};
    \node[state] (38) [below =0.5cm of 37] {};
    \node[state] (39) [below =0.5cm of 38] {4};
    
    \path[draw,thick]
    (1) edge node {} (2)
    (2) edge node {} (6)
    (2) edge node {} (3)
    (3) edge node {} (4)
    (3) edge node {} (7)
    (4) edge node {} (8)
    (1) edge node {} (5)
    (5) edge node {} (6)
    (5) edge node {} (9)
    (6) edge node {} (10)
    (6) edge node {} (7)
    (7) edge node {} (11)
    (7) edge node {} (8)
    (8) edge node {} (12)
    (10) edge node {} (11)
    (9) edge node {} (10)
    (11) edge node {} (12)
    (12) edge node {} (16)
    (13) edge node {} (9)
    (13) edge node {} (14)
    (14) edge node {} (15)
    (15) edge node {} (16)
    (10) edge node {} (14)
    (11) edge node {} (15)
    (12) edge node {} (16)
    (17) edge node {} (13)
    (18) edge node {} (14)
    (19) edge node {} (15)
    (20) edge node {} (16)
    (17) edge node {} (18)
    (18) edge node {} (19)
    (19) edge node {} (20)
    
    (4) edge node {} (21)
    (8) edge node {} (25)
    (12) edge node {} (29)
    (16) edge node {} (33)
    (20) edge node {} (37)
    
    (21) edge node {} (22)
    (22) edge node {} (26)
    (22) edge node {} (23)
    (23) edge node {} (27)
    (21) edge node {} (25)
    (25) edge node {} (26)
    (25) edge node {} (29)
    (26) edge node {} (30)
    (26) edge node {} (27)
    (27) edge node {} (31)
    (30) edge node {} (31)
    (29) edge node {} (30)
    (33) edge node {} (29)
    (33) edge node {} (34)
    (34) edge node {} (35)
    (30) edge node {} (34)
    (31) edge node {} (35)
    (37) edge node {} (33)
    (38) edge node {} (34)
    (39) edge node {} (35)
    (37) edge node {} (38)
    (38) edge node {} (39);

\end{tikzpicture}
}
\caption{}
\end{subfigure}
\begin{subfigure}[b]{0.3\linewidth}
          \centering

          \resizebox{\linewidth}{!}{

\tikzset{state/.style={circle,fill=blue!20,draw,minimum size=1.5em,inner sep=0pt},
            }
\begin{tikzpicture}

\filldraw[fill=blue!70!white, draw=black] (-0.4,0.4) rectangle (4.8,-3.7);
\filldraw[fill=red!50!white, draw=black, opacity=0.6] (-0.4,-2.8) rectangle (4.8,-6.875);

    \node[state] (1) at (0,0)  {1};
    \node[state] (2) [below=0.5cm of 1] {2};
    \node[state] (3) [below =0.5cm of 2] {3};
    \node[state] (4) [below =0.5cm of 3] {4};
    \node[state] (5) [right =0.5cm of 1] {5};
    \node[state] (6) [below=0.5cm of 5 ] {6};
    \node[state] (7) [below=0.5cm of 6] {7};
    \node[state] (8) [below =0.5cm of 7] {8};
    \node[state] (9) [right =0.5cm of 5] {9};
    \node[state] (10) [below =0.5cm of 9] {10};
    \node[state] (11) [below =0.5cm of 10] {10};
    \node[state] (12) [below =0.5cm of 11] {9};
    \node[state] (13) [right =0.5cm of 9] {8};
    \node[state] (14) [below =0.5cm of 13] {7};
    \node[state] (15) [below =0.5cm of 14] {6};
    \node[state] (16) [below =0.5cm of 15] {5};
    \node[state] (17) [right =0.5cm of 13] {4};
    \node[state] (18) [below =0.5cm of 17] {3};
    \node[state] (19) [below =0.5cm of 18] {2};
    \node[state] (20) [below =0.5cm of 19] {1};
    
    \node[state] (21) [below=0.5cm of 4] {2};
    \node[state] (22) [below=0.5cm of 21] {3};
    \node[state] (23) [below =0.5cm of 22] {1};
    \node[state] (25) [right =0.5cm of 21] {6};
    \node[state] (26) [below=0.5cm of 25 ] {7};
    \node[state] (27) [below=0.5cm of 26] {5};
    \node[state] (29) [right =0.5cm of 25] {10};
    \node[state] (30) [below =0.5cm of 29] {10};
    \node[state] (31) [below =0.5cm of 30] {9};
    \node[state] (33) [right =0.5cm of 29] {7};
    \node[state] (34) [below =0.5cm of 33] {6};
    \node[state] (35) [below =0.5cm of 34] {8};
    \node[state] (37) [right =0.5cm of 33] {3};
    \node[state] (38) [below =0.5cm of 37] {2};
    \node[state] (39) [below =0.5cm of 38] {4};
    
    \path[draw,thick]
    (1) edge node {} (2)
    (2) edge node {} (6)
    (2) edge node {} (3)
    (3) edge node {} (4)
    (3) edge node {} (7)
    (4) edge node {} (8)
    (1) edge node {} (5)
    (5) edge node {} (6)
    (5) edge node {} (9)
    (6) edge node {} (10)
    (6) edge node {} (7)
    (7) edge node {} (11)
    (7) edge node {} (8)
    (8) edge node {} (12)
    (10) edge node {} (11)
    (9) edge node {} (10)
    (11) edge node {} (12)
    (12) edge node {} (16)
    (13) edge node {} (9)
    (13) edge node {} (14)
    (14) edge node {} (15)
    (15) edge node {} (16)
    (10) edge node {} (14)
    (11) edge node {} (15)
    (12) edge node {} (16)
    (17) edge node {} (13)
    (18) edge node {} (14)
    (19) edge node {} (15)
    (20) edge node {} (16)
    (17) edge node {} (18)
    (18) edge node {} (19)
    (19) edge node {} (20)
    
    (4) edge node {} (21)
    (8) edge node {} (25)
    (12) edge node {} (29)
    (16) edge node {} (33)
    (20) edge node {} (37)
    
    (21) edge node {} (22)
    (22) edge node {} (26)
    (22) edge node {} (23)
    (23) edge node {} (27)
    (21) edge node {} (25)
    (25) edge node {} (26)
    (25) edge node {} (29)
    (26) edge node {} (30)
    (26) edge node {} (27)
    (27) edge node {} (31)
    (30) edge node {} (31)
    (29) edge node {} (30)
    (33) edge node {} (29)
    (33) edge node {} (34)
    (34) edge node {} (35)
    (30) edge node {} (34)
    (31) edge node {} (35)
    (37) edge node {} (33)
    (38) edge node {} (34)
    (39) edge node {} (35)
    (37) edge node {} (38)
    (38) edge node {} (39);

\end{tikzpicture}
}
\caption{}
\end{subfigure}
\caption{}
\label{vt}
\end{figure}

\begin{lemma}
\label{r2}
For any positive integer $k$,
\[d_k\left(G_{2(k+1)+I,k+2}\right) \geq \frac{(k+1)(k+2)}{2},\] for any $1\leq I<\left\lceil \frac{k+1}{2} \right\rceil$.
\end{lemma}

\begin{proof}
Consider this graph as three copies of $G_{k+1,k+2}$. The first copy, call it graph $A$, is comprised of the first $k+1$ rows. The second copy, call it graph $B$, is comprised of the $k+1$ rows which begin at row $\left\lceil \frac{k+1}{2} \right\rceil +1$ and end with row $\left\lceil \frac{k+1}{2} \right\rceil +k+1$. $B$ overlaps with $A$ in \[t_1 = \left\lfloor \frac{k+1}{2} \right\rfloor\] rows - the first $t_1$ rows of $B$ are the same as the final $t_1$ rows of $A$. The third copy, call it graph $C$, is comprised of the final $k+1$ rows which means that it overlaps $B$ in \[t_2 = \left\lceil \frac{k+1}{2} \right\rceil - I\] rows - the first $t_2$ rows of $C$ are the final $s_2$ rows of $B$.

Now, we color $A$ with $\varphi$. Then we color $B$ with $v_{t_1}(\varphi)$. Since $t_1 \leq \left\lfloor \frac{k+1}{2} \right\rfloor$, then these two colorings agree on the vertices where they overlap. Finally, we color $C$ with $h_{t_2}(h_{t_1}(\varphi))$. Since $t_2 \leq \left\lfloor \frac{k+1}{2} \right\rfloor$, then the coloring of $C$ agrees with the coloring of $B$ on the vertices where both colors are applied. This process is illustrated in Figure~\ref{vtt}.
\end{proof}

\begin{figure}
\begin{subfigure}[b]{0.3\linewidth}
          \centering

          \resizebox{\linewidth}{!}{

\tikzset{state/.style={circle,fill=blue!20,draw,minimum size=1.5em,inner sep=0pt},
            }
\begin{tikzpicture}

\filldraw[fill=blue!70!white, draw=black] (-0.4,0.4) rectangle (4.8,-3.7);

    \node[state] (1) at (0,0)  {1};
    \node[state] (2) [below=0.5cm of 1] {2};
    \node[state] (3) [below =0.5cm of 2] {3};
    \node[state] (4) [below =0.5cm of 3] {4};
    \node[state] (5) [right =0.5cm of 1] {5};
    \node[state] (6) [below=0.5cm of 5 ] {6};
    \node[state] (7) [below=0.5cm of 6] {7};
    \node[state] (8) [below =0.5cm of 7] {8};
    \node[state] (9) [right =0.5cm of 5] {9};
    \node[state] (10) [below =0.5cm of 9] {10};
    \node[state] (11) [below =0.5cm of 10] {10};
    \node[state] (12) [below =0.5cm of 11] {9};
    \node[state] (13) [right =0.5cm of 9] {8};
    \node[state] (14) [below =0.5cm of 13] {7};
    \node[state] (15) [below =0.5cm of 14] {6};
    \node[state] (16) [below =0.5cm of 15] {5};
    \node[state] (17) [right =0.5cm of 13] {4};
    \node[state] (18) [below =0.5cm of 17] {3};
    \node[state] (19) [below =0.5cm of 18] {2};
    \node[state] (20) [below =0.5cm of 19] {1};
    
    \node[state] (21) [below=0.5cm of 4] {};
    \node[state] (22) [below=0.5cm of 21] {};
    \node[state] (23) [below =0.5cm of 22] {};
    \node[state] (24) [below =0.5cm of 23] {};
    \node[state] (25) [right =0.5cm of 21] {};
    \node[state] (26) [below=0.5cm of 25 ] {};
    \node[state] (27) [below=0.5cm of 26] {};
    \node[state] (28) [below =0.5cm of 27] {};
    \node[state] (29) [right =0.5cm of 25] {};
    \node[state] (30) [below =0.5cm of 29] {};
    \node[state] (31) [below =0.5cm of 30] {};
    \node[state] (32) [below =0.5cm of 31] {};
    \node[state] (33) [right =0.5cm of 29] {};
    \node[state] (34) [below =0.5cm of 33] {};
    \node[state] (35) [below =0.5cm of 34] {};
    \node[state] (36) [below =0.5cm of 35] {};
    \node[state] (37) [right =0.5cm of 33] {};
    \node[state] (38) [below =0.5cm of 37] {};
    \node[state] (39) [below =0.5cm of 38] {};
    \node[state] (40) [below =0.5cm of 39] {};
    
    \node[state] (41) [below =0.5cm of 24] {};
    \node[state] (42) [below =0.5cm of 28] {};
    \node[state] (43) [below =0.5cm of 32] {};
    \node[state] (44) [below =0.5cm of 36] {};
    \node[state] (45) [below =0.5cm of 40] {};    
    
    \path[draw,thick]
    (1) edge node {} (2)
    (2) edge node {} (6)
    (2) edge node {} (3)
    (3) edge node {} (4)
    (3) edge node {} (7)
    (4) edge node {} (8)
    (1) edge node {} (5)
    (5) edge node {} (6)
    (5) edge node {} (9)
    (6) edge node {} (10)
    (6) edge node {} (7)
    (7) edge node {} (11)
    (7) edge node {} (8)
    (8) edge node {} (12)
    (10) edge node {} (11)
    (9) edge node {} (10)
    (11) edge node {} (12)
    (12) edge node {} (16)
    (13) edge node {} (9)
    (13) edge node {} (14)
    (14) edge node {} (15)
    (15) edge node {} (16)
    (10) edge node {} (14)
    (11) edge node {} (15)
    (12) edge node {} (16)
    (17) edge node {} (13)
    (18) edge node {} (14)
    (19) edge node {} (15)
    (20) edge node {} (16)
    (17) edge node {} (18)
    (18) edge node {} (19)
    (19) edge node {} (20)
    
    (4) edge node {} (21)
    (8) edge node {} (25)
    (12) edge node {} (29)
    (16) edge node {} (33)
    (20) edge node {} (37)
    
    (21) edge node {} (22)
    (22) edge node {} (26)
    (22) edge node {} (23)
    (23) edge node {} (24)
    (23) edge node {} (27)
    (24) edge node {} (28)
    (21) edge node {} (25)
    (25) edge node {} (26)
    (25) edge node {} (29)
    (26) edge node {} (30)
    (26) edge node {} (27)
    (27) edge node {} (31)
    (27) edge node {} (28)
    (28) edge node {} (32)
    (30) edge node {} (31)
    (29) edge node {} (30)
    (31) edge node {} (32)
    (32) edge node {} (36)
    (33) edge node {} (29)
    (33) edge node {} (34)
    (34) edge node {} (35)
    (35) edge node {} (36)
    (30) edge node {} (34)
    (31) edge node {} (35)
    (32) edge node {} (36)
    (37) edge node {} (33)
    (38) edge node {} (34)
    (39) edge node {} (35)
    (40) edge node {} (36)
    (37) edge node {} (38)
    (38) edge node {} (39)
    (39) edge node {} (40)
    
    (41) edge node {} (42)
    (42) edge node {} (43)
    (43) edge node {} (44)
    (44) edge node {} (45)
    
    (24) edge node {} (41)
    (28) edge node {} (42)
    (32) edge node {} (43)
    (36) edge node {} (44)
    (40) edge node {} (45);

\end{tikzpicture}
}
\caption{}
\end{subfigure}
\begin{subfigure}[b]{0.3\linewidth}
          \centering

          \resizebox{\linewidth}{!}{

\tikzset{state/.style={circle,fill=blue!20,draw,minimum size=1.5em,inner sep=0pt},
            }
\begin{tikzpicture}

\filldraw[fill=blue!70!white, draw=black] (-0.4,0.4) rectangle (4.8,-3.7);
\filldraw[fill=red!50!white, draw=black, opacity=0.6] (-0.4,-1.75) rectangle (4.8,-6);

    \node[state] (1) at (0,0)  {1};
    \node[state] (2) [below=0.5cm of 1] {2};
    \node[state] (3) [below =0.5cm of 2] {3};
    \node[state] (4) [below =0.5cm of 3] {4};
    \node[state] (5) [right =0.5cm of 1] {5};
    \node[state] (6) [below=0.5cm of 5 ] {6};
    \node[state] (7) [below=0.5cm of 6] {7};
    \node[state] (8) [below =0.5cm of 7] {8};
    \node[state] (9) [right =0.5cm of 5] {9};
    \node[state] (10) [below =0.5cm of 9] {10};
    \node[state] (11) [below =0.5cm of 10] {10};
    \node[state] (12) [below =0.5cm of 11] {9};
    \node[state] (13) [right =0.5cm of 9] {8};
    \node[state] (14) [below =0.5cm of 13] {7};
    \node[state] (15) [below =0.5cm of 14] {6};
    \node[state] (16) [below =0.5cm of 15] {5};
    \node[state] (17) [right =0.5cm of 13] {4};
    \node[state] (18) [below =0.5cm of 17] {3};
    \node[state] (19) [below =0.5cm of 18] {2};
    \node[state] (20) [below =0.5cm of 19] {1};
    
    \node[state] (21) [below=0.5cm of 4] {1};
    \node[state] (22) [below=0.5cm of 21] {2};
    \node[state] (23) [below =0.5cm of 22] {};
    \node[state] (24) [below =0.5cm of 23] {};
    \node[state] (25) [right =0.5cm of 21] {5};
    \node[state] (26) [below=0.5cm of 25 ] {6};
    \node[state] (27) [below=0.5cm of 26] {};
    \node[state] (28) [below =0.5cm of 27] {};
    \node[state] (29) [right =0.5cm of 25] {9};
    \node[state] (30) [below =0.5cm of 29] {10};
    \node[state] (31) [below =0.5cm of 30] {};
    \node[state] (32) [below =0.5cm of 31] {};
    \node[state] (33) [right =0.5cm of 29] {8};
    \node[state] (34) [below =0.5cm of 33] {7};
    \node[state] (35) [below =0.5cm of 34] {};
    \node[state] (36) [below =0.5cm of 35] {};
    \node[state] (37) [right =0.5cm of 33] {4};
    \node[state] (38) [below =0.5cm of 37] {3};
    \node[state] (39) [below =0.5cm of 38] {};
    \node[state] (40) [below =0.5cm of 39] {};
    
    \node[state] (41) [below =0.5cm of 24] {};
    \node[state] (42) [below =0.5cm of 28] {};
    \node[state] (43) [below =0.5cm of 32] {};
    \node[state] (44) [below =0.5cm of 36] {};
    \node[state] (45) [below =0.5cm of 40] {};    
    
    \path[draw,thick]
    (1) edge node {} (2)
    (2) edge node {} (6)
    (2) edge node {} (3)
    (3) edge node {} (4)
    (3) edge node {} (7)
    (4) edge node {} (8)
    (1) edge node {} (5)
    (5) edge node {} (6)
    (5) edge node {} (9)
    (6) edge node {} (10)
    (6) edge node {} (7)
    (7) edge node {} (11)
    (7) edge node {} (8)
    (8) edge node {} (12)
    (10) edge node {} (11)
    (9) edge node {} (10)
    (11) edge node {} (12)
    (12) edge node {} (16)
    (13) edge node {} (9)
    (13) edge node {} (14)
    (14) edge node {} (15)
    (15) edge node {} (16)
    (10) edge node {} (14)
    (11) edge node {} (15)
    (12) edge node {} (16)
    (17) edge node {} (13)
    (18) edge node {} (14)
    (19) edge node {} (15)
    (20) edge node {} (16)
    (17) edge node {} (18)
    (18) edge node {} (19)
    (19) edge node {} (20)
    
    (4) edge node {} (21)
    (8) edge node {} (25)
    (12) edge node {} (29)
    (16) edge node {} (33)
    (20) edge node {} (37)
    
    (21) edge node {} (22)
    (22) edge node {} (26)
    (22) edge node {} (23)
    (23) edge node {} (24)
    (23) edge node {} (27)
    (24) edge node {} (28)
    (21) edge node {} (25)
    (25) edge node {} (26)
    (25) edge node {} (29)
    (26) edge node {} (30)
    (26) edge node {} (27)
    (27) edge node {} (31)
    (27) edge node {} (28)
    (28) edge node {} (32)
    (30) edge node {} (31)
    (29) edge node {} (30)
    (31) edge node {} (32)
    (32) edge node {} (36)
    (33) edge node {} (29)
    (33) edge node {} (34)
    (34) edge node {} (35)
    (35) edge node {} (36)
    (30) edge node {} (34)
    (31) edge node {} (35)
    (32) edge node {} (36)
    (37) edge node {} (33)
    (38) edge node {} (34)
    (39) edge node {} (35)
    (40) edge node {} (36)
    (37) edge node {} (38)
    (38) edge node {} (39)
    (39) edge node {} (40)
    
    (41) edge node {} (42)
    (42) edge node {} (43)
    (43) edge node {} (44)
    (44) edge node {} (45)
    
    (24) edge node {} (41)
    (28) edge node {} (42)
    (32) edge node {} (43)
    (36) edge node {} (44)
    (40) edge node {} (45);

\end{tikzpicture}
}
\caption{}
\end{subfigure}
\begin{subfigure}[b]{0.3\linewidth}
          \centering

          \resizebox{\linewidth}{!}{

\tikzset{state/.style={circle,fill=blue!20,draw,minimum size=1.5em,inner sep=0pt},
            }
\begin{tikzpicture}

\filldraw[fill=blue!70!white, draw=black] (-0.4,0.4) rectangle (4.8,-3.7);
\filldraw[fill=red!50!white, draw=black, opacity=0.6] (-0.4,-1.75) rectangle (4.8,-6);
\filldraw[fill=green!50!white, draw=black, opacity=0.6] (-0.4,-5) rectangle (4.8,-9.05);

    \node[state] (1) at (0,0)  {1};
    \node[state] (2) [below=0.5cm of 1] {2};
    \node[state] (3) [below =0.5cm of 2] {3};
    \node[state] (4) [below =0.5cm of 3] {4};
    \node[state] (5) [right =0.5cm of 1] {5};
    \node[state] (6) [below=0.5cm of 5 ] {6};
    \node[state] (7) [below=0.5cm of 6] {7};
    \node[state] (8) [below =0.5cm of 7] {8};
    \node[state] (9) [right =0.5cm of 5] {9};
    \node[state] (10) [below =0.5cm of 9] {10};
    \node[state] (11) [below =0.5cm of 10] {10};
    \node[state] (12) [below =0.5cm of 11] {9};
    \node[state] (13) [right =0.5cm of 9] {8};
    \node[state] (14) [below =0.5cm of 13] {7};
    \node[state] (15) [below =0.5cm of 14] {6};
    \node[state] (16) [below =0.5cm of 15] {5};
    \node[state] (17) [right =0.5cm of 13] {4};
    \node[state] (18) [below =0.5cm of 17] {3};
    \node[state] (19) [below =0.5cm of 18] {2};
    \node[state] (20) [below =0.5cm of 19] {1};
    
    \node[state] (21) [below=0.5cm of 4] {1};
    \node[state] (22) [below=0.5cm of 21] {2};
    \node[state] (23) [below =0.5cm of 22] {4};
    \node[state] (24) [below =0.5cm of 23] {1};
    \node[state] (25) [right =0.5cm of 21] {5};
    \node[state] (26) [below=0.5cm of 25 ] {6};
    \node[state] (27) [below=0.5cm of 26] {8};
    \node[state] (28) [below =0.5cm of 27] {5};
    \node[state] (29) [right =0.5cm of 25] {9};
    \node[state] (30) [below =0.5cm of 29] {10};
    \node[state] (31) [below =0.5cm of 30] {9};
    \node[state] (32) [below =0.5cm of 31] {9};
    \node[state] (33) [right =0.5cm of 29] {8};
    \node[state] (34) [below =0.5cm of 33] {7};
    \node[state] (35) [below =0.5cm of 34] {5};
    \node[state] (36) [below =0.5cm of 35] {8};
    \node[state] (37) [right =0.5cm of 33] {4};
    \node[state] (38) [below =0.5cm of 37] {3};
    \node[state] (39) [below =0.5cm of 38] {1};
    \node[state] (40) [below =0.5cm of 39] {4};
    
    \node[state] (41) [below =0.5cm of 24] {3};
    \node[state] (42) [below =0.5cm of 28] {7};
    \node[state] (43) [below =0.5cm of 32] {10};
    \node[state] (44) [below =0.5cm of 36] {6};
    \node[state] (45) [below =0.5cm of 40] {2};    
    
    \path[draw,thick]
    (1) edge node {} (2)
    (2) edge node {} (6)
    (2) edge node {} (3)
    (3) edge node {} (4)
    (3) edge node {} (7)
    (4) edge node {} (8)
    (1) edge node {} (5)
    (5) edge node {} (6)
    (5) edge node {} (9)
    (6) edge node {} (10)
    (6) edge node {} (7)
    (7) edge node {} (11)
    (7) edge node {} (8)
    (8) edge node {} (12)
    (10) edge node {} (11)
    (9) edge node {} (10)
    (11) edge node {} (12)
    (12) edge node {} (16)
    (13) edge node {} (9)
    (13) edge node {} (14)
    (14) edge node {} (15)
    (15) edge node {} (16)
    (10) edge node {} (14)
    (11) edge node {} (15)
    (12) edge node {} (16)
    (17) edge node {} (13)
    (18) edge node {} (14)
    (19) edge node {} (15)
    (20) edge node {} (16)
    (17) edge node {} (18)
    (18) edge node {} (19)
    (19) edge node {} (20)
    
    (4) edge node {} (21)
    (8) edge node {} (25)
    (12) edge node {} (29)
    (16) edge node {} (33)
    (20) edge node {} (37)
    
    (21) edge node {} (22)
    (22) edge node {} (26)
    (22) edge node {} (23)
    (23) edge node {} (24)
    (23) edge node {} (27)
    (24) edge node {} (28)
    (21) edge node {} (25)
    (25) edge node {} (26)
    (25) edge node {} (29)
    (26) edge node {} (30)
    (26) edge node {} (27)
    (27) edge node {} (31)
    (27) edge node {} (28)
    (28) edge node {} (32)
    (30) edge node {} (31)
    (29) edge node {} (30)
    (31) edge node {} (32)
    (32) edge node {} (36)
    (33) edge node {} (29)
    (33) edge node {} (34)
    (34) edge node {} (35)
    (35) edge node {} (36)
    (30) edge node {} (34)
    (31) edge node {} (35)
    (32) edge node {} (36)
    (37) edge node {} (33)
    (38) edge node {} (34)
    (39) edge node {} (35)
    (40) edge node {} (36)
    (37) edge node {} (38)
    (38) edge node {} (39)
    (39) edge node {} (40)
    
    (41) edge node {} (42)
    (42) edge node {} (43)
    (43) edge node {} (44)
    (44) edge node {} (45)
    
    (24) edge node {} (41)
    (28) edge node {} (42)
    (32) edge node {} (43)
    (36) edge node {} (44)
    (40) edge node {} (45);

\end{tikzpicture}
}
\caption{}
\end{subfigure}

\caption{}
\label{vtt}
\end{figure}

\subsection{Extending to general cases}

\begin{lemma}
For any positive integer $k$,
\[d_k\left(G_{r,l}\right) \geq \frac{(k+1)(k+2)}{2},\] where $r=n(k+1)+I$ and $l=(k+2)+J$ for positive integers $n,m$ and integers $0 \leq I < k+1$, $0 \leq J < k+2$ if at least one of the following is true:
\begin{itemize}
\item $n,m \geq 2$;
\item $n \geq 2$, $m \geq 1$, and either $J=0$ or $\left\lceil \frac{k+2}{2} \right\rceil \leq J < k+2$;
\item $m \geq 2$, $n \geq 1$, and either $I=0$ or  $\left\lceil \frac{k+1}{2} \right\rceil \leq I < k+1$.
\end{itemize}
\end{lemma}

\begin{proof}
Let $r=n(k+1)+I$ and $l=(k+2)+J$ which satisfy one or more of the conditions given. We will consider how to color the grid graph $G_{r,l}$ in cases depending on the value of $I$. Each of cases will then break up into subcases depending on the value of $J$.

\textbf{Case 1: $\mathbf{I=0}$.} If $J=0$, then we can simply repeat the coloring $\varphi$ $nm$ times over the entire grid, by breaking it up into disjoint copies of $G_{k+1,k+2}$. If $\left\lceil \frac{k+2}{2} \right\rceil \leq J < k+2$, then we can break up $G_{r,l}$ into $n(m-1)$ disjoint copies of $G_{k+1,k+2}$ (the first $(m-1)(k+2)$ columns) which can each be colored with $\varphi$ as well as $n$ copies of $G_{k+1,k+2+J}$ (the final $k+2+J$ columns) which can each be colored according to Lemma~\ref{c1}. Similarly, if $1 \leq J < \left\lceil \frac{k+2}{2} \right\rceil$, then $m \geq 2$, and $G_{r,l}$ can be broken up into $n(m-2)$ disjoint copies of $G_{k+1,k+2}$ (the first $(m-2)(k+2)$ columns) and $n$ copies of $G_{k+1,2(k+2)+J}$ (the final $2(k+2)+J$ columns). Again, the copies of $G_{k+1,k+2}$ can each be colored with $\varphi$ and each copy of $G_{k+1,2(k+2)+J}$ can be colored according to Lemma~\ref{c2}.

\textbf{Case 2: $\mathbf{\left\lceil \frac{k+1}{2} \right\rceil \leq I < k+1}$}. If $J=0$, then we can consider $G_{r,l}$ as $(n-1)m$ disjoint copies of $G_{k+1,k+2}$ (the first $(n-1)(k+1)$ rows) and $m$ disjoint copies of $G_{k+1+I,k+2}$ (the final $k+1+I$ rows). The copies of $G_{k+1,k+2}$ can each be colored with $\varphi$ and each copy of $G_{k+1+I,k+2}$ can be colored according to Lemma~\ref{r1}.

If $\left\lceil \frac{k+2}{2} \right\rceil \leq J < k+2$, then $G_{r,l}$ breaks up into a copy of $G_{r,(m-1)(k+2)}$ (the first $(m-1)(k+2)$ columns) and a copy of $G_{r,k+2+J}$ (the final $k+2+J$ columns). Color the copy of $G_{r,(m-1)(k+2)}$ as above. The copy of $G_{r,k+2+J}$ breaks up further into $n-1$ disjoint copies of $G_{k+1,k+2+J}$ (the first $(n-1)(k+1)$ rows) and one copy of $G_{k+1+I,k+2+J}$ (the final $k+1+I$ rows). Each copy of $G_{k+1,k+2+J}$ can be colored as in Lemma~\ref{c1}.

The copy of $G_{k+1+I,k+2+J}$ can be colored by first coloring the initial $G_{k+1,k+2}$ block (the first $k+1$ rows and the first $k+2$ columns) by the standard coloring $\varphi$. We can then extend this coloring in two different directions - ``horizontally" and ``vertically." Extend the coloring $\varphi$ to the remaining vertices of the first $k+1$ rows as in Lemma~\ref{c1} with the coloring $h_s(\varphi)$ on the final $k+2$ columns for $s=k+2-J$. Similarly, extend $\varphi$ to the remaining vertices of the first $k+2$ columns as in Lemma~\ref{r1} with the coloring $v_t(\varphi)$ on the final $k+1$ rows where $t=k+1-I$. Finally, we may either extend the coloring $h_s(\varphi)$ to the final vertices with the coloring $v_t(h_s(\varphi))$ or alternatively, extend the coloring $v_t(\varphi)$ to $h_s(v_t(\varphi))$. By Lemma~\ref{commute}, the extensions will agree with one another. This case is illustrated in Figure~\ref{hsvt}.

\begin{figure}
\begin{subfigure}[b]{0.48\linewidth}
          \centering

          \resizebox{\linewidth}{!}{

\tikzset{state/.style={circle,fill=blue!20,draw,minimum size=1.5em,inner sep=0pt},
            }
\begin{tikzpicture}

\filldraw[fill=blue!70!white, draw=black] (-0.4,0.4) rectangle (4.8,-3.7);
\filldraw[fill=green!50!white, draw=black, opacity=0.6] (-0.4,-1.8) rectangle (4.8,-5.8);
\filldraw[fill=black!50!white, draw=black, opacity=0.6] (3.8,-1.8) rectangle (9.1,-5.8);

    \node[state] (1) {1};
    \node[state] (2) [below=.5 of 1] {2};
    \node[state] (3) [below =.5 of 2] {3};
    \node[state] (4) [below =.5 of 3] {4};
    \node[state] (5) [right =.5 of 1] {5};
    \node[state] (6) [below=.5 of 5 ] {6};
    \node[state] (7) [below=.5 of 6] {7};
    \node[state] (8) [below =.5 of 7] {8};
    \node[state] (9) [right =.5 of 5] {9};
    \node[state] (10) [below =.5 of 9] {10};
    \node[state] (11) [below =.5 of 10] {10};
    \node[state] (12) [below =.5 of 11] {9};
    \node[state] (13) [right =.5 of 9] {8};
    \node[state] (14) [below =.5 of 13] {7};
    \node[state] (15) [below =.5 of 14] {6};
    \node[state] (16) [below =.5 of 15] {5};
    \node[state] (17) [right =.5 of 13] {4};
    \node[state] (18) [below =.5 of 17] {3};
    \node[state] (19) [below =.5 of 18] {2};
    \node[state] (20) [below =.5 of 19] {1};
    
    \node[state] (21) [below=.5 of 4] {1};
    \node[state] (22) [below=.5 of 21] {2};
  
    \node[state] (25) [right =.5 of 21] {5};
    \node[state] (26) [below=.5 of 25 ] {6};
 
    \node[state] (29) [right =.5 of 25] {9};
    \node[state] (30) [below =.5 of 29] {10};
   
    \node[state] (33) [right =.5 of 29] {8};
    \node[state] (34) [below =.5 of 33] {7};
  
    \node[state] (37) [right =.5 of 33] {4};
    \node[state] (38) [below =.5 of 37] {3};
    
    \node[state] (41) [right=.5 of 17] {};
    \node[state] (42) [below=.5 of 41] {};
    \node[state] (43) [below =.5 of 42] {7};
    \node[state] (44) [below =.5 of 43] {8};
    \node[state] (45) [right =.5 of 41] {};
    \node[state] (46) [below=.5 of 45 ] {};
    \node[state] (47) [below=.5 of 46] {10};
    \node[state] (48) [below =.5 of 47] {9};
    \node[state] (49) [right =.5 of 45] {};
    \node[state] (50) [below =.5 of 49] {};
    \node[state] (51) [below =.5 of 50] {6};
    \node[state] (52) [below =.5 of 51] {5};
    \node[state] (53) [right =.5 of 49] {};
    \node[state] (54) [below =.5 of 53] {};
    \node[state] (55) [below =.5 of 54] {3};
    \node[state] (56) [below =.5 of 55] {4};
    
    \node[state] (61) [right=.5 of 37] {5};
    \node[state] (62) [below=.5 of 61] {6};

    \node[state] (65) [right =.5 of 61] {9};
    \node[state] (66) [below=.5 of 65 ] {10};

    \node[state] (69) [right =.5 of 65] {8};
    \node[state] (70) [below =.5 of 69] {7};

    \node[state] (73) [right =.5 of 69] {1};
    \node[state] (74) [below =.5 of 73] {2};

    \path[draw,thick]
    (1) edge node {} (2)
    (2) edge node {} (6)
    (2) edge node {} (3)
    (3) edge node {} (4)
    (3) edge node {} (7)
    (4) edge node {} (8)
    (1) edge node {} (5)
    (5) edge node {} (6)
    (5) edge node {} (9)
    (6) edge node {} (10)
    (6) edge node {} (7)
    (7) edge node {} (11)
    (7) edge node {} (8)
    (8) edge node {} (12)
    (10) edge node {} (11)
    (9) edge node {} (10)
    (11) edge node {} (12)
    (12) edge node {} (16)
    (13) edge node {} (9)
    (13) edge node {} (14)
    (14) edge node {} (15)
    (15) edge node {} (16)
    (10) edge node {} (14)
    (11) edge node {} (15)
    (12) edge node {} (16)
    (17) edge node {} (13)
    (18) edge node {} (14)
    (19) edge node {} (15)
    (20) edge node {} (16)
    (17) edge node {} (18)
    (18) edge node {} (19)
    (19) edge node {} (20)
    
    (4) edge node {} (21)
    (8) edge node {} (25)
    (12) edge node {} (29)
    (16) edge node {} (33)
    (20) edge node {} (37)
    
    (17) edge node {} (41)
    (18) edge node {} (42)
    (19) edge node {} (43)
    (20) edge node {} (44)
    
    (21) edge node {} (22)
    (22) edge node {} (26)
    (21) edge node {} (25)
    (25) edge node {} (26)
    (25) edge node {} (29)
    (26) edge node {} (30)
    (29) edge node {} (30)
    (33) edge node {} (29)
    (33) edge node {} (34)
    (30) edge node {} (34)
    (37) edge node {} (33)
    (38) edge node {} (34)
    (37) edge node {} (38)

    (37) edge node {} (61)
    (38) edge node {} (62)
 
    (41) edge node {} (42)
    (42) edge node {} (46)
    (42) edge node {} (43)
    (43) edge node {} (44)
    (43) edge node {} (47)
    (44) edge node {} (48)
    (41) edge node {} (45)
    (45) edge node {} (46)
    (45) edge node {} (49)
    (46) edge node {} (50)
    (46) edge node {} (47)
    (47) edge node {} (51)
    (47) edge node {} (48)
    (48) edge node {} (52)
    (50) edge node {} (51)
    (49) edge node {} (50)
    (51) edge node {} (52)
    (52) edge node {} (56)
    (53) edge node {} (49)
    (53) edge node {} (54)
    (54) edge node {} (55)
    (55) edge node {} (56)
    (50) edge node {} (54)
    (51) edge node {} (55)
    (52) edge node {} (56)
    
    (44) edge node {} (61)
    (48) edge node {} (65)
    (52) edge node {} (69)
    (56) edge node {} (73)
    
    (61) edge node {} (62)
    (62) edge node {} (66)
    (61) edge node {} (65)
    (65) edge node {} (66)
    (65) edge node {} (69)
    (66) edge node {} (70)
    (69) edge node {} (70)
    (73) edge node {} (69)
    (73) edge node {} (74)
    (70) edge node {} (74);
\end{tikzpicture}
}
\caption{$(h_1 \circ v_2)(\varphi)$}
\end{subfigure}
\begin{subfigure}[b]{0.48\linewidth}
          \centering

          \resizebox{\linewidth}{!}{

\tikzset{state/.style={circle,fill=blue!20,draw,minimum size=1.5em,inner sep=0pt},
            }
\begin{tikzpicture}

\filldraw[fill=blue!70!white, draw=black] (-0.4,0.4) rectangle (4.8,-3.7);
\filldraw[fill=red!50!white, draw=black, opacity=0.6] (3.8,0.4) rectangle (9.1,-3.7);
\filldraw[fill=black!50!white, draw=black, opacity=0.6] (3.8,-1.8) rectangle (9.1,-5.8);

    \node[state] (1) {1};
    \node[state] (2) [below=.5 of 1] {2};
    \node[state] (3) [below =.5 of 2] {3};
    \node[state] (4) [below =.5 of 3] {4};
    \node[state] (5) [right =.5 of 1] {5};
    \node[state] (6) [below=.5 of 5 ] {6};
    \node[state] (7) [below=.5 of 6] {7};
    \node[state] (8) [below =.5 of 7] {8};
    \node[state] (9) [right =.5 of 5] {9};
    \node[state] (10) [below =.5 of 9] {10};
    \node[state] (11) [below =.5 of 10] {10};
    \node[state] (12) [below =.5 of 11] {9};
    \node[state] (13) [right =.5 of 9] {8};
    \node[state] (14) [below =.5 of 13] {7};
    \node[state] (15) [below =.5 of 14] {6};
    \node[state] (16) [below =.5 of 15] {5};
    \node[state] (17) [right =.5 of 13] {4};
    \node[state] (18) [below =.5 of 17] {3};
    \node[state] (19) [below =.5 of 18] {2};
    \node[state] (20) [below =.5 of 19] {1};
    
    \node[state] (21) [below=.5 of 4] {};
    \node[state] (22) [below=.5 of 21] {};
  
    \node[state] (25) [right =.5 of 21] {};
    \node[state] (26) [below=.5 of 25 ] {};
 
    \node[state] (29) [right =.5 of 25] {};
    \node[state] (30) [below =.5 of 29] {};
   
    \node[state] (33) [right =.5 of 29] {};
    \node[state] (34) [below =.5 of 33] {};
  
    \node[state] (37) [right =.5 of 33] {4};
    \node[state] (38) [below =.5 of 37] {3};
    
    \node[state] (41) [right=.5 of 17] {5};
    \node[state] (42) [below=.5 of 41] {6};
    \node[state] (43) [below =.5 of 42] {7};
    \node[state] (44) [below =.5 of 43] {8};
    \node[state] (45) [right =.5 of 41] {9};
    \node[state] (46) [below=.5 of 45 ] {10};
    \node[state] (47) [below=.5 of 46] {10};
    \node[state] (48) [below =.5 of 47] {9};
    \node[state] (49) [right =.5 of 45] {8};
    \node[state] (50) [below =.5 of 49] {7};
    \node[state] (51) [below =.5 of 50] {6};
    \node[state] (52) [below =.5 of 51] {5};
    \node[state] (53) [right =.5 of 49] {1};
    \node[state] (54) [below =.5 of 53] {2};
    \node[state] (55) [below =.5 of 54] {3};
    \node[state] (56) [below =.5 of 55] {4};
    
    \node[state] (61) [right=.5 of 37] {5};
    \node[state] (62) [below=.5 of 61] {6};

    \node[state] (65) [right =.5 of 61] {9};
    \node[state] (66) [below=.5 of 65 ] {10};

    \node[state] (69) [right =.5 of 65] {8};
    \node[state] (70) [below =.5 of 69] {7};

    \node[state] (73) [right =.5 of 69] {1};
    \node[state] (74) [below =.5 of 73] {2};

    \path[draw,thick]
    (1) edge node {} (2)
    (2) edge node {} (6)
    (2) edge node {} (3)
    (3) edge node {} (4)
    (3) edge node {} (7)
    (4) edge node {} (8)
    (1) edge node {} (5)
    (5) edge node {} (6)
    (5) edge node {} (9)
    (6) edge node {} (10)
    (6) edge node {} (7)
    (7) edge node {} (11)
    (7) edge node {} (8)
    (8) edge node {} (12)
    (10) edge node {} (11)
    (9) edge node {} (10)
    (11) edge node {} (12)
    (12) edge node {} (16)
    (13) edge node {} (9)
    (13) edge node {} (14)
    (14) edge node {} (15)
    (15) edge node {} (16)
    (10) edge node {} (14)
    (11) edge node {} (15)
    (12) edge node {} (16)
    (17) edge node {} (13)
    (18) edge node {} (14)
    (19) edge node {} (15)
    (20) edge node {} (16)
    (17) edge node {} (18)
    (18) edge node {} (19)
    (19) edge node {} (20)
    
    (4) edge node {} (21)
    (8) edge node {} (25)
    (12) edge node {} (29)
    (16) edge node {} (33)
    (20) edge node {} (37)
    
    (17) edge node {} (41)
    (18) edge node {} (42)
    (19) edge node {} (43)
    (20) edge node {} (44)
    
    (21) edge node {} (22)
    (22) edge node {} (26)
    (21) edge node {} (25)
    (25) edge node {} (26)
    (25) edge node {} (29)
    (26) edge node {} (30)
    (29) edge node {} (30)
    (33) edge node {} (29)
    (33) edge node {} (34)
    (30) edge node {} (34)
    (37) edge node {} (33)
    (38) edge node {} (34)
    (37) edge node {} (38)

    (37) edge node {} (61)
    (38) edge node {} (62)
 
    (41) edge node {} (42)
    (42) edge node {} (46)
    (42) edge node {} (43)
    (43) edge node {} (44)
    (43) edge node {} (47)
    (44) edge node {} (48)
    (41) edge node {} (45)
    (45) edge node {} (46)
    (45) edge node {} (49)
    (46) edge node {} (50)
    (46) edge node {} (47)
    (47) edge node {} (51)
    (47) edge node {} (48)
    (48) edge node {} (52)
    (50) edge node {} (51)
    (49) edge node {} (50)
    (51) edge node {} (52)
    (52) edge node {} (56)
    (53) edge node {} (49)
    (53) edge node {} (54)
    (54) edge node {} (55)
    (55) edge node {} (56)
    (50) edge node {} (54)
    (51) edge node {} (55)
    (52) edge node {} (56)
    
    (44) edge node {} (61)
    (48) edge node {} (65)
    (52) edge node {} (69)
    (56) edge node {} (73)
    
    (61) edge node {} (62)
    (62) edge node {} (66)
    (61) edge node {} (65)
    (65) edge node {} (66)
    (65) edge node {} (69)
    (66) edge node {} (70)
    (69) edge node {} (70)
    (73) edge node {} (69)
    (73) edge node {} (74)
    (70) edge node {} (74);
\end{tikzpicture}
}
\caption{$(v_2 \circ h_1)(\varphi)$}
\end{subfigure}

\begin{subfigure}[b]{0.48\linewidth}
          \centering

          \resizebox{\linewidth}{!}{

\tikzset{state/.style={circle,fill=blue!20,draw,minimum size=1.5em,inner sep=0pt},
            }
\begin{tikzpicture}

\filldraw[fill=blue!70!white, draw=black] (-0.4,0.4) rectangle (4.8,-3.7);
\filldraw[fill=red!50!white, draw=black, opacity=0.6] (3.8,0.4) rectangle (9.1,-3.7);
\filldraw[fill=green!50!white, draw=black, opacity=0.6] (-0.4,-1.8) rectangle (4.8,-5.8);
\filldraw[fill=black!50!white, draw=black, opacity=0.6] (3.8,-1.8) rectangle (9.1,-5.8);

    \node[state] (1) {1};
    \node[state] (2) [below=.5 of 1] {2};
    \node[state] (3) [below =.5 of 2] {3};
    \node[state] (4) [below =.5 of 3] {4};
    \node[state] (5) [right =.5 of 1] {5};
    \node[state] (6) [below=.5 of 5 ] {6};
    \node[state] (7) [below=.5 of 6] {7};
    \node[state] (8) [below =.5 of 7] {8};
    \node[state] (9) [right =.5 of 5] {9};
    \node[state] (10) [below =.5 of 9] {10};
    \node[state] (11) [below =.5 of 10] {10};
    \node[state] (12) [below =.5 of 11] {9};
    \node[state] (13) [right =.5 of 9] {8};
    \node[state] (14) [below =.5 of 13] {7};
    \node[state] (15) [below =.5 of 14] {6};
    \node[state] (16) [below =.5 of 15] {5};
    \node[state] (17) [right =.5 of 13] {4};
    \node[state] (18) [below =.5 of 17] {3};
    \node[state] (19) [below =.5 of 18] {2};
    \node[state] (20) [below =.5 of 19] {1};
    
    \node[state] (21) [below=.5 of 4] {1};
    \node[state] (22) [below=.5 of 21] {2};
  
    \node[state] (25) [right =.5 of 21] {5};
    \node[state] (26) [below=.5 of 25 ] {6};
 
    \node[state] (29) [right =.5 of 25] {9};
    \node[state] (30) [below =.5 of 29] {10};
   
    \node[state] (33) [right =.5 of 29] {8};
    \node[state] (34) [below =.5 of 33] {7};
  
    \node[state] (37) [right =.5 of 33] {4};
    \node[state] (38) [below =.5 of 37] {3};
    
    \node[state] (41) [right=.5 of 17] {5};
    \node[state] (42) [below=.5 of 41] {6};
    \node[state] (43) [below =.5 of 42] {7};
    \node[state] (44) [below =.5 of 43] {8};
    \node[state] (45) [right =.5 of 41] {9};
    \node[state] (46) [below=.5 of 45 ] {10};
    \node[state] (47) [below=.5 of 46] {10};
    \node[state] (48) [below =.5 of 47] {9};
    \node[state] (49) [right =.5 of 45] {8};
    \node[state] (50) [below =.5 of 49] {7};
    \node[state] (51) [below =.5 of 50] {6};
    \node[state] (52) [below =.5 of 51] {5};
    \node[state] (53) [right =.5 of 49] {1};
    \node[state] (54) [below =.5 of 53] {2};
    \node[state] (55) [below =.5 of 54] {3};
    \node[state] (56) [below =.5 of 55] {4};
    
    \node[state] (61) [right=.5 of 37] {5};
    \node[state] (62) [below=.5 of 61] {6};

    \node[state] (65) [right =.5 of 61] {9};
    \node[state] (66) [below=.5 of 65 ] {10};

    \node[state] (69) [right =.5 of 65] {8};
    \node[state] (70) [below =.5 of 69] {7};

    \node[state] (73) [right =.5 of 69] {1};
    \node[state] (74) [below =.5 of 73] {2};

    \path[draw,thick]
    (1) edge node {} (2)
    (2) edge node {} (6)
    (2) edge node {} (3)
    (3) edge node {} (4)
    (3) edge node {} (7)
    (4) edge node {} (8)
    (1) edge node {} (5)
    (5) edge node {} (6)
    (5) edge node {} (9)
    (6) edge node {} (10)
    (6) edge node {} (7)
    (7) edge node {} (11)
    (7) edge node {} (8)
    (8) edge node {} (12)
    (10) edge node {} (11)
    (9) edge node {} (10)
    (11) edge node {} (12)
    (12) edge node {} (16)
    (13) edge node {} (9)
    (13) edge node {} (14)
    (14) edge node {} (15)
    (15) edge node {} (16)
    (10) edge node {} (14)
    (11) edge node {} (15)
    (12) edge node {} (16)
    (17) edge node {} (13)
    (18) edge node {} (14)
    (19) edge node {} (15)
    (20) edge node {} (16)
    (17) edge node {} (18)
    (18) edge node {} (19)
    (19) edge node {} (20)
    
    (4) edge node {} (21)
    (8) edge node {} (25)
    (12) edge node {} (29)
    (16) edge node {} (33)
    (20) edge node {} (37)
    
    (17) edge node {} (41)
    (18) edge node {} (42)
    (19) edge node {} (43)
    (20) edge node {} (44)
    
    (21) edge node {} (22)
    (22) edge node {} (26)
    (21) edge node {} (25)
    (25) edge node {} (26)
    (25) edge node {} (29)
    (26) edge node {} (30)
    (29) edge node {} (30)
    (33) edge node {} (29)
    (33) edge node {} (34)
    (30) edge node {} (34)
    (37) edge node {} (33)
    (38) edge node {} (34)
    (37) edge node {} (38)

    (37) edge node {} (61)
    (38) edge node {} (62)
 
    (41) edge node {} (42)
    (42) edge node {} (46)
    (42) edge node {} (43)
    (43) edge node {} (44)
    (43) edge node {} (47)
    (44) edge node {} (48)
    (41) edge node {} (45)
    (45) edge node {} (46)
    (45) edge node {} (49)
    (46) edge node {} (50)
    (46) edge node {} (47)
    (47) edge node {} (51)
    (47) edge node {} (48)
    (48) edge node {} (52)
    (50) edge node {} (51)
    (49) edge node {} (50)
    (51) edge node {} (52)
    (52) edge node {} (56)
    (53) edge node {} (49)
    (53) edge node {} (54)
    (54) edge node {} (55)
    (55) edge node {} (56)
    (50) edge node {} (54)
    (51) edge node {} (55)
    (52) edge node {} (56)
    
    (44) edge node {} (61)
    (48) edge node {} (65)
    (52) edge node {} (69)
    (56) edge node {} (73)
    
    (61) edge node {} (62)
    (62) edge node {} (66)
    (61) edge node {} (65)
    (65) edge node {} (66)
    (65) edge node {} (69)
    (66) edge node {} (70)
    (69) edge node {} (70)
    (73) edge node {} (69)
    (73) edge node {} (74)
    (70) edge node {} (74);
\end{tikzpicture}
}
\caption{}
\end{subfigure}
\caption{}
\label{hsvt}
\end{figure}

If  $1 \leq J < \left\lceil \frac{k+2}{2} \right\rceil$, then $m \geq 2$ and $G_{r,l}$ breaks up into a copy of $G_{r,(m-2)(k+2)}$ (the first $(m-2)(k+2)$ columns) and a copy of $G_{r,2(k+2)+J}$ (the final $2(k+2)+J$ columns). Color the copy of $G_{r,(m-2)(k+2)}$ as above. The copy of $G_{r,2(k+2)+J}$ breaks up further into $n-1$ disjoint copies of $G_{k+1,2(k+2)+J}$ (the first $(n-1)(k+1)$ rows) and one copy of $G_{k+1+I,2(k+2)+J}$ (the final $k+1+I$ rows). Each copy of $G_{k+1,2(k+2)+J}$ can be colored as in Lemma~\ref{c2}. The copy of $G_{k+1+I,2(k+2)+J}$ can be colored by first coloring the initial $G_{k+1,k+2}$ block (the first $k+1$ rows and the first $k+2$ columns) by the standard coloring $\varphi$. As in Lemma~\ref{c2}, we then extend this coloring to the remaining vertices of the first $k+1$ rows with the colorings $h_{s_1}(\varphi)$ and $h_{s_2}(h_{s_1}(\varphi))$. We also extend down to the final $I$ rows of the first $k+2$ columns with the coloring $v_t(\varphi)$ as in Lemma~\ref{r1} where $t = k+1-I$. Now we can complete the coloring by extending $v_t(\varphi)$ horizontally with $h_{s_1}(v_t(\varphi))$ and then $h_{s_2}(h_{s_1}(v_t(\varphi)))$. By Lemma~\ref{commute} we know that these extensions agree on all vertices where they overlap with the colorings $h_{s_1}(\varphi)$ and $h_{s_2}(h_{s_1}(\varphi))$.

\textbf{Case 3: $\mathbf{1 \leq I < \left\lceil \frac{k+1}{2} \right\rceil}$}. Now we may assume that $n \geq 2$. If $J=0$, then we can consider $G_{r,l}$ as $(n-2)m$ disjoint copies of $G_{k+1,k+2}$ (the first $(n-2)(k+1)$ rows) and $m$ disjoint copies of $G_{2(k+1)+I,k+2}$ (the final $k+1+I$ rows). The copies of $G_{k+1,k+2}$ can each be colored with $\varphi$ and each copy of $G_{2(k+1)+I,k+2}$ can be colored according to Lemma~\ref{r2}.

If $\left\lceil \frac{k+2}{2} \right\rceil \leq J < k+2$, then $G_{r,l}$ breaks up into a copy of $G_{r,(m-1)(k+2)}$ (the first $(m-1)(k+2)$ columns) and a copy of $G_{r,k+2+J}$ (the final $k+2+J$ columns). Color the copy of $G_{r,(m-1)(k+2)}$ as above. The copy of $G_{r,k+2+J}$ breaks up further into $n-2$ disjoint copies of $G_{k+1,k+2+J}$ (the first $(n-1)(k+1)$ rows) and one copy of $G_{2(k+1)+I,k+2+J}$ (the final $2(k+1)+I$ rows). Each copy of $G_{k+1,k+2+J}$ can be colored as in Lemma~\ref{c1}.

The copy of $G_{2(k+1)+I,k+2+J}$ can be colored by first coloring the initial $G_{k+1,k+2}$ block (the first $k+1$ rows and the first $k+2$ columns) by the standard coloring $\varphi$. As before, we can then extend this coloring in two different directions. Extend the coloring $\varphi$ to the remaining vertices of the first $k+1$ rows as in Lemma~\ref{c1} with the coloring $h_s(\varphi)$ on the final $k+2$ columns for $s=k+2-J$. Similarly, extend $\varphi$ in the first $k+2$ columns to the rows $k+2$ through $k+1+\left\lceil \frac{k+1}{2} \right\rceil$ with the coloring $v_{t_1}(\varphi)$ where $t_1=\left\lfloor \frac{k+1}{2} \right\rfloor$. Then extend this coloring to the final $\left\lfloor \frac{k+1}{2} \right\rfloor + I$ rows in the first $k+2$ columns by $v_{t_2}(v_{t_1}(\varphi))$ where \[t_2 = \left\lceil \frac{k+2}{2} \right\rceil -I,\] which is valid because of the conditions on $I$. Next, we extend the current coloring to the final vertices with $v_{t_1}(h_s(\varphi))$ and $v_{t_2}(v_{t_1}(h_s(\varphi)))$. Alternatively, we could finish the coloring with the extensions $h_s(v_{t_1}(\varphi))$ and $v_{t_2}(h_s(v_{t_1}(\varphi)))$ or $h_s(v_{t_1}(\varphi))$ as well as $h_s(v_{t_2}(v_{t_1}(\varphi)))$. By Lemma~\ref{commute}, these extensions will all agree with one another.

Finally, if  $1 \leq J < \left\lceil \frac{k+2}{2} \right\rceil$, then $m \geq 2$ and $G_{r,l}$ breaks up into a copy of $G_{r,(m-2)(k+2)}$ (the first $(m-2)(k+2)$ columns) and a copy of $G_{r,2(k+2)+J}$ (the final $2(k+2)+J$ columns). Color the copy of $G_{r,(m-2)(k+2)}$ as above. The copy of $G_{r,2(k+2)+J}$ breaks up further into $n-2$ disjoint copies of $G_{k+1,2(k+2)+J}$ (the first $(n-2)(k+1)$ rows) and one copy of $G_{2(k+1)+I,2(k+2)+J}$ (the final $2(k+1)+I$ rows). Each copy of $G_{k+1,2(k+2)+J}$ can be colored as in Lemma~\ref{c2}. The copy of $G_{2(k+1)+I,2(k+2)+J}$ can be colored by first coloring the initial $G_{k+1,k+2}$ block (the first $k+1$ rows and the first $k+2$ columns) by the standard coloring $\varphi$. As before, we can then extend this coloring in two different directions. Extend the coloring $\varphi$ to the remaining vertices of the first $k+1$ rows as in Lemma~\ref{c2} with the colorings $h_{s_1}(\varphi)$ and then $h_{s_2}(h_{s_1}(\varphi))$. Also, extend $\varphi$ to the remaining vertices in the first $k+2$ columns as in Lemma~\ref{r2} with the colorings $v_{t_1}(\varphi)$ and then $v_{t_2}(v_{t_1}(\varphi))$. We may then complete the coloring to the other vertices of the copy of $G_{2(k+1)+I,2(k+2)+J}$ with the colorings $v_{t_1}(h_{s_1}(\varphi))$, $v_{t_2}(v_{t_1}(h_{s_1}(\varphi)))$, $v_{t_1}(h_{s_2}(h_{s_1}(\varphi)))$, and $v_{t_2}(v_{t_1}(h_{s_2}(h_{s_1}(\varphi))))$. Again, by Lemma~\ref{commute}, these extensions agree with all other possible extensions so the coloring is valid.
\end{proof}

Therefore, we get the following result as a corollary.

\begin{lemma}
\label{mainthm}
Let $k$ be a positive integer, then \[d_k\left(G_{r,l}\right) \geq \frac{(k+1)(k+2)}{2}\] whenever
\begin{itemize}
\item $r=k+1$ and $l=k+2$;
\item $r=k+1$ and $l\geq k+2 +\left\lceil \frac{k+2}{2} \right\rceil$;
\item $r\geq k+1 +\left\lceil \frac{k+1}{2} \right\rceil$ and $l=k+2$;
\item or $r\geq k+1 +\left\lceil \frac{k+1}{2} \right\rceil$ and $l\geq k+2 +\left\lceil \frac{k+2}{2} \right\rceil$.
\end{itemize}
\end{lemma}

\section{Cases when $r \leq k$}
\label{smallsection}

\begin{lemma}
\label{small}
For each $1 \leq r \leq k$, all but finitely many grid graphs $G_{r,l}$ are domatically-full. That is, \[d_k(G_{r,l}) \geq r(k+1)-\frac{r(r-1)}{2}\] for $l=2k-r+3$ and all $l \geq 2k-r+3 + \left\lceil \frac{2k-r+3}{2} \right\rceil$.
\end{lemma}

\begin{proof}
By Lemma~\ref{blockproof} we know that this is already true when $l=2k-r+3$ by the standard block coloring $\varphi_r$. We can extend this coloring to cover all cases when $l \geq 2k-r+3 + \left\lceil \frac{2k-r+3}{2} \right\rceil$ in exactly the same way as we extended the coloring $\varphi$ ``horizontally" in Lemmas~\ref{c1} and~\ref{c2}. That is, let $l = m(2k-r+3)+J$ for some integers $m \geq 1$ and $0 \leq J \leq 2k-r+2$. If $J=0$, then we simply use the $\varphi_r$ coloring $m$ times. If $\left\lceil \frac{2k-r+3}{2} \right\rceil \leq J \leq 2k-r+2$, then we use the $\varphi_r$ coloring $m$ times on the first $m(2k-r+2)$ columns, then noting that the final $2k-r+2$ columns are already partially colored in the first $s=2k-r+2-J$ columns. Since $\left\lceil \frac{2k-r+3}{2} \right\rceil \leq J$, then $s \leq \left\lfloor \frac{2k-r+3}{2} \right\rfloor$. Therefore, none of the colors have been repeated, and so this partial coloring can be extended to a coloring isomorphic to $\varphi_r$.

Similarly, if $1 \leq J \leq \left\lceil \frac{2k-r+3}{2} \right\rceil -1$ and if $m \geq 2$, then we can use $\varphi_r$ $m-1$ times on the first $(m-1)(2k-r+3)$ columns, then extend the coloring given to columns $(m-1)(2k-r+3)-\left\lfloor \frac{2k-r+3}{2} \right\rfloor + 1$ through $(m-1)(2k-r+3)$ to a coloring on columns $(m-1)(2k-r+3)+1$ through $(m-1)(2k-r+3) + \left\lceil \frac{2k-r+3}{2} \right\rceil$ as above ($h_{s_1}(\varphi_r)$ for $s_1=\left\lfloor \frac{2k-r+3}{2} \right\rfloor$). Then extend the coloring again to the final $J+\left\lfloor \frac{2k-r+3}{2} \right\rfloor$ columns in the same manner as above since $ \left\lceil \frac{2k-r+3}{2} \right\rceil  \leq J+\left\lfloor \frac{2k-r+3}{2} \right\rfloor 2k-r+2$.
\end{proof}

\section{All cases when $k=3$}
\label{k=3}

Between Chang \cite{chang94} and Kiser \cite{masters} we know the $k$-domatic numbers for every two-dimensional grid $G_{r,l}$ when $k=1,2$. In this section we determine $d_3(G_{r,l})$ for every two-dimensional grid. First, if $r=1$, then we know that $d_3(G_{1,l}) = \min{\{4,l\}}$ is trivial.

\begin{figure}
\begin{subfigure}[b]{0.33\linewidth}
          \centering

          \resizebox{\linewidth}{!}{
\usetikzlibrary{positioning}
\tikzset{main node/.style={circle,fill=blue!20,draw,minimum size=1.5em,inner sep=0pt},
            }
\begin{tikzpicture}
	\node[main node] (1) {1};
	\node[main node] (2) [below= 1 cm of 1] {2};
	\node[main node] (3) [right= 1 cm of 1] {3};
	\node[main node] (4) [below= 1 cm of 3] {4};
	\node[main node] (5) [right = 1 cm of 3] {5};
	\node[main node] (6) [below= 1 cm of 5] {6};
	\node[main node] (7) [right = 1 cm of 5] {1};
	\node[main node] (8) [below= 1 cm of 7] {2};
	\path[draw,thick]
    	(1) edge node {} (2)
	(1) edge node {} (3)
	(2) edge node {} (4)
	(3) edge node {} (4)
	(3) edge node {} (5)
	(5) edge node {} (6)
	(4) edge node {} (6)
	(6) edge node {} (8)
	(7) edge node {} (5)
	(8) edge node {} (7);
\end{tikzpicture}
}
\caption{$d_3(G_{2,4}) = 6$}
\end{subfigure}
\begin{subfigure}[b]{0.41\linewidth}
          \centering

          \resizebox{\linewidth}{!}{
\usetikzlibrary{positioning}
\tikzset{main node/.style={circle,fill=blue!20,draw,minimum size=1.5em,inner sep=0pt},
            }
\begin{tikzpicture}
	\node[main node] (1) {1};
	\node[main node] (2) [below= 1 cm of 1] {2};
	\node[main node] (3) [right= 1 cm of 1] {3};
	\node[main node] (4) [below= 1 cm of 3] {4};
	\node[main node] (5) [right = 1 cm of 3] {5};
	\node[main node] (6) [below= 1 cm of 5] {6};
	\node[main node] (7) [right = 1 cm of 5] {1};
	\node[main node] (8) [below= 1 cm of 7] {2};
	\node[main node] (9) [right = 1 cm of 7] {3};
	\node[main node] (10) [below= 1 cm of 9] {4};
	\path[draw,thick]
    	(1) edge node {} (2)
	(1) edge node {} (3)
	(2) edge node {} (4)
	(3) edge node {} (4)
	(3) edge node {} (5)
	(5) edge node {} (6)
	(4) edge node {} (6)
	(6) edge node {} (8)
	(7) edge node {} (5)
	(8) edge node {} (7)
	(7) edge node {} (9)
	(8) edge node {} (10)
	(9) edge node {} (10);
	\end{tikzpicture}
	}
\caption{$d_3(G_{2,5}) = 6$}
\end{subfigure}

\begin{subfigure}[b]{0.4\linewidth}
          \centering

          \resizebox{\linewidth}{!}{
\usetikzlibrary{positioning}
\tikzset{main node/.style={circle,fill=blue!20,draw,minimum size=1.5em,inner sep=0pt},
            }
\begin{tikzpicture}
	\node[main node] (1) {1};
	\node[main node] (2) [below= 1 cm of 1] {2};
	\node[main node] (3) [right= 1 cm of 1] {3};
	\node[main node] (4) [below= 1 cm of 3] {4};
	\node[main node] (5) [right = 1 cm of 3] {5};
	\node[main node] (6) [below= 1 cm of 5] {6};
	\node[main node] (7) [right = 1 cm of 5] {1};
	\node[main node] (8) [below= 1 cm of 7] {2};
	\node[main node] (9) [right = 1 cm of 7] {3};
	\node[main node] (10) [below= 1 cm of 9] {4};
	\node[main node] (11) [right = 1 cm of 9] {5};
	\node[main node] (12) [below= 1 cm of 11] {6};
	\path[draw,thick]
    	(1) edge node {} (2)
	(1) edge node {} (3)
	(2) edge node {} (4)
	(3) edge node {} (4)
	(3) edge node {} (5)
	(5) edge node {} (6)
	(4) edge node {} (6)
	(6) edge node {} (8)
	(7) edge node {} (5)
	(8) edge node {} (7)
	(7) edge node {} (9)
	(8) edge node {} (10)
	(9) edge node {} (10)
	(9) edge node {} (11)
	(10) edge node {} (12)
	(11) edge node {} (12);
	\end{tikzpicture}
	}
\caption{$d_3(G_{2,6}) = 6$}
\end{subfigure}
\begin{subfigure}[b]{0.53\linewidth}
          \centering

          \resizebox{\linewidth}{!}{
\usetikzlibrary{positioning}
\tikzset{main node/.style={circle,fill=blue!20,draw,minimum size=1.5em,inner sep=0pt},
            }
\begin{tikzpicture}
	\node[main node] (1) {1};
	\node[main node] (2) [below= 1 cm of 1] {2};
	\node[main node] (3) [right= 1 cm of 1] {3};
	\node[main node] (4) [below= 1 cm of 3] {4};
	\node[main node] (5) [right = 1 cm of 3] {5};
	\node[main node] (6) [below= 1 cm of 5] {6};
	\node[main node] (7) [right = 1 cm of 5] {7};
	\node[main node] (8) [below= 1 cm of 7] {7};
	\node[main node] (9) [right = 1 cm of 7] {7};
	\node[main node] (10) [below= 1 cm of 9] {7};
	\node[main node] (11) [right = 1 cm of 9] {1};
	\node[main node] (12) [below= 1 cm of 11] {2};
	\node[main node] (13) [right= 1 cm of 11] {3};
	\node[main node] (14) [below= 1 cm of 13] {4};
	\node[main node] (15) [right= 1 cm of 13] {5};
	\node[main node] (16) [below= 1 cm of 15] {6};
	\path[draw,thick]
    	(1) edge node {} (2)
	(1) edge node {} (3)
	(2) edge node {} (4)
	(3) edge node {} (4)
	(3) edge node {} (5)
	(5) edge node {} (6)
	(4) edge node {} (6)
	(6) edge node {} (8)
	(7) edge node {} (5)
	(8) edge node {} (7)
	(7) edge node {} (9)
	(8) edge node {} (10)
	(9) edge node {} (10)
	(9) edge node {} (11)
	(10) edge node {} (12)
	(11) edge node {} (12)
	(11) edge node {} (13)
	(12) edge node {} (14)
	(13) edge node {} (14)
	(13) edge node {} (15)
	(14) edge node {} (16)
	(15) edge node {} (16);
	\end{tikzpicture}
	}
\caption{$d_3(G_{2,8}) = 7$}
\end{subfigure}
\begin{subfigure}[b]{0.65\linewidth}
          \centering

          \resizebox{\linewidth}{!}{
\usetikzlibrary{positioning}
\tikzset{main node/.style={circle,fill=blue!20,draw,minimum size=1.5em,inner sep=0pt},
            }
\begin{tikzpicture}
	\node[main node] (1) {1};
	\node[main node] (2) [below= 1 cm of 1] {2};
	\node[main node] (3) [right= 1 cm of 1] {3};
	\node[main node] (4) [below= 1 cm of 3] {4};
	\node[main node] (5) [right = 1 cm of 3] {5};
	\node[main node] (6) [below= 1 cm of 5] {6};
	\node[main node] (7) [right = 1 cm of 5] {7};
	\node[main node] (8) [below= 1 cm of 7] {7};
	\node[main node] (9) [right = 1 cm of 7] {1};
	\node[main node] (10) [below= 1 cm of 9] {2};
	\node[main node] (11) [right = 1 cm of 9] {7};
	\node[main node] (12) [below= 1 cm of 11] {7};
	\node[main node] (13) [right= 1 cm of 11] {6};
	\node[main node] (14) [below= 1 cm of 13] {5};
	\node[main node] (15) [right= 1 cm of 13] {4};
	\node[main node] (16) [below= 1 cm of 15] {3};
	\node[main node] (17) [right= 1 cm of 15] {2};
	\node[main node] (18) [below= 1 cm of 17] {1};
	\path[draw,thick]
    	(1) edge node {} (2)
	(1) edge node {} (3)
	(2) edge node {} (4)
	(3) edge node {} (4)
	(3) edge node {} (5)
	(5) edge node {} (6)
	(4) edge node {} (6)
	(6) edge node {} (8)
	(7) edge node {} (5)
	(8) edge node {} (7)
	(7) edge node {} (9)
	(8) edge node {} (10)
	(9) edge node {} (10)
	(9) edge node {} (11)
	(10) edge node {} (12)
	(11) edge node {} (12)
	(11) edge node {} (13)
	(12) edge node {} (14)
	(13) edge node {} (14)
	(13) edge node {} (15)
	(14) edge node {} (16)
	(15) edge node {} (16)
	(15) edge node {} (17)
	(16) edge node {} (18)
	(17) edge node {} (18);
	\end{tikzpicture}
	}
\caption{$d_3(G_{2,9}) = 7$}
\end{subfigure}
\begin{subfigure}[b]{0.7\linewidth}
          \centering

          \resizebox{\linewidth}{!}{
\usetikzlibrary{positioning}
\tikzset{main node/.style={circle,fill=blue!20,draw,minimum size=1.5em,inner sep=0pt},
            }
\begin{tikzpicture}
	\node[main node] (1) {1};
	\node[main node] (2) [below= 1 cm of 1] {2};
	\node[main node] (3) [right= 1 cm of 1] {3};
	\node[main node] (4) [below= 1 cm of 3] {4};
	\node[main node] (5) [right = 1 cm of 3] {5};
	\node[main node] (6) [below= 1 cm of 5] {6};
	\node[main node] (7) [right = 1 cm of 5] {7};
	\node[main node] (8) [below= 1 cm of 7] {7};
	\node[main node] (9) [right = 1 cm of 7] {1};
	\node[main node] (10) [below= 1 cm of 9] {2};
	\node[main node] (11) [right = 1 cm of 9] {3};
	\node[main node] (12) [below= 1 cm of 11] {4};
	\node[main node] (13) [right= 1 cm of 11] {7};
	\node[main node] (14) [below= 1 cm of 13] {7};
	\node[main node] (15) [right= 1 cm of 13] {6};
	\node[main node] (16) [below= 1 cm of 15] {5};
	\node[main node] (17) [right= 1 cm of 15] {4};
	\node[main node] (18) [below= 1 cm of 17] {3};
	\node[main node] (19) [right= 1 cm of 17] {2};
	\node[main node] (20) [below= 1 cm of 19] {1};
	\path[draw,thick]
    	(1) edge node {} (2)
	(1) edge node {} (3)
	(2) edge node {} (4)
	(3) edge node {} (4)
	(3) edge node {} (5)
	(5) edge node {} (6)
	(4) edge node {} (6)
	(6) edge node {} (8)
	(7) edge node {} (5)
	(8) edge node {} (7)
	(7) edge node {} (9)
	(8) edge node {} (10)
	(9) edge node {} (10)
	(9) edge node {} (11)
	(10) edge node {} (12)
	(11) edge node {} (12)
	(11) edge node {} (13)
	(12) edge node {} (14)
	(13) edge node {} (14)
	(13) edge node {} (15)
	(14) edge node {} (16)
	(15) edge node {} (16)
	(15) edge node {} (17)
	(16) edge node {} (18)
	(17) edge node {} (18)
	(17) edge node {} (19)
	(18) edge node {} (20)
	(19) edge node {} (20);
	\end{tikzpicture}
	}
\caption{$d_3(G_{2,10}) = 7$}
\end{subfigure}
\caption{}
\label{constructions2}
\end{figure}

When $r=2$, we first see that $d_3(G_{2,2})=4$ and $d_3(G_{2,3})=6$ are also trivial. For $l=4,5,6$ we find that $d_3(G_{2,l})=6$, giving us our first cases where the grid graph is not $3$-domatically full. To see that $d_3(G_{2,4}) \leq 6$ we note that if we try to color its eight vertices with seven colors, then six of the color classes will contain one vertex each, and so $G_{2,4}$ must have at least six $3$-dominating vertices, when it only has four. Similarly, if we attempt to color the ten vertices of $G_{2,5}$ with seven colors, then at least four color classes will each contain one vertex. So $G_{2,5}$ would need at least four $3$-dominating vertices when it only has two, a contradiction. The same reasoning applies to $G_{2,6}$ when seven colors distributed over twelve vertices means the graph must have at least two $3$-dominating vertices when it has none. To show the lower bounds, $d_3(G_{2,l}) \geq 6$ for $l=4,5,6$, we give constructions in Figure~\ref{constructions2}. By Lemma~\ref{small}, we know that $d_3(G_{2,7})=7$ and $d_3(G_{2,l})=7$ for all $l \geq 11$. So we only need to consider $l=8,9,10$ which all end up being domatically full, $d_3(G_{2,l}=7$, by the constructions shown in Figure~\ref{constructions2}.

When $r=3$, then we know by Lemma~\ref{small} that $d_3(G_{3,l}) = 9$ when $l=6$ and for all $l \geq 9$ so we need only consider the cases where $l=3,4,5,7,8$. We find that $d_3(G_{3,7})=d_3(G_{3,8})=9$ which is demonstrated by constructions in Figure~\ref{constructions3}. For the other cases we get that $d_3(G_{3,3})=d_3(G_{3,4})=7$ and that $d_3(G_{3,5})=8$. To demonstrate the lower bounds, the constructions are also shown in Figure~\ref{constructions3}. To show the upper bounds, we again use the same pigeonhole argument as in the cases where $r=2$. First, $d_3(G_{3,3}) \leq 7$ since if we try to color the nine vertices with eight colors, then seven vertices must be $3$-dominating when only five are. Similarly, $d_3(G_{3,4}) \leq 7$ since distributing eight colors to twelve vertices implies that $G_{3,4}$ must have at least four $3$-dominating vertices when it only has two, and $d_3(G_{3,5}) \leq 8$ because nine colors distributed to fifteen vertices implies that three vertices of $G_{3,5}$ must be $3$-dominating when only one is.

\begin{figure}
\begin{subfigure}[b]{0.23\linewidth}
          \centering

          \resizebox{\linewidth}{!}{
\usetikzlibrary{positioning}
\tikzset{state/.style={circle,fill=blue!20,draw,minimum size=1.5 em,inner sep=0pt},
            }
\begin{tikzpicture}
    \node[state] (1) {6};
    \node[state] (2) [below=1 cm of 1] {1};
    \node[state] (3) [below =1 cm of 2] {7};
    \node[state] (4) [right =1 cm of 1] {4};
    \node[state] (5) [below =1 cm of 4] {5};
    \node[state] (6) [below=1 cm of 5 ] {2};
    \node[state] (7) [right=1 cm of 4] {7};
    \node[state] (8) [below =1 cm of 7] {3};
    \node[state] (9) [below =1 cm of 8] {6};
    \path[draw,thick]
    (1) edge node {} (2)
    (2) edge node {} (3)
    (4) edge node {} (5)
    (5) edge node {} (6)
    (7) edge node {} (8)
    (8) edge node {} (9)
    (1) edge node {} (4)
    (4) edge node {} (7)
    (2) edge node {} (5)
    (5) edge node {} (8)
    (3) edge node {} (6)
    (6) edge node {} (9);
\end{tikzpicture}
}
\caption{$d_3(G_{3,3}) = 7$}
\end{subfigure}
\begin{subfigure}[b]{0.32\linewidth}
          \centering

          \resizebox{\linewidth}{!}{
\usetikzlibrary{positioning}
\tikzset{state/.style={circle,fill=blue!20,draw,minimum size=1.5 em,inner sep=0pt},
            }
\begin{tikzpicture}
    \node[state] (1) {6};
    \node[state] (2) [below=1 cm of 1] {1};
    \node[state] (3) [below =1 cm of 2] {7};
    \node[state] (4) [right =1 cm of 1] {4};
    \node[state] (5) [below =1 cm of 4] {5};
    \node[state] (6) [below=1 cm of 5 ] {2};
    \node[state] (7) [right=1 cm of 4] {7};
    \node[state] (8) [below =1 cm of 7] {3};
    \node[state] (9) [below =1 cm of 8] {6};
    \node[state] (10) [right =1 cm of 7] {1};
    \node[state] (11) [below =1 cm of 10] {2};
    \node[state] (12) [below =1 cm of 11] {4};
    \path[draw,thick]
    (1) edge node {} (2)
    (2) edge node {} (3)
    (4) edge node {} (5)
    (5) edge node {} (6)
    (7) edge node {} (8)
    (8) edge node {} (9)
    (1) edge node {} (4)
    (4) edge node {} (7)
    (2) edge node {} (5)
    (5) edge node {} (8)
    (3) edge node {} (6)
    (6) edge node {} (9)
    (7) edge node {} (10)
    (8) edge node {} (11)
    (9) edge node {} (12)
    (10) edge node {} (11)
    (11) edge node {} (12);
\end{tikzpicture}
}
\caption{$d_3{(G_{3,4})}=7$}
\end{subfigure}
\begin{subfigure}[b]{0.4\linewidth}
          \centering

          \resizebox{\linewidth}{!}{
\usetikzlibrary{positioning}
\tikzset{state/.style={circle,fill=blue!20,draw,minimum size=1.5 em,inner sep=0pt},
            }
\begin{tikzpicture}
    \node[state] (1) {4};
    \node[state] (2) [below=1 cm of 1] {7};
    \node[state] (3) [below =1 cm of 2] {3};
    \node[state] (4) [right =1 cm of 1] {8};
    \node[state] (5) [below =1 cm of 4] {2};
    \node[state] (6) [below=1 cm of 5 ] {6};
    \node[state] (7) [right=1 cm of 4] {6};
    \node[state] (8) [below =1 cm of 7] {1};
    \node[state] (9) [below =1 cm of 8] {5};
    \node[state] (10) [right =1 cm of 7] {5};
    \node[state] (11) [below =1 cm of 10] {3};
    \node[state] (12) [below =1 cm of 11] {7};
    \node[state] (13) [right =1 cm of 10] {2};
    \node[state] (14) [below =1 cm of 13] {8};
    \node[state] (15) [below =1 cm of 14] {4};
    \path[draw,thick]
    (1) edge node {} (2)
    (2) edge node {} (3)
    (4) edge node {} (5)
    (5) edge node {} (6)
    (7) edge node {} (8)
    (8) edge node {} (9)
    (10) edge node {} (11)
    (11) edge node {} (12)
    (13) edge node {} (14)
    (14) edge node {} (15)
    (1) edge node {} (4)
    (4) edge node {} (7)
    (7) edge node {} (10)
    (10) edge node {} (13)
    (2) edge node {} (5)
    (5) edge node {} (8)
    (8) edge node {} (11)
    (11) edge node {} (14)
    (3) edge node {} (6)
    (6) edge node {} (9)
    (9) edge node {} (12)
    (12) edge node {} (15);
\end{tikzpicture}
}
\caption{$d_3(G_{3,5}) = 8$}
\end{subfigure}
\begin{subfigure}[b]{0.4\linewidth}
          \centering

          \resizebox{\linewidth}{!}{
\usetikzlibrary{positioning}
\tikzset{main node/.style={circle,fill=blue!20,draw,minimum size=1.5 em,inner sep=0pt},
            }
\begin{tikzpicture}
    \node[main node] (1) {1};
    \node[main node] (2) [below = 1cm of 1] {2};
    \node[main node] (3) [below = 1cm of 2] {3};
    \node[main node] (4) [right =1cm of 1] {4};
    \node[main node] (5) [below=1cm of 4] {5};
    \node[main node] (6) [below = 1cm of 5] {6};
    \node[main node] (7) [right = 1cm of 4] {7};
    \node[main node] (8) [below = 1cm of 7] {8};
    \node[main node] (9) [below = 1cm of 8] {9};
    \node[main node] (10) [right = 1cm of 7] {9};
    \node[main node] (11) [below = 1cm of 10] {8};
    \node[main node] (12) [below = 1cm of 11] {7};
    \node[main node] (13) [right = 1cm of 10] {6};
    \node[main node] (14) [below = 1cm of 13] {5};
    \node[main node] (15) [below = 1cm of 14] {4};
    \node[main node] (16) [right = 1cm of 13] {3};
    \node[main node] (17) [below = 1cm of 16] {2};
    \node[main node] (18) [below = 1cm of 17] {1};
    \path[draw,thick]
    (1) edge node {} (2)
    (2) edge node {} (3)
    (1) edge node {} (4)
    (2) edge node {} (5)
    (3) edge node {} (6)
    (4) edge node {} (5)
    (5) edge node {} (6)
    (7) edge node {} (4)
    (8) edge node {} (5)
    (9) edge node {} (6)
    (7) edge node {} (8)
    (8) edge node {} (9)
    (10) edge node {} (7)
    (11) edge node {} (8)
    (12) edge node {} (9)
    (10) edge node {} (11)
    (11) edge node {} (12)
    (13) edge node {} (10)
    (14) edge node {} (11)
    (15) edge node {} (12)
    (13) edge node {} (14)
    (14) edge node {} (15)
    (16) edge node {} (13)
    (17) edge node {} (14)
    (18) edge node {} (15)
    (16) edge node {} (17)
    (17) edge node {} (18);
\end{tikzpicture}
}
\caption{$d_3(G_{3,6}) = 9$}
\end{subfigure}
\begin{subfigure}[b]{0.47\linewidth}
          \centering

          \resizebox{\linewidth}{!}{
\usetikzlibrary{positioning}
\tikzset{main node/.style={circle,fill=blue!20,draw,minimum size=1.5 em,inner sep=0pt},
            }
            \begin{tikzpicture}
    \node[main node] (1) {1};
    \node[main node] (2) [below = 1cm of 1] {2};
    \node[main node] (3) [below = 1cm of 2] {3};
    \node[main node] (4) [right =1cm of 1] {4};
    \node[main node] (5) [below=1cm of 4] {5};
    \node[main node] (6) [below = 1cm of 5] {6};
    \node[main node] (7) [right = 1cm of 4] {7};
    \node[main node] (8) [below = 1cm of 7] {8};
    \node[main node] (9) [below = 1cm of 8] {9};
    \node[main node] (10) [right = 1cm of 7] {9};
    \node[main node] (11) [below = 1cm of 10] {3};
    \node[main node] (12) [below = 1cm of 11] {7};
    \node[main node] (13) [right = 1cm of 10] {7};
    \node[main node] (14) [below = 1cm of 13] {1};
    \node[main node] (15) [below = 1cm of 14] {9};
    \node[main node] (16) [right = 1cm of 13] {6};
    \node[main node] (17) [below = 1cm of 16] {2};
    \node[main node] (18) [below = 1cm of 17] {4};
    \node[main node] (19) [right = 1cm of 16] {3};
    \node[main node] (20) [below = 1cm of 19] {5};
    \node[main node] (21) [below = 1cm of 20] {8};
    \path[draw,thick]
    (1) edge node {} (2)
    (2) edge node {} (3)
    (1) edge node {} (4)
    (2) edge node {} (5)
    (3) edge node {} (6)
    (4) edge node {} (5)
    (5) edge node {} (6)
    (7) edge node {} (4)
    (8) edge node {} (5)
    (9) edge node {} (6)
    (7) edge node {} (8)
    (8) edge node {} (9)
    (10) edge node {} (7)
    (11) edge node {} (8)
    (12) edge node {} (9)
    (10) edge node {} (11)
    (11) edge node {} (12)
    (13) edge node {} (10)
    (14) edge node {} (11)
    (15) edge node {} (12)
    (13) edge node {} (14)
    (14) edge node {} (15)
    (16) edge node {} (13)
    (17) edge node {} (14)
    (18) edge node {} (15)
    (16) edge node {} (17)
    (17) edge node {} (18)
    (16) edge node {} (19)
    (17) edge node {} (20)
    (18) edge node {} (21)
    (19) edge node {} (20)
    (20) edge node {} (21);
\end{tikzpicture}
}
\caption{$d_3(G_{3,7}) = 9$}
\end{subfigure}
\caption{}
\label{constructions3}
\end{figure}

When $r=4$ we know by Theorem~\ref{mainthm} that $d_3(G_{4,l})=10$ when $l=5$ and for all $l \geq 8$. So we need only consider the cases for $l=4,6,7$. We find that $d_3(G_{4,7})=10$ as well with the construction given in Figure~\ref{construction47}. The other two cases are not $3$-domatically full. We get that $d_3(G_{4,4})=8$ with the lower bound demonstrated by construction in Figure~\ref{construction44} and the lower bound again given by the pigeonhole argument. If we distribute nine colors to the sixteen vertices, then two color classes have one vertex each, which means that these two vertices must be $3$-dominating, a contradiction since $G_{4,4}$ has no $3$-dominating vertices.

\begin{figure}
\begin{subfigure}[b]{0.46\linewidth}
          \centering

          \resizebox{\linewidth}{!}{
\usetikzlibrary{positioning}
\tikzset{state/.style={circle,fill=red!20,draw,minimum size=1.5 em,inner sep=0pt},
            }
\tikzset{state2/.style={circle,fill=green!20,draw,minimum size=1.5 em,inner sep=0pt},
            }
\tikzset{state3/.style={circle,fill=blue!20,draw,minimum size=1.5 em,inner sep=0pt},
            }
\begin{tikzpicture}
   \node[state] (1) {S};
    \node[state] (2) [below=1 cm of 1] {S};
    \node[state] (3) [below =1 cm of 2] {S};
    \node[state] (4) [below =1 cm of 3] {S};
    \node[state] (5) [right =1 cm of 1] {S};
    \node[state] (6) [below=1 cm of 5 ] {S};
    \node[state] (7) [below=1 cm of 6] {S};
    \node[state3] (8) [below =1 cm of 7] {};
    \node[state] (9) [right =1 cm of 5] {S};
    \node[state] (10) [below =1 cm of 9] {S};
    \node[state3] (11) [below =1 cm of 10] {};
    \node[state2] (12) [below =1 cm of 11] {T};
    \node[state] (13) [right =1 cm of 9] {S};
    \node[state3] (14) [below =1 cm of 13] {};
    \node[state2] (15) [below =1 cm of 14] {T};
    \node[state2] (16) [below =1 cm of 15] {T};
    \node[state3] (17) [right =1 cm of 13] {};
    \node[state2] (18) [below =1 cm of 17] {T};
    \node[state2] (19) [below =1 cm of 18] {T};
    \node[state2] (20) [below =1 cm of 19] {T};
    \node[state2] (21) [right =1 cm of 17] {T};
    \node[state2] (22) [below =1 cm of 21] {T};
    \node[state2] (23) [below =1 cm of 22] {T};
    \node[state2] (24) [below =1 cm of 23] {T};
    \path[draw,thick]
    (1) edge node {} (2)
    (2) edge node {} (3)
    (3) edge node {} (4)
    (5) edge node {} (6)
    (6) edge node {} (7)
    (7) edge node {} (8)
    (9) edge node {} (10)
    (10) edge node {} (11)
    (11) edge node {} (12)
    (13) edge node {} (14)
    (14) edge node {} (15)
    (15) edge node {} (16)
    (17) edge node {} (18)
    (18) edge node {} (19)
    (19) edge node {} (20)
    (21) edge node {} (22)
    (22) edge node {} (23)
    (23) edge node {} (24)
    (1) edge node {} (5)
    (5) edge node {} (9)
    (9) edge node {} (13)
    (13) edge node {} (17)
    (17) edge node {} (21)
    (2) edge node {} (6)
    (6) edge node {} (10)
    (10) edge node {} (14)
    (14) edge node {} (18)
    (18) edge node {} (22)
    (3) edge node {} (7)
    (7) edge node {} (11)   
    (11) edge node {} (15)
    (15) edge node {} (19)
    (19) edge node {} (23)
    (4) edge node {} (8)
    (8) edge node {} (12)
    (12) edge node {} (16)
    (16) edge node {} (20)
    (20) edge node {} (24);
\end{tikzpicture}
}
\caption{}
\label{helpful1}
\end{subfigure}
\begin{subfigure}[b]{0.46\linewidth}
          \centering

          \resizebox{\linewidth}{!}{
\usetikzlibrary{positioning}
\tikzset{state/.style={circle,fill=red!20,draw,minimum size=1.5 em,inner sep=0pt},
            }
\tikzset{state2/.style={circle,fill=green!20,draw,minimum size=1.5 em,inner sep=0pt},
            }
\tikzset{state3/.style={circle,fill=blue!20,draw,minimum size=1.5 em,inner sep=0pt},
            }
\begin{tikzpicture}
    \node[state] (1) {1};
    \node[state] (2) [below=1 cm of 1] {2};
    \node[state] (3) [below =1 cm of 2] {3};
    \node[state] (4) [below =1 cm of 3] {4};
    \node[state] (5) [right =1 cm of 1] {5};
    \node[state] (6) [below=1 cm of 5 ] {6};
    \node[state] (7) [below=1 cm of 6] {7};
    \node[state3] (8) [below =1 cm of 7] {};
    \node[state] (9) [right =1 cm of 5] {8};
    \node[state] (10) [below =1 cm of 9] {9};
    \node[state3] (11) [below =1 cm of 10] {};
    \node[state2] (12) [below =1 cm of 11] {T};
    \node[state] (13) [right =1 cm of 9] {10};
    \node[state3] (14) [below =1 cm of 13] {};
    \node[state2] (15) [below =1 cm of 14] {T};
    \node[state2] (16) [below =1 cm of 15] {T};
    \node[state3] (17) [right =1 cm of 13] {};
    \node[state2] (18) [below =1 cm of 17] {T};
    \node[state2] (19) [below =1 cm of 18] {T};
    \node[state2] (20) [below =1 cm of 19] {T};
    \node[state2] (21) [right =1 cm of 17] {T};
    \node[state2] (22) [below =1 cm of 21] {T};
    \node[state2] (23) [below =1 cm of 22] {T};
    \node[state2] (24) [below =1 cm of 23] {T};
    \path[draw,thick]
    (1) edge node {} (2)
    (2) edge node {} (3)
    (3) edge node {} (4)
    (5) edge node {} (6)
    (6) edge node {} (7)
    (7) edge node {} (8)
    (9) edge node {} (10)
    (10) edge node {} (11)
    (11) edge node {} (12)
    (13) edge node {} (14)
    (14) edge node {} (15)
    (15) edge node {} (16)
    (17) edge node {} (18)
    (18) edge node {} (19)
    (19) edge node {} (20)
    (21) edge node {} (22)
    (22) edge node {} (23)
    (23) edge node {} (24)
    (1) edge node {} (5)
    (5) edge node {} (9)
    (9) edge node {} (13)
    (13) edge node {} (17)
    (17) edge node {} (21)
    (2) edge node {} (6)
    (6) edge node {} (10)
    (10) edge node {} (14)
    (14) edge node {} (18)
    (18) edge node {} (22)
    (3) edge node {} (7)
    (7) edge node {} (11)   
    (11) edge node {} (15)
    (15) edge node {} (19)
    (19) edge node {} (23)
    (4) edge node {} (8)
    (8) edge node {} (12)
    (12) edge node {} (16)
    (16) edge node {} (20)
    (20) edge node {} (24);
\end{tikzpicture}
}
\caption{}
\label{helpful2}
\end{subfigure}
\caption{}
\end{figure}

For $l=6$ we find that $d_3(G_{4,6})=9$. Again, the lower bound is given by an explicit coloring in Figure~\ref{construction46}. The upper bound is a slightly more complicated version of the pigeonhole arguments given above. Suppose we try to give a proper $3$-domatic coloring of the 24 vertices of $G_{4,6}$ with ten colors. Then consider the sets \[S=\{x \in V(G_{4,6}) : d(x,(1,1)) \leq 3\},\] and \[T=\{x \in V(G_{4,6}) : d(x,(4,6)) \leq 3\}.\] Since $S \cap T = \emptyset$ and $|S|=|T|=10$ (see Figure~\ref{helpful1}), then it follows that each color must appear at least twice, once on a vertex in $S$ and once on a vertex in $T$. So at most four colors can appear more than twice. Without loss of generality, we may assume the vertices of $S$ are colored as in Figure~\ref{helpful2}. Note that the color $1$ on $(1,1)$ does not reach the vertices $(4,2)$ or $(1,6)$. Since \[d((4,2),(1,5)) = 7,\] then it is not possible to use the color $1$ only once more. So it must be used at least three times. Similarly, color $4$ does not cover $(1,2)$ or $(4,6)$ which are distance $7$ away from each other, color $5$ does not cover $(4,1)$ and $(1,6)$ which are at distance $8$, color $8$ does not cover $(4,1)$ or $(2,6)$ which are at distance $7$, and color $10$ does not cover $(2,1)$ or $(4,6)$ which are at distance $7$. Hence, there are at least five color classes which require three vertices each, a contradiction.

\begin{figure}
\begin{subfigure}[b]{0.33\linewidth}
          \centering

          \resizebox{\linewidth}{!}{
\usetikzlibrary{positioning}
\tikzset{state/.style={circle,fill=blue!20,draw,minimum size=1.5 em,inner sep=0pt},
            }
\begin{tikzpicture}
    \node[state] (1) {1};
    \node[state] (2) [below=1 cm of 1] {2};
    \node[state] (3) [below =1 cm of 2] {3};
    \node[state] (4) [below =1 cm of 3] {4};
    \node[state] (5) [right =1 cm of 1] {5};
    \node[state] (6) [below =1 cm of 5 ] {7};
    \node[state] (7) [below =1 cm of 6] {8};
    \node[state] (8) [below =1 cm of 7] {6};
    \node[state] (9) [right =1 cm of 5] {6};
    \node[state] (10) [below =1 cm of 9] {8};
    \node[state] (11) [below =1 cm of 10] {7};
    \node[state] (12) [below =1 cm of 11] {5};
    \node[state] (13) [right =1 cm of 9] {4};
    \node[state] (14) [below =1 cm of 13] {3};
    \node[state] (15) [below =1 cm of 14] {2};
    \node[state] (16) [below =1 cm of 15] {1};
    \path[draw,thick]
    (1) edge node {} (2)
    (2) edge node {} (3)
    (3) edge node {} (4)
    (1) edge node {} (5)
    (2) edge node {} (6)
    (3) edge node {} (7)
    (4) edge node {} (8)
    (5) edge node {} (6)
    (6) edge node {} (7)
    (7) edge node {} (8)
    (5) edge node {} (9)
    (6) edge node {} (10)
    (7) edge node {} (11)
    (8) edge node {} (12)
    (9) edge node {} (10)
    (10) edge node {} (11)
    (11) edge node {} (12)
    (9) edge node {} (13)
    (10) edge node {} (14)
    (11) edge node {} (15)
    (12) edge node {} (16)
    (13) edge node {} (14)
    (14) edge node {} (15)
    (15) edge node {} (16);
\end{tikzpicture}
}
\caption{$d_3(G_{4,4}) = 8$}
\label{construction44}
\end{subfigure}
\begin{subfigure}[b]{0.5\linewidth}
          \centering

          \resizebox{\linewidth}{!}{
\usetikzlibrary{positioning}
\tikzset{state/.style={circle,fill=blue!20,draw,minimum size=1.5 em,inner sep=0pt},
            }
\begin{tikzpicture}
   \node[state] (1) {5};
    \node[state] (2) [below=1 cm of 1] {2};
    \node[state] (3) [below =1 cm of 2] {4};
    \node[state] (4) [below =1 cm of 3] {6};
    \node[state] (5) [right =1 cm of 1] {7};
    \node[state] (6) [below=1 cm of 5 ] {1};
    \node[state] (7) [below=1 cm of 6] {3};
    \node[state] (8) [below =1 cm of 7] {7};
    \node[state] (9) [right =1 cm of 5] {8};
    \node[state] (10) [below =1 cm of 9] {9};
    \node[state] (11) [below =1 cm of 10] {8};
    \node[state] (12) [below =1 cm of 11] {5};
    \node[state] (13) [right =1 cm of 9] {6};
    \node[state] (14) [below =1 cm of 13] {9};
    \node[state] (15) [below =1 cm of 14] {8};
    \node[state] (16) [below =1 cm of 15] {9};
    \node[state] (17) [right =1 cm of 13] {7};
    \node[state] (18) [below =1 cm of 17] {4};
    \node[state] (19) [below =1 cm of 18] {2};
    \node[state] (20) [below =1 cm of 19] {7};
    \node[state] (21) [right =1 cm of 17] {5};
    \node[state] (22) [below =1 cm of 21] {3};
    \node[state] (23) [below =1 cm of 22] {1};
    \node[state] (24) [below =1 cm of 23] {6};
    \path[draw,thick]
    (1) edge node {} (2)
    (2) edge node {} (3)
    (3) edge node {} (4)
    (5) edge node {} (6)
    (6) edge node {} (7)
    (7) edge node {} (8)
    (9) edge node {} (10)
    (10) edge node {} (11)
    (11) edge node {} (12)
    (13) edge node {} (14)
    (14) edge node {} (15)
    (15) edge node {} (16)
    (17) edge node {} (18)
    (18) edge node {} (19)
    (19) edge node {} (20)
    (21) edge node {} (22)
    (22) edge node {} (23)
    (23) edge node {} (24)
    (1) edge node {} (5)
    (5) edge node {} (9)
    (9) edge node {} (13)
    (13) edge node {} (17)
    (17) edge node {} (21)
    (2) edge node {} (6)
    (6) edge node {} (10)
    (10) edge node {} (14)
    (14) edge node {} (18)
    (18) edge node {} (22)
    (3) edge node {} (7)
    (7) edge node {} (11)   
    (11) edge node {} (15)
    (15) edge node {} (19)
    (19) edge node {} (23)
    (4) edge node {} (8)
    (8) edge node {} (12)
    (12) edge node {} (16)
    (16) edge node {} (20)
    (20) edge node {} (24);
\end{tikzpicture}
}
\caption{$d_3(G_{4,6}) = 9$}
\label{construction46}
\end{subfigure}
\begin{subfigure}[b]{0.52\linewidth}
          \centering

          \resizebox{\linewidth}{!}{
\usetikzlibrary{positioning}
\tikzset{main node/.style={circle,fill=blue!20,draw,minimum size=1.5 em,inner sep=0pt},
            }
\begin{tikzpicture}
    \node[main node] (1) {1};
    \node[main node] (2) [below = 1cm of 1] {10};
    \node[main node] (3) [below = 1cm of 2] {8};
    \node[main node] (4) [below =1cm of 3] {2};
    \node[main node] (5) [right =1cm of 1] {7};
    \node[main node] (6) [below = 1cm of 5] {5};
    \node[main node] (7) [below = 1cm of 6] {6};
    \node[main node] (8) [below = 1cm of 7] {9};
    \node[main node] (9) [right = 1cm of 5] {9};
    \node[main node] (10) [below = 1cm of 9] {4};
    \node[main node] (11) [below = 1cm of 10] {3};
    \node[main node] (12) [below = 1cm of 11] {7};
    \node[main node] (13) [right = 1cm of 9] {3};
    \node[main node] (14) [below = 1cm of 13] {2};
    \node[main node] (15) [below  = 1cm of 14] {1};
    \node[main node] (16) [below = 1cm of 15] {4};
    \node[main node] (17) [right = 1cm of 13] {10};
    \node[main node] (18) [below = 1cm of 17] {4};
    \node[main node] (19) [below = 1cm of 18] {3};
    \node[main node] (20) [below = 1cm of 19] {8};
    \node[main node] (21) [right = 1cm of 17] {8};
    \node[main node] (22) [below = 1cm of 21] {6};
    \node[main node] (23) [below = 1cm of 22] {5};
    \node[main node] (24) [below = 1cm of 23] {10};
    \node[main node] (25) [right = 1cm of 21] {1};
    \node[main node] (26) [below = 1cm of 25] {9};
    \node[main node] (27) [below = 1cm of 26] {7};
    \node[main node] (28) [below = 1cm of 27] {2};
    \path[draw,thick]
    (1) edge node {} (2)
    (2) edge node {} (3)
    (3) edge node {} (4)
    (1) edge node {} (5)
    (2) edge node {} (6)
    (3) edge node {} (7)
    (4) edge node {} (8)
    (5) edge node {} (6)
    (6) edge node {} (7)
    (7) edge node {} (8)
    (5) edge node {} (9)
    (6) edge node {} (10)
    (7) edge node {} (11)
    (8) edge node {} (12)
    (9) edge node {} (10)
    (10) edge node {} (11)
    (11) edge node {} (12)
    (9) edge node {} (13)
    (10) edge node {} (14)
    (11) edge node {} (15)
    (12) edge node {} (16)
    (13) edge node {} (14)
    (14) edge node {} (15)
    (15) edge node {} (16)
    (13) edge node {} (17)
    (14) edge node {} (18)
    (15) edge node {} (19)
    (16) edge node {} (20)
    (17) edge node {} (18)
    (18) edge node {} (19)
    (19) edge node {} (20)
    (17) edge node {} (21)
    (18) edge node {} (22)
    (19) edge node {} (23)
    (20) edge node {} (24)
    (21) edge node {} (22)
    (22) edge node {} (23)
    (23) edge node {} (24)
    (21) edge node {} (25)
    (22) edge node {} (26)
    (23) edge node {} (27)
    (24) edge node {} (28)
    (25) edge node {} (26)
    (26) edge node {} (27)
    (27) edge node {} (28);
\end{tikzpicture}
}
\caption{$d_3(G_{4,7}) = 10$}
\label{construction47}
\end{subfigure}
\begin{subfigure}[b]{0.35\linewidth}
          \centering

          \resizebox{\linewidth}{!}{
\usetikzlibrary{positioning}
\tikzset{state/.style={circle,fill=blue!20,draw,minimum size=1.5 em,inner sep=0pt},
            }
\begin{tikzpicture}
    \node[state] (1) {9};
    \node[state] (2) [below=1 cm of 1] {6};
    \node[state] (3) [below =1 cm of 2] {1};
    \node[state] (4) [below =1 cm of 3] {7};
    \node[state] (5) [below =1 cm of 4] {9};
    \node[state] (6) [right=1 cm of 1 ] {8};
    \node[state] (7) [below=1 cm of 6] {5};
    \node[state] (8) [below =1 cm of 7] {2};
    \node[state] (9) [below =1 cm of 8] {8};
    \node[state] (10) [below =1 cm of 9] {5};
    \node[state] (11) [right =1 cm of 6] {3};
    \node[state] (12) [below =1 cm of 11] {4};
    \node[state] (13) [below =1 cm of 12] {9};
    \node[state] (14) [below =1 cm of 13] {3};
    \node[state] (15) [below =1 cm of 14] {4};
    \node[state] (16) [right =1 cm of 11] {7};
    \node[state] (17) [below =1 cm of 16] {6};
    \node[state] (18) [below =1 cm of 17] {1};
    \node[state] (19) [below =1 cm of 18] {7};
    \node[state] (20) [below =1 cm of 19] {6};
    \node[state] (21) [right =1 cm of 16] {9};
    \node[state] (22) [below =1 cm of 21] {5};
    \node[state] (23) [below =1 cm of 22] {2};
    \node[state] (24) [below =1 cm of 23] {8};
    \node[state] (25) [below =1 cm of 24] {9};
    \path[draw,thick]
    (1) edge node {} (2)
    (2) edge node {} (3)
    (3) edge node {} (4)
    (4) edge node {} (5)
    (6) edge node {} (7)
    (7) edge node {} (8)
    (8) edge node {} (9)
    (9) edge node {} (10)
    (11) edge node {} (12)
    (12) edge node {} (13)
    (13) edge node {} (14)
    (14) edge node {} (15)
    (16) edge node {} (17)
    (17) edge node {} (18)
    (18) edge node {} (19)
    (19) edge node {} (20)
    (21) edge node {} (22)
    (22) edge node {} (23)
    (23) edge node {} (24)
    (24) edge node {} (25)
    (1) edge node {} (6)
    (6) edge node {} (11)
    (11) edge node {} (16)
    (16) edge node {} (21)
    (2) edge node {} (7)
    (7) edge node {} (12)
    (12) edge node {} (17)
    (17) edge node {} (22)
    (3) edge node {} (8)
    (8) edge node {} (13)
    (13) edge node {} (18)
    (18) edge node {} (23)   
    (4) edge node {} (9)
    (9) edge node {} (14)
    (14) edge node {} (19)
    (19) edge node {} (24)
    (5) edge node {} (10)
    (10) edge node {} (15)
    (15) edge node {} (20)
    (20) edge node {} (25);
\end{tikzpicture}
}
\caption{$d_3(G_{5,5}) = 9$}
\label{construction55}
\end{subfigure}
\caption{}
\end{figure}

When $r=5$, we note that by reversing the roles of rows and columns, then the standard block coloring and its vertical (now horizontal) extensions demonstrate that $d_3(G_{5,l})=10$ for all $l \geq 6$. So we need only consider $l=5$. We find that $d_3(G_{5,5})=9$. The lower bound is demonstrated by an explicit coloring given in Figure~\ref{construction55}. The upper bound is shown with an argument similar to the one given for $G_{4,6}$. Suppose we try to give a proper $3$-domatic coloring of the 25 vertices of $G_{5,5}$ with ten colors. Then consider the sets \[S=\{x \in V(G_{5,5}) : d(x,(1,1)) \leq 3\},\] and \[T=\{x \in V(G_{5,5}) : d(x,(5,5)) \leq 3\}.\] Since $S \cap T = \emptyset$ and $|S|=|T|=10$ (see Figure~\ref{helpful3}), then it follows that each color must appear at least twice, once on a vertex in $S$ and once on a vertex in $T$. So at most five colors can appear more than twice. Without loss of generality, we may assume the vertices of $S$ are colored as in Figure~\ref{helpful4}. Now, color $1$ does not cover $(5,1)$ or $(1,5)$, color $2$ does not cover $(5,2)$ or $(1,5)$, color $4$ does not cover $(1,2)$ or $(5,5)$, color $5$ does not cover $(2,5)$ or $(5,1)$, color $6$ does not cover $(1,5)$ or $(5,1)$, and color $10$ does not cover $(2,1)$ or $(5,5)$. Since each of these pairs are all at least distance $7$ from each other, then each of these six colors must be used three or more times, a contradiction.

\begin{figure}
\begin{subfigure}[b]{0.42\linewidth}
          \centering

          \resizebox{\linewidth}{!}{

\usetikzlibrary{positioning}
\tikzset{state/.style={circle,fill=red!20,draw,minimum size=1.5 em,inner sep=0pt},
            }
\tikzset{state2/.style={circle,fill=green!20,draw,minimum size=1.5 em,inner sep=0pt},
            }
\tikzset{state3/.style={circle,fill=blue!20,draw,minimum size=1.5 em,inner sep=0pt},
            }
\begin{tikzpicture}
    \node[state] (1) {S};
    \node[state] (2) [below=1 cm of 1] {S};
    \node[state] (3) [below =1 cm of 2] {S};
    \node[state] (4) [below =1 cm of 3] {S};
    \node[state3] (5) [below =1 cm of 4] {};
    \node[state] (6) [right=1 cm of 1 ] {S};
    \node[state] (7) [below=1 cm of 6] {S};
    \node[state] (8) [below =1 cm of 7] {S};
    \node[state3] (9) [below =1 cm of 8] {};
    \node[state2] (10) [below =1 cm of 9] {T};
    \node[state] (11) [right =1 cm of 6] {S};
    \node[state] (12) [below =1 cm of 11] {S};
    \node[state3] (13) [below =1 cm of 12] {};
    \node[state2] (14) [below =1 cm of 13] {T};
    \node[state2] (15) [below =1 cm of 14] {T};
    \node[state] (16) [right =1 cm of 11] {S};
    \node[state3] (17) [below =1 cm of 16] {};
    \node[state2] (18) [below =1 cm of 17] {T};
    \node[state2] (19) [below =1 cm of 18] {T};
    \node[state2] (20) [below =1 cm of 19] {T};
    \node[state3] (21) [right =1 cm of 16] {};
    \node[state2] (22) [below =1 cm of 21] {T};
    \node[state2] (23) [below =1 cm of 22] {T};
    \node[state2] (24) [below =1 cm of 23] {T};
    \node[state2] (25) [below =1 cm of 24] {T};
    \path[draw,thick]
    (1) edge node {} (2)
    (2) edge node {} (3)
    (3) edge node {} (4)
    (4) edge node {} (5)
    (6) edge node {} (7)
    (7) edge node {} (8)
    (8) edge node {} (9)
    (9) edge node {} (10)
    (11) edge node {} (12)
    (12) edge node {} (13)
    (13) edge node {} (14)
    (14) edge node {} (15)
    (16) edge node {} (17)
    (17) edge node {} (18)
    (18) edge node {} (19)
    (19) edge node {} (20)
    (21) edge node {} (22)
    (22) edge node {} (23)
    (23) edge node {} (24)
    (24) edge node {} (25)
    (1) edge node {} (6)
    (6) edge node {} (11)
    (11) edge node {} (16)
    (16) edge node {} (21)
    (2) edge node {} (7)
    (7) edge node {} (12)
    (12) edge node {} (17)
    (17) edge node {} (22)
    (3) edge node {} (8)
    (8) edge node {} (13)
    (13) edge node {} (18)
    (18) edge node {} (23)   
    (4) edge node {} (9)
    (9) edge node {} (14)
    (14) edge node {} (19)
    (19) edge node {} (24)
    (5) edge node {} (10)
    (10) edge node {} (15)
    (15) edge node {} (20)
    (20) edge node {} (25);
\end{tikzpicture}
}
\caption{}
\label{helpful3}
\end{subfigure}
\begin{subfigure}[b]{0.42\linewidth}
          \centering

          \resizebox{\linewidth}{!}{

\usetikzlibrary{positioning}
\tikzset{state/.style={circle,fill=red!20,draw,minimum size=1.5 em,inner sep=0pt},
            }
\tikzset{state2/.style={circle,fill=green!20,draw,minimum size=1.5 em,inner sep=0pt},
            }
\tikzset{state3/.style={circle,fill=blue!20,draw,minimum size=1.5 em,inner sep=0pt},
            }
\begin{tikzpicture}
    \node[state] (1) {1};
    \node[state] (2) [below=1 cm of 1] {2};
    \node[state] (3) [below =1 cm of 2] {3};
    \node[state] (4) [below =1 cm of 3] {4};
    \node[state3] (5) [below =1 cm of 4] {};
    \node[state] (6) [right=1 cm of 1 ] {5};
    \node[state] (7) [below=1 cm of 6] {6};
    \node[state] (8) [below =1 cm of 7] {7};
    \node[state3] (9) [below =1 cm of 8] {};
    \node[state2] (10) [below =1 cm of 9] {T};
    \node[state] (11) [right =1 cm of 6] {8};
    \node[state] (12) [below =1 cm of 11] {9};
    \node[state3] (13) [below =1 cm of 12] {};
    \node[state2] (14) [below =1 cm of 13] {T};
    \node[state2] (15) [below =1 cm of 14] {T};
    \node[state] (16) [right =1 cm of 11] {10};
    \node[state3] (17) [below =1 cm of 16] {};
    \node[state2] (18) [below =1 cm of 17] {T};
    \node[state2] (19) [below =1 cm of 18] {T};
    \node[state2] (20) [below =1 cm of 19] {T};
    \node[state3] (21) [right =1 cm of 16] {};
    \node[state2] (22) [below =1 cm of 21] {T};
    \node[state2] (23) [below =1 cm of 22] {T};
    \node[state2] (24) [below =1 cm of 23] {T};
    \node[state2] (25) [below =1 cm of 24] {T};
    \path[draw,thick]
    (1) edge node {} (2)
    (2) edge node {} (3)
    (3) edge node {} (4)
    (4) edge node {} (5)
    (6) edge node {} (7)
    (7) edge node {} (8)
    (8) edge node {} (9)
    (9) edge node {} (10)
    (11) edge node {} (12)
    (12) edge node {} (13)
    (13) edge node {} (14)
    (14) edge node {} (15)
    (16) edge node {} (17)
    (17) edge node {} (18)
    (18) edge node {} (19)
    (19) edge node {} (20)
    (21) edge node {} (22)
    (22) edge node {} (23)
    (23) edge node {} (24)
    (24) edge node {} (25)
    (1) edge node {} (6)
    (6) edge node {} (11)
    (11) edge node {} (16)
    (16) edge node {} (21)
    (2) edge node {} (7)
    (7) edge node {} (12)
    (12) edge node {} (17)
    (17) edge node {} (22)
    (3) edge node {} (8)
    (8) edge node {} (13)
    (13) edge node {} (18)
    (18) edge node {} (23)   
    (4) edge node {} (9)
    (9) edge node {} (14)
    (14) edge node {} (19)
    (19) edge node {} (24)
    (5) edge node {} (10)
    (10) edge node {} (15)
    (15) edge node {} (20)
    (20) edge node {} (25);
\end{tikzpicture}
}
\caption{}
\label{helpful4}
\end{subfigure}
\caption{}
\end{figure}

Finally, we get that the remaining grid graph, for all $6 \leq r \leq l$ are $3$-domatically full. When $r \geq 6$, Theorem~\ref{mainthm} gives us this for all $l \geq 8$. To show that \[d_3(G_{6,6})=d_3(G_{6,7}) =d_3(G_{7,7}) = 10,\] we use the constructions given in Figure~\ref{constructions67}.

\begin{figure}
\begin{subfigure}[b]{0.42\linewidth}
          \centering

          \resizebox{\linewidth}{!}{

\usetikzlibrary{positioning}
\tikzset{main node/.style={circle,fill=blue!20,draw,minimum size=1.5 em,inner sep=0pt},
            }
\begin{tikzpicture}
    \node[main node] (1) {4};
    \node[main node] (2) [below = 1cm of 1] {10};
    \node[main node] (3) [below = 1cm of 2] {2};
    \node[main node] (4) [below =1cm of 3] {5};
    \node[main node] (5) [below =1cm of 4] {4};
    \node[main node] (6) [below = 1cm of 5] {1};
    \node[main node] (7) [right = 1cm of 1] {3};
    \node[main node] (8) [below = 1cm of 7] {7};
    \node[main node] (9) [below = 1cm of 8] {8};
    \node[main node] (10) [below = 1cm of 9] {9};
    \node[main node] (11) [below = 1cm of 10] {6};
    \node[main node] (12) [below = 1cm of 11] {10};
    \node[main node] (13) [right = 1cm of 7] {6};
    \node[main node] (14) [below = 1cm of 13] {9};
    \node[main node] (15) [below  = 1cm of 14] {1};
    \node[main node] (16) [below = 1cm of 15] {2};
    \node[main node] (17) [below = 1cm of 16] {7};
    \node[main node] (18) [below = 1cm of 17] {3};
    \node[main node] (19) [right = 1cm of 13] {1};
    \node[main node] (20) [below = 1cm of 19] {5};
    \node[main node] (21) [below = 1cm of 20] {4};
    \node[main node] (22) [below = 1cm of 21] {3};
    \node[main node] (23) [below = 1cm of 22] {9};
    \node[main node] (24) [below = 1cm of 23] {8};
    \node[main node] (25) [right = 1cm of 19] {10};
    \node[main node] (26) [below = 1cm of 25] {8};
    \node[main node] (27) [below = 1cm of 26] {9};
    \node[main node] (28) [below = 1cm of 27] {6};
    \node[main node] (29) [below = 1cm of 28] {5};
    \node[main node] (30) [below = 1cm of 29] {1};
    \node[main node] (31) [right = 1cm of 25] {3};
    \node[main node] (32) [below = 1cm of 31] {2};
    \node[main node] (33) [below = 1cm of 32] {7};
    \node[main node] (34) [below = 1cm of 33] {4};
    \node[main node] (35) [below = 1cm of 34] {10};
    \node[main node] (36) [below = 1cm of 35] {2};
    \path[draw,thick]
    (1) edge node {} (2)
    (2) edge node {} (3)
    (3) edge node {} (4)
    (4) edge node {} (5)
    (5) edge node {} (6)
    (1) edge node {} (7)
    (2) edge node {} (8)
    (3) edge node {} (9)
    (4) edge node {} (10)
    (5) edge node {} (11)
    (6) edge node {} (12)
    (7) edge node {} (8)
    (8) edge node {} (9)
    (9) edge node {} (10)
    (10) edge node {} (11)
    (11) edge node {} (12)
    (7) edge node {} (13)
    (8) edge node {} (14)
    (9) edge node {} (15)
    (10) edge node {} (16)
    (11) edge node {} (17)
    (12) edge node {} (18)
    (13) edge node {} (14)
    (14) edge node {} (15)
    (15) edge node {} (16)
    (16) edge node {} (17)
    (17) edge node {} (18)
    (13) edge node {} (19)
    (14) edge node {} (20)
    (15) edge node {} (21)
    (16) edge node {} (22)
    (17) edge node {} (23)
    (18) edge node {} (24)
    (19) edge node {} (20)
    (20) edge node {} (21)
    (21) edge node {} (22)
    (22) edge node {} (23)
    (23) edge node {} (24)
    (19) edge node {} (25)
    (20) edge node {} (26)
    (21) edge node {} (27)
    (22) edge node {} (28)
    (23) edge node {} (29)
    (24) edge node {} (30)
    (25) edge node {} (26)
    (26) edge node {} (27)
    (27) edge node {} (28)
    (28) edge node {} (29)
    (29) edge node {} (30)
    (25) edge node {} (31)
    (26) edge node {} (32)
    (27) edge node {} (33)
    (28) edge node {} (34)
    (29) edge node {} (35)
    (30) edge node {} (36)
    (31) edge node {} (32)
    (32) edge node {} (33)
    (33) edge node {} (34)
    (34) edge node {} (35)
    (35) edge node {} (36);
\end{tikzpicture}
}
\caption{$d_3(G_{6,6}) = 10$}
\end{subfigure}
\begin{subfigure}[b]{0.48\linewidth}
          \centering

          \resizebox{\linewidth}{!}{

\usetikzlibrary{positioning}
\tikzset{main node/.style={circle,fill=blue!20,draw,minimum size=1.5 em,inner sep=0pt},
            }
\begin{tikzpicture}
    \node[main node] (1) {1};
    \node[main node] (2) [below = 1cm of 1] {8};
    \node[main node] (3) [below = 1cm of 2] {5};
    \node[main node] (4) [below =1cm of 3] {7};
    \node[main node] (5) [below =1cm of 4] {6};
    \node[main node] (6) [below = 1cm of 5] {4};
    \node[main node] (7) [right = 1cm of 1] {2};
    \node[main node] (8) [below = 1cm of 7] {3};
    \node[main node] (9) [below = 1cm of 8] {9};
    \node[main node] (10) [below = 1cm of 9] {10};
    \node[main node] (11) [below = 1cm of 10] {2};
    \node[main node] (12) [below = 1cm of 11] {3};
    \node[main node] (13) [right = 1cm of 7] {6};
    \node[main node] (14) [below = 1cm of 13] {4};
    \node[main node] (15) [below  = 1cm of 14] {2};
    \node[main node] (16) [below = 1cm of 15] {10};
    \node[main node] (17) [below = 1cm of 16] {1};
    \node[main node] (18) [below = 1cm of 17] {9};
    \node[main node] (19) [right = 1cm of 13] {10};
    \node[main node] (20) [below = 1cm of 19] {8};
    \node[main node] (21) [below = 1cm of 20] {5};
    \node[main node] (22) [below = 1cm of 21] {6};
    \node[main node] (23) [below = 1cm of 22] {7};
    \node[main node] (24) [below = 1cm of 23] {8};
    \node[main node] (25) [right = 1cm of 19] {7};
    \node[main node] (26) [below = 1cm of 25] {3};
    \node[main node] (27) [below = 1cm of 26] {9};
    \node[main node] (28) [below = 1cm of 27] {4};
    \node[main node] (29) [below = 1cm of 28] {2};
    \node[main node] (30) [below = 1cm of 29] {5};
    \node[main node] (31) [right = 1cm of 25] {1};
    \node[main node] (32) [below = 1cm of 31] {4};
    \node[main node] (33) [below = 1cm of 32] {8};
    \node[main node] (34) [below = 1cm of 33] {10};
    \node[main node] (35) [below = 1cm of 34] {1};
    \node[main node] (36) [below = 1cm of 35] {4};
    \node[main node] (37) [right = 1cm of 31] {2};
    \node[main node] (38) [below = 1cm of 37] {5};
    \node[main node] (39) [below = 1cm of 38] {9};
    \node[main node] (40) [below = 1cm of 39] {6};
    \node[main node] (41) [below = 1cm of 40] {7};
    \node[main node] (42) [below = 1cm of 41] {3};
    \path[draw,thick]
    (1) edge node {} (2)
    (2) edge node {} (3)
    (3) edge node {} (4)
    (4) edge node {} (5)
    (5) edge node {} (6)
    (1) edge node {} (7)
    (2) edge node {} (8)
    (3) edge node {} (9)
    (4) edge node {} (10)
    (5) edge node {} (11)
    (6) edge node {} (12)
    (7) edge node {} (8)
    (8) edge node {} (9)
    (9) edge node {} (10)
    (10) edge node {} (11)
    (11) edge node {} (12)
    (7) edge node {} (13)
    (8) edge node {} (14)
    (9) edge node {} (15)
    (10) edge node {} (16)
    (11) edge node {} (17)
    (12) edge node {} (18)
    (13) edge node {} (14)
    (14) edge node {} (15)
    (15) edge node {} (16)
    (16) edge node {} (17)
    (17) edge node {} (18)
    (13) edge node {} (19)
    (14) edge node {} (20)
    (15) edge node {} (21)
    (16) edge node {} (22)
    (17) edge node {} (23)
    (18) edge node {} (24)
    (19) edge node {} (20)
    (20) edge node {} (21)
    (21) edge node {} (22)
    (22) edge node {} (23)
    (23) edge node {} (24)
    (19) edge node {} (25)
    (20) edge node {} (26)
    (21) edge node {} (27)
    (22) edge node {} (28)
    (23) edge node {} (29)
    (24) edge node {} (30)
    (25) edge node {} (26)
    (26) edge node {} (27)
    (27) edge node {} (28)
    (28) edge node {} (29)
    (29) edge node {} (30)
    (25) edge node {} (31)
    (26) edge node {} (32)
    (27) edge node {} (33)
    (28) edge node {} (34)
    (29) edge node {} (35)
    (30) edge node {} (36)
    (31) edge node {} (32)
    (32) edge node {} (33)
    (33) edge node {} (34)
    (34) edge node {} (35)
    (35) edge node {} (36)
    (31) edge node {} (37)
    (32) edge node {} (38)
    (33) edge node {} (39)
    (34) edge node {} (40)
    (35) edge node {} (41)
    (36) edge node {} (42)
    (37) edge node {} (38)
    (38) edge node {} (39)
    (39) edge node {} (40)
    (40) edge node {} (41)
    (41) edge node {} (42);  
\end{tikzpicture}
}
\caption{$d_3(G_{6,7}) = 10$}
\end{subfigure}

\begin{subfigure}[b]{0.48\linewidth}
          \centering

          \resizebox{\linewidth}{!}{

\usetikzlibrary{positioning}
\tikzset{main node/.style={circle,fill=blue!20,draw,minimum size=1.5 em,inner sep=0pt},
            }
\begin{tikzpicture}
    \node[main node] (1) {8};
    \node[main node] (2) [below = 1cm of 1] {7};
    \node[main node] (3) [below = 1cm of 2] {9};
    \node[main node] (4) [below =1cm of 3] {5};
    \node[main node] (5) [below =1cm of 4] {9};
    \node[main node] (6) [below = 1cm of 5] {8};
    \node[main node] (7) [below = 1cm of 6] {10};
    \node[main node] (8) [right = 1cm of 1] {10};
    \node[main node] (9) [below = 1cm of 8] {2};
    \node[main node] (10) [below = 1cm of 9] {3};
    \node[main node] (11) [below = 1cm of 10] {7};
    \node[main node] (12) [below = 1cm of 11] {3};
    \node[main node] (13) [below = 1cm of 12] {1};
    \node[main node] (14) [below = 1cm of 13] {7};
    \node[main node] (15) [right  = 1cm of 8] {4};
    \node[main node] (16) [below = 1cm of 15] {1};
    \node[main node] (17) [below = 1cm of 16] {6};
    \node[main node] (18) [below = 1cm of 17] {4};
    \node[main node] (19) [below = 1cm of 18] {6};
    \node[main node] (20) [below = 1cm of 19] {2};
    \node[main node] (21) [below = 1cm of 20] {4};
    \node[main node] (22) [right = 1cm of 15] {6};
    \node[main node] (23) [below = 1cm of 22] {9};
    \node[main node] (24) [below = 1cm of 23] {8};
    \node[main node] (25) [below = 1cm of 24] {3};
    \node[main node] (26) [below = 1cm of 25] {10};
    \node[main node] (27) [below = 1cm of 26] {9};
    \node[main node] (28) [below = 1cm of 27] {6};
    \node[main node] (29) [right = 1cm of 22] {5};
    \node[main node] (30) [below = 1cm of 29] {1};
    \node[main node] (31) [below = 1cm of 30] {6};
    \node[main node] (32) [below = 1cm of 31] {5};
    \node[main node] (33) [below = 1cm of 32] {6};
    \node[main node] (34) [below = 1cm of 33] {2};
    \node[main node] (35) [below = 1cm of 34] {5};
    \node[main node] (36) [right = 1cm of 29] {7};
    \node[main node] (37) [below = 1cm of 36] {2};
    \node[main node] (38) [below = 1cm of 37] {3};
    \node[main node] (39) [below = 1cm of 38] {7};
    \node[main node] (40) [below = 1cm of 39] {3};
    \node[main node] (41) [below = 1cm of 40] {1};
    \node[main node] (42) [below = 1cm of 41] {7};
    \node[main node] (43) [right = 1cm of 36] {8};
    \node[main node] (44) [below = 1cm of 43] {10};
    \node[main node] (45) [below = 1cm of 44] {9};
    \node[main node] (46) [below = 1cm of 45] {4};
    \node[main node] (47) [below = 1cm of 46] {9};
    \node[main node] (48) [below = 1cm of 47] {10};
    \node[main node] (49) [below = 1cm of 48] {8};
    \path[draw,thick]
    (1) edge node {} (2)
    (2) edge node {} (3)
    (3) edge node {} (4)
    (4) edge node {} (5)
    (5) edge node {} (6)
    (6) edge node {} (7)
    (1) edge node {} (8)
    (2) edge node {} (9)
    (3) edge node {} (10)
    (4) edge node {} (11)
    (5) edge node {} (12)
    (6) edge node {} (13)
    (7) edge node {} (14)
    (8) edge node {} (9)
    (9) edge node {} (10)
    (10) edge node {} (11)
    (11) edge node {} (12)
    (12) edge node {} (13)
    (13) edge node {} (14)
    (8) edge node {} (15)
    (9) edge node {} (16)
    (10) edge node {} (17)
    (11) edge node {} (18)
    (12) edge node {} (19)
    (13) edge node {} (20)
    (14) edge node {} (21)
    (15) edge node {} (16)
    (16) edge node {} (17)
    (17) edge node {} (18)
    (18) edge node {} (19)
    (19) edge node {} (20)
    (20) edge node {} (21)
    (15) edge node {} (22)
    (16) edge node {} (23)
    (17) edge node {} (24)
    (18) edge node {} (25)
    (19) edge node {} (26)
    (20) edge node {} (27)
    (21) edge node {} (28)
    (22) edge node {} (23)
    (23) edge node {} (24)
    (24) edge node {} (25)
    (25) edge node {} (26)
    (26) edge node {} (27)
    (27) edge node {} (28)
    (22) edge node {} (29)
    (23) edge node {} (30)
    (24) edge node {} (31)
    (25) edge node {} (32)
    (26) edge node {} (33)
    (27) edge node {} (34)
    (28) edge node {} (35)
    (29) edge node {} (30)
    (30) edge node {} (31)
    (31) edge node {} (32)
    (32) edge node {} (33)
    (33) edge node {} (34)
    (34) edge node {} (35)
    (29) edge node {} (36)
    (30) edge node {} (37)
    (31) edge node {} (38)
    (32) edge node {} (39)
    (33) edge node {} (40)
    (34) edge node {} (41)
    (35) edge node {} (42)
    (36) edge node {} (37)
    (37) edge node {} (38)
    (38) edge node {} (39)
    (39) edge node {} (40)
    (40) edge node {} (41)
    (41) edge node {} (42)
    (36) edge node {} (43)
    (37) edge node {} (44)
    (38) edge node {} (45)
    (39) edge node {} (46)
    (40) edge node {} (47)
    (41) edge node {} (48)
    (42) edge node {} (49)
    (43) edge node {} (44)
    (44) edge node {} (45)
    (45) edge node {} (46)
    (46) edge node {} (47)
    (47) edge node {} (48)
    (48) edge node {} (49);
\end{tikzpicture}
}
\caption{$d_3(G_{7,7}) = 10$}
\end{subfigure}
\caption{}
\label{constructions67}
\end{figure}

Note that the only two-dimensional grid graphs that are not $3$-domatically full are $G_{2,4}$, $G_{2,5}$, $G_{2,6}$, $G_{3,3}$, $G_{3,4}$, $G_{3,5}$, $G_{4,4}$, $G_{4,6}$, and $G_{5,5}$. Moreover, $G_{3,4}$ and $G_{4,4}$ have $3$-domatic numbers that are actually two less than $3$-domatically full, then rest are only one less.

\section{The infinite case}
\label{infsection}

\begin{theorem}
Let $k$ be a positive integer, then \[d_k\left(G_{\infty,\infty}\right) = 2k^2+2k+1.\]  Therefore, the two-dimensional infinite grid graph is domatically-full.
\end{theorem}

\begin{proof}
The number of vertices within distance $k$ of any given vertex of $G_{\infty,\infty}$ is \[1+4k+4\sum_{i=1}^{k-1}i = 2k^2+2k+1.\] Therefore, \[d_k\left(G_{\infty,\infty}\right) \leq 2k^2+2k+1.\] We now give a coloring \[\chi:V \rightarrow \mathbb{Z}_{2k^2+2k+1}\] to show that this is also a lower bound. Let \[\chi((a,b)) = (2k+1)a+b \mod{2k^2+2k+1}.\]

Suppose, towards a contradiction, that some vertex is within distance $k$ away from two different vertices with the same color. This is equivalent to assuming that there exist two distinct vertices $(a_1,b_1)$ and $(a_2,b_2)$ for which \[a_1(2k+1)+b_1 \equiv a_2(2k+1)+b_2 \mod{2k^2+2k+1},\] and \[|a_1-a_2|+|b_1-b_2| \leq 2k.\] We may assume without loss of generality that $a_1 \leq a_2$. Let $A=a_2-a_1$ and $B=|b_1-b_2|$. Then it follows that \[A(2k+1) \pm B \equiv 0 \mod{2k^2+2k+1}.\] Therefore, $A(2k+1) \pm B$ must equal some nonzero multiple of $2k^2+2k+1$. If $A=0$, then \[-2k \leq A(2k+1) \pm B \leq 2k.\] Therefore, we may assume that $A \geq 1$ and that $A(2k+1) \pm B$ must be a positive multiple of $2k^2+2k+1$. However, it is easy to see that in general \[A(2k+1) \pm B \leq 4k^2 +2k < 2(2k^2+2k+1).\] Hence, we need \[A(2k+1) \pm B = 2k^2+2k+1\] exactly.

Given some fixed $1 \leq A \leq 2k$, we get that $-(2k-A) \leq \pm B \leq 2k-A$ and so we need that \[(2k+1)A-(2k-A) \leq 2k^2+2k+1 \leq A(2k+1)+(2k-A).\] Therefore,
\begin{align*}
(2k+2)A - 2k &\leq 2k^2+2k+1\\
A &\leq \left\lfloor \frac{2k^2+4k+1}{2k+2} \right\rfloor \\
A &\leq k,
\end{align*}
using the fact that $A$ must be an integer to introduce the floor function. Similarly,
\begin{align*}
2k^2+2k+1 &\leq 2kA +2k\\
\left\lceil \frac{2k^2+1}{2k} \right\rceil &\leq A\\
k+1 &\leq A,
\end{align*}
a contradiction.
\end{proof}

\section{conclusion}

So far we have been unsuccessful in our attempts to extend some version of the standard block coloring to the general cases when \[k+3 \leq r \leq k+ \left\lceil\frac{k+1}{2}\right\rceil.\] When $k \leq 3$, this does not matter, but when $k \geq 4$, this begins to leave an infinite number of cases where the grid graph might not be $k$-distance domatically-full. We conjecture that this is not the case for any $k$ - that there are only finitely many two-dimensional grid graphs that are not $k$-distance domatically-full for any $k$. We believe we have demonstrated this to ourselves for $k=4$, but have not included the result here since there is no obvious way to generalize it. So we leave the question open.

\bibliography{distancedomatic4}
\bibliographystyle{plain}

\end{document}